\documentclass[12pt, a4]{amsart}

\usepackage{latexsym,amsmath,amssymb,amsthm,mathrsfs,color}
\usepackage{float}

\usepackage[bbgreekl]{mathbbol}
\usepackage{bbm}
\usepackage{tikz-cd}
\usepackage[all]{xy}
\usepackage{stmaryrd}
   \usepackage[pagebackref,colorlinks=true, citecolor=brown,
    filecolor=black,
   linkcolor=blue,
   urlcolor=black]{hyperref}

\usepackage[toc,page]{appendix}

\usepackage[shortalphabetic]{amsrefs}
 \makeatletter
\def\append@label@year@{%
    \safe@set\@tempcnta\bib@year
    \edef\bib@citeyear{\the\@tempcnta}%
    \ifnum\bib@citeyear>9
      \append@to@stem{%
          \ifx\bib@year\@empty
          \else
            \@xp\year@short \bib@citeyear \@nil
          \fi
      }%
    \fi
}
\makeatother

\let\oldtocsection=\tocsection
\renewcommand{\tocsection}[2]{\hspace{0em}\oldtocsection{#1}{#2}}

\usepackage[a4paper]{geometry}

\def\upddots{\mathinner{\mkern 1mu\raise 1pt \hbox{.}\mkern 2mu
\mkern 2mu \raise 4pt\hbox{.}\mkern 1mu \raise 7pt\vbox {\kern 7
pt\hbox{.}}} }

\numberwithin{equation}{section}

\begin{document}
\setlength{\unitlength}{2.5cm}

\newtheorem{thm}{Theorem}[section]
\newtheorem{lm}[thm]{Lemma}
\newtheorem{prop}[thm]{Proposition}
\newtheorem{cor}[thm]{Corollary}
\newtheorem{conj}[thm]{Conjecture}
\newtheorem{specu}[thm]{Speculation}

\theoremstyle{definition}
\newtheorem{dfn}[thm]{Definition}
\newtheorem{eg}[thm]{Example}
\newtheorem{rmk}[thm]{Remark}

\newcommand{\N}{\mathbbm{N}}
\newcommand{\R}{\mathbbm{R}}
\newcommand{\C}{\mathbbm{C}}
\newcommand{\Z}{\mathbbm{Z}}
\newcommand{\Q}{\mathbbm{Q}}

\newcommand{\Mp}{{\rm Mp}}
\newcommand{\Sp}{{\rm Sp}}
\newcommand{\GSp}{{\rm GSp}}
\newcommand{\GL}{{\rm GL}}
\newcommand{\PGL}{{\rm PGL}}
\newcommand{\SL}{{\rm SL}}
\newcommand{\SO}{{\rm SO}}
\newcommand{\Spin}{{\rm Spin}}
\newcommand{\GSpin}{{\rm GSpin}}
\newcommand{\Ind}{{\rm Ind}}
\newcommand{\Res}{{\rm Res}}
\newcommand{\Hom}{{\rm Hom}}
\newcommand{\End}{{\rm End}}
\newcommand{\msc}[1]{\mathscr{#1}}
\newcommand{\mfr}[1]{\mathfrak{#1}}
\newcommand{\mca}[1]{\mathcal{#1}}
\newcommand{\mbf}[1]{{\bf #1}}

\newcommand{\mbm}[1]{\mathbbm{#1}}

\newcommand{\into}{\hookrightarrow}
\newcommand{\onto}{\twoheadrightarrow}

\newcommand{\s}{\mathbf{s}}
\newcommand{\cc}{\mathbf{c}}
\newcommand{\bfa}{\mathbf{a}}
\newcommand{\id}{{\rm id}}
\newcommand{\g}{\mathbf{g}_{\psi^{-1}}}
\newcommand{\w}{\mathbbm{w}}
\newcommand{\Ftn}{{\sf Ftn}}
\newcommand{\p}{\mathbf{p}}
\newcommand{\bq}{\mathbf{q}}
\newcommand{\WD}{\text{WD}}
\newcommand{\W}{\text{W}}
\newcommand{\Wh}{{\rm Wh}}
\newcommand{\ggma}{\omega}
\newcommand{\sct}{\text{\rm sc}}
\newcommand{\Of}{\mca{O}^\digamma}
\newcommand{\gk}{c_{\sf gk}}
\newcommand{\Irr}{ {\rm Irr} }
\newcommand{\Irrg}{ {\rm Irr}_{\rm gen} }
\newcommand{\diag}{{\rm diag}}
\newcommand{\uchi}{ \underline{\chi} }
\newcommand{\Tr}{ {\rm Tr} }
\newcommand{\der}\de
\newcommand{\Stab}{{\rm Stab}}
\newcommand{\Ker}{{\rm Ker}}
\newcommand{\bfp}{\mathbf{p}}
\newcommand{\bfq}{\mathbf{q}}
\newcommand{\KP}{{\rm KP}}
\newcommand{\Sav}{{\rm Sav}}
\newcommand{\de}{{\rm der}}
\newcommand{\tnu}{{\tilde{\nu}}}
\newcommand{\lest}{\leqslant}
\newcommand{\gest}{\geqslant}
\newcommand{\tu}{\widetilde}
\newcommand{\tchi}{\tilde{\chi}}
\newcommand{\tomega}{\tilde{\omega}}
\newcommand{\Rep}{{\rm Rep}}
\newcommand{\A}{{\mbf A}}
\newcommand{\BDI}{{\rm Inv}_{\rm BD}}

\newcommand{\cu}[1]{\textsc{\underline{#1}}}
\newcommand{\set}[1]{\left\{#1\right\}}
\newcommand{\ul}[1]{\underline{#1}}
\newcommand{\wt}[1]{\overline{#1}}
\newcommand{\wtsf}[1]{\wt{\sf #1}}
\newcommand{\anga}[1]{{\left\langle #1 \right\rangle}}
\newcommand{\angb}[2]{{\left\langle #1, #2 \right\rangle}}
\newcommand{\wm}[1]{\wt{\mbf{#1}}}
\newcommand{\elt}[1]{\pmb{\big[} #1\pmb{\big]} }
\newcommand{\ceil}[1]{\left\lceil #1 \right\rceil}
\newcommand{\floor}[1]{\left\lfloor #1 \right\rfloor}
\newcommand{\val}[1]{\left| #1 \right|}

\newcommand{\exc}{ {\rm exc} }

\newcommand{\motimes}{\text{\raisebox{0.25ex}{\scalebox{0.8}{$\bigotimes$}}}}

\makeatletter
\newcommand{\extp}{\@ifnextchar^\@extp{\@extp^{\,}}}
\def\@extp^#1{\mathop{\bigwedge\nolimits^{\!#1}}}
\makeatother

\newcommand{\nequiv}{\not \equiv}
\newcommand{\half}{\frac{1}{2}}
\newcommand{\psii}{\widetilde{\psi}}
\newcommand{\ab} {|\!|}
\newcommand{\mb}{{\widetilde{B(\F)}}}

\title{Quasi-admissible, raisable nilpotent orbits, and theta representations}

\author{Fan Gao, Baiying Liu, and Wan-Yu Tsai}

\address{Fan Gao: School of Mathematical Sciences, Zhejiang University, 866 Yuhangtang Road, Hangzhou, China 310058}
\email{gaofan@zju.edu.cn}
\address{Baiying Liu: Department of Mathematics, Purdue University, 
West Lafayette, IN, 47907, USA}
\email{liu2053@purdue.edu}
\address{Wan-Yu Tsai: Department of Mathematics, National Central University, No. 300, Zhongda Rd., Zhongli District, Taoyuan City 320317, Taiwan}
\email{wytsai@math.ncu.edu.tw}

\maketitle

\begin{abstract} 
We study the  quasi-admissibility and raisablility of some nilpotent orbits of a covering group. In particular, we determine the degree of the cover such that a given split nilpotent orbit is quasi-admissible and non-raisable. The speculated wavefront sets of theta representations are also computed explicitly, and are shown to be quasi-admissible and non-raisable. Lastly, we determine the leading coefficients in the Harish-Chandra character expansion of theta representations of covers of the general linear groups.
\end{abstract}
\tableofcontents

\section{Introduction}
Let $F$ be a number field or any of its local fields. Denote  by $F^{\rm al}$ the algebraic closure of $F$. Let $\mbf{G}$ be  a split connected linear reductive group over $F$. If $F$ is a local field, then we are interested in the group $G_F:=\mbf{G}(F)$ or its finite degree central covers
$$\begin{tikzcd}
\mu_n \ar[r, hook] & \wt{G}_F \ar[r, two heads] & G_F,
\end{tikzcd}$$
where we assume that $\mu_n(F) = \mu_n(F^{\rm al})$. If $F$ is a number field with adele ring ${\mathbb{A}}$, the object of interest is $G_{\mathbb{A}}:=\mbf{G}({\mathbb{A}})$ or its $n$-fold central cover
$$\begin{tikzcd}
\mu_n \ar[r, hook] & \wt{G}_\mathbb{A} \ar[r, two heads] & G_\mathbb{A}.
\end{tikzcd}$$
For example, a family of such central extensions naturally arises from the Brylinski--Deligne framework \cite{BD}. If $n=1$, then the covering groups are just the linear algebraic groups. 
We consider exclusively genuine representations of $\wt{G}_F$ or $\wt{G}_\mathbb{A}$, where $\mu_n$ acts via a fixed embedding $\mu_n \into \C^\times$.
We will study the leading wavefront sets of genuine representations $\pi$ of $\wt{G}_F$ or $\wt{G}_\mathbb{A}$, which are related to the generalized Whittaker models of $\pi$. We first give a brief recapitulation.

Assume that $F$ is a $p$-adic local field. Every irreducible admissible representation $(\pi, V_\pi)$ of $\wt{G}_F$ defines a character distribution $\chi_\pi$ in a neighborhood of $0$ in $\mfr{g}_F={\rm Lie}(G_F)$. Moreover, there exists a compact open subset $S_\pi$ of $0$ such that for every smooth function $f$ with compact support in $S_\pi$, one has (see \cites{How1, HC99, Li3})
\begin{equation} \label{E:char}
\chi_\pi(f) = \sum_{\mca{O} \in \mca{N}}c_\mca{O}(\pi) \cdot \int \hat{f} \ \mu_\mca{O},
\end{equation}
where $\mca{N}$ denotes the set of nilpotent orbits in $\mfr{g}_F$ under the conjugation action of $G_F$. Here $\mu_\mca{O}$ is a certain Haar measure on $\mca{O}$ properly normalized, and $\hat{f}$ is the Fourier transform of $f$ with respect to the Cartan--Killing form on $\mfr{g}_F$ and a non-trivial character 
$$\psi: F\to \C^\times;$$
one has $c_\mca{O}(\pi):=c_{\mca{O}, \psi}(\pi) \in \C$. An analogue of \eqref{E:char} for $F=\R$ was given by Barbasch--Vogan \cite{BV3}. Denote
$$\mca{N}_{\rm tr}(\pi) = \set{\mca{O} \in \mca{N}: \ c_\mca{O}(\pi) \ne 0}$$
and let
$$\mca{N}_{\rm tr}^{\rm max}(\pi) \subset \mca{N}_{\rm tr}(\pi)$$
be the subset consisting of all maximal elements in $\mca{N}_{\rm tr}(\pi)$.

The set  $\mca{N}_{\rm tr}^{\rm max}(\pi)$  gives the Gelfand--Kirillov dimension
$$d_{\rm GK}(\pi):=\frac{1}{2} \max \set{ \dim \mca{O}: \ \mca{O} \in \mca{N}_{\rm tr}^{\rm max}(\pi)}$$
which satisfies
$$\dim \pi^K \approx {\rm vol}(K)^{-d_{\rm GK}(\pi)},$$
as the open compact subgroup $K \subset G$ approaches to the identity, see \cite{Sav94}. It is also known (see \cites{MW1, Var0, Pate}) that the set $\mca{N}_{\rm tr}^{\rm max}(\pi)$ is equal to the set of maximal nilpotent orbits with respect to which the generalized Whittaker models for $\pi$ are nontrivial. 

More precisely, let $(f, h, u) \subset \mfr{g}_F$ be  an $\mfr{sl}_2$-triple. In this case, $h$ is called a neutral element and it gives a filtration
$$\mfr{g}_F = \bigoplus_{i \in \Z} \mfr{g}^h[i],$$
and $u \in \mfr{g}^h[-2]$.
Write
$$\mfr{n}_{h, u}:= \bigoplus_{i \geq 2} \mfr{g}^h[i] \subset \mfr{g}_F$$
and let $N_{h, u}:=\exp(\mfr{n}_{h,u}) \subset G_F$ be the unipotent subgroup of $G_F$. There is a character $\psi_u: N_{h, u} \to \C^\times$ given by
$$\psi_u(n)=\psi(\kappa(u, \log(n))),$$
where $\kappa: \mfr{g}_F \times \mfr{g}_F \to F$ is the Killing form. For every $\pi\in \Irrg(\wt{G}_F)$, the twisted Jacquet module is given by 
$$\pi_{N_{h, u}, \psi_u}:=\Hom_{N_{h, u}}(\pi, \psi_u).$$
The above depends only on the $G_F$-orbit $\mca{O} \subset \mfr{g}_F$ of $u$, and thus we may write
$$\pi_{N_\mca{O}, \psi_\mca{O}} = \pi_{N_{h, u}, \psi_u}$$
for any choice of $u\in \mca{O}$. This gives 
$$\mca{N}_{\rm Wh}(\pi)=\set{\mca{O} \subset \mfr{g}_F: \ \pi_{N_\mca{O}, \psi_\mca{O}} \ne 0}$$
and we let $\mca{N}_{\rm Wh}^{\rm max}(\pi) \subset \mca{N}_{\rm Wh}(\pi)$ be the subset of maximal elements. Then it was shown in \cites{MW1, Var0, Pate} that
\begin{enumerate}
\item[--] $\mca{N}_{\rm tr}^{\rm max}(\pi) = \mca{N}_\Wh^{\rm max}(\pi)$ for every $\pi \in \Irrg(\wt{G}_F)$, and
\item[--] for every $\mca{O}$ in the above (equal) sets, one has $c_\mca{O} = \dim_\C \pi_{N_\mca{O}, \psi_\mca{O}}$.
\end{enumerate}
We call $\mca{N}_{\rm tr}^{\rm max}(\pi)$ and thus also $\mca{N}_\Wh^{\rm max}(\pi)$ the wavefront set of $\pi$. For $\pi \in \Irr(\GL_r)$, a relation between the full set $\set{c_{\mca{O}}: \mca{O} \in \mca{N}_{\rm tr}(\pi)}$ and certain degenerate Whittaker models of $\pi$ is given in the recent work of M. Gurevich \cite{Gur23}.

If $F$ is a number field, and $\mca{O} \subset \mfr{g}_F$ a nilpotent orbit, then for any genuine automorphic representation $\pi=\otimes_v \pi_v$ of $\wt{G}_\mathbb{A}$, one can define the Whittaker--Fourier $\psi_\mca{O}$-coefficients of $\pi$. Moreover, one has the analogous $\mca{N}_\Wh(\pi)$ and $\mca{N}_\Wh^{\rm max}(\pi)$ of nilpotent orbits in $\mfr{g}_F$ associated with $\pi$.  These two sets are intimately related to $\mca{N}_\Wh(\pi_v)$ and $\mca{N}_\Wh^{\rm max}(\pi_v)$. For an  either local or global representation $\pi$, it is important to understand the set $\mca{N}_\Wh^{\rm max}(\pi)$, as shown above. It also has deep applications, for example in the theory of descent and the Gan--Gross--Prasad conjectures and others, see \cites{GRS11, FG18, JiZh20} and references therein.

Nevertheless, determining the set $\mca{N}_\Wh^{\rm max}(\pi)$ is a difficult problem for general $\pi$. For $\GL_r$, Ginzburg \cite{Gin0} and Dihua Jiang \cite{Jia14} gave several speculations for the set $\mca{N}_{\rm Wh}^{\rm max}(\pi)$ in terms of the M{\oe}glin--Waldspurger classification of the spectrum of $\GL_r$ (see \cite{LX21} for some recent progress). For general $G$, the Arthur parametrization of certain unitary representations $\pi$ is expected to encode information on $\mca{N}_\Wh^{\rm max}(\pi)$ in terms of their Arthur parameters. This in part is tied closely with the local archimedean case as studied by Joseph, Barbasch, Vogan, Adams and many others. The development in the archimedean case was more satisfactory, as for example for complex Lie groups, the set $\mca{N}_{\rm tr}^{\rm max}(\pi)$ is well-understood from Barbasch--Vogan \cites{BV4, BV5}. For linear $G$, as alluded to above, Jiang \cite{Jia14} formulated several conjectures regarding $\mca{N}_\Wh^{\rm max}(\pi)$ using the parameter of $\pi$. There has been progress towards it as in \cites{JiLi16, JiZh17, JL22, LS22a, LS22b, JLZ}, among many others. We also have the recent work by Ciubotaru, Mason-Brown and Okada towards understanding the stable wavefront sets and their refined versions, see \cites{Oka, CMO}.

For covering groups, again, advance has been made more for the archimedean case, as can be seen from \cites{ABPTV, Tsa19} and references therein. Some of these results rely on the theory of Harish-Chandra modules and their primitive ideals. The ``easiest" genuine representation of a covering group $\wt{G}$ is the theta representation $\Theta(\wt{G})$, a prototype of which is called the even Weil representation of the double cover $\wt{\Sp}_{2r}^{(2)}$. It is desirable to have a full description of $\mca{N}_\Wh^{\rm max}(\Theta(\wt{G}))$.  It was proved by Y.-Q. Cai and Savin that
\begin{equation} \label{WFGL}
\mca{N}_\Wh^{\rm max}(\Theta(\wt{\GL}_r^{(n)})) = \{(n^a b)\},
\end{equation}
where $r=an + b$ with $0\lest b < n$. For $n\gest r$, this recovers the earlier result of Kazhdan--Patterson on generic theta representations. Formulas analogous to \eqref{WFGL} were proved for $\wt{\SO}_{2r+1}^{(4)}$ and $\wt{\GSpin}_{2r}^{(2)}$, see \cites{BFrG, BFG} and \cite{Kap17-1} respectively. The case of $\wt{\Sp}_{2r}^{(n)}$ was also studied extensively by Friedberg and Ginzburg \cites{FG15, FG17-2, FG18, FG20}.
In \cite{GaTs}, a uniform but speculative formula of $\mca{N}_\Wh^{\rm max}(\Theta(\wt{G}))$ was given for general $\wt{G}$, which was motivated from the works mentioned above.

It has been expected for a long time that over the algebraic closure the wavefront set is a singleton for general $\pi$. On the contrary, C.-C. Tsai gave a counter-example recently in \cite{Tsa22} for certain epipelagic supercuspidal representations of $U_7$. However, we still expect,
as from \cites{Oka, CMO} for linear groups, 
that for the genuine Iwahori-spherical representations, in particular unramified genuine theta representations, the wavefront set is a singleton over $F^{\rm al}$.

\subsection{Main results}

Although it is difficult to determine precisely $\mca{N}_\Wh^{\rm max}(\pi)$ for general $\pi$, some work has been done in the literature to determine whether a candidate $\mca{O}$ could possibly lie in $\mca{N}_\Wh^{\rm max}(\pi)$ for some $\pi$. 
Pertaining to this are the notions of admissible or quasi-admissible orbits. Equivalently, one has some sufficient condition for an orbit to lie in the complement set
\begin{equation} \label{WFc}
\mca{N} - \bigcup_{\pi \in \Irrg(\wt{G})} {\rm WF}(\pi),
\end{equation}
especially when $\wt{G} =G$ is a linear group.

For real Lie groups, the notion of admissibility was first proposed by Duflo \cite{Duf80}. 
This notion was studied both in the $p$-adic and the real setting in many works later, for example \cites{Tor, Nev99, Nev02}, and we refer the reader to loc. cit. for more extensive references and historical discussion. Utilizing a more general ``Whittaker pair", the recent work by Gomez, Gourevitch and Sahi in \cites{GGS17, GGS21} enables one to consider degenerate Whittaker models, where the discussion on admissibility and quasi-admissibility of nilpotent orbits was framed generally and applies easily to covering groups as well.

Closely related to admissible (or slightly weaker, quasi-admissible) orbits is the notion of special orbits, which correspond to the special Weyl group representations via the Springer correspondence. It was shown by M{\oe}glin \cite{Moe96} that for classical groups, every admissible orbit has to be special. Thus, a leading wavefront orbit for classical linear algebraic groups is necessarily a special orbit. This fact fails for covering groups, as seen already from the three-fold cover of $G_2$ and many covering groups considered in this paper. In fact, in a different direction, Jiang, Savin and the second-named author  considered in \cite{JLS} the notion of raisability of nilpotent orbits which equally applies to covering groups as well. Roughly speaking, if a nilpotent orbit is raisable, then it must lie in the set \eqref{WFc}.

Our present paper is motivated from the works above. It could be in part considered as an application and explication of \cites{GGS21, JLS} in the covering setting, and in part as a sequel to \cite{GaTs}. Below we give an elaboration of this and also state our main results.

First, by applying and analyzing the techniques from \cites{GGS17, GGS21} and \cite{JLS}, we determine the orbits $\mca{O} \subset \mfr{g}_F$ which are $\wt{G}_F^{(n)}$-quasi-admissible or $\wt{G}_F^{(n)}$-raisable. As alluded to above, the importance of these two notions is the following: 
\begin{enumerate}
\item[--] if $\mca{O}$ is $\wt{G}^{(n)}$-raisable or not $\wt{G}^{(n)}$-quasi-admissible, then it does not lie in $\mca{N}^{\rm max}_{\rm Wh}(\pi)$ for any $\pi\in \Irrg(\wt{G}^{(n)})$.
\end{enumerate}
For technical reasons, we only consider $F$-split nilpotent orbits. In \S \ref{S:q-r}, we consider general $\wt{G}$ and give equivalent criteria for quasi-admissibility which are amenable to explicit verification, similarly for raisability; see Propositions \ref{P:qadm}, \ref{P:exp-spl} and Proposition \ref{P:rais}.

In \S \ref{S:ex-ana}, applying results from \S \ref{S:q-r}, we analyze groups of each Cartan type and determine the quasi-admissibility and raisability of each $F$-split orbit $\mca{O}$ of $G$, whenever possible. 

\begin{thm} \label{T:main1}
Let $\wt{G}^{(n)}$ be the $n$-fold cover of any of $\GL_r, \SO_{2r+1}, \SO_{2r}, \Sp_{2r}$. Conditions for quasi-admissibility and raisability of an $F$-split orbit $\mca{O}=(p_1^{d_1} \cdots p_i^{d_i} \cdots p_k^{d_k}) \in \mca{N}$ are stipulated in Theorems \ref{T:typeA}, \ref{T:tyBD}, \ref{T:tyC}. For covering groups of the simply-connected exceptional group $G_2, F_4, E_r, 6\lest r \lest 8$, quasi-admissible and raisable orbits are given in Tables \ref{table 1}-\ref{table 5}, whenever our method applies.
\end{thm}

For exceptional groups, we check the quasi-admissibility and raisability for all orbits of $G_2$ and $F_4$, but only for those orbits which are speculatively the leading wavefront set of theta representations for $E_{r}, 6\lest r \lest 8$. However, the method of the computation clearly applies to any orbit of $E_r$.

Next, we consider in \S \ref{S:wf}  theta representations $\Theta(\nu)$ of $\wt{G}^{(n)}$, where $\nu \in X \otimes \R$ is a certain exceptional character. Naturally associated to $\nu$ is the nilpotent orbit $\mca{O}_{\rm Spr}(j_{W_\tnu}^W \varepsilon_\tnu) \subset \mfr{g}_{F^{\rm al}}$, which arises from the Springer correspondence and the $j$-induction of the sign character $\varepsilon_{\tnu}$ of $W_{\tnu} \subset W$. This orbit is expected to be $\mca{N}_{\rm Wh}^{\rm max}(\Theta(\nu))\otimes F^{\rm al}$.
In \S \ref{SS:som-com}, we show that the computation of $\mca{O}_{\rm Spr}(j_{W_\tnu}^W \varepsilon_\tnu) \subset \mfr{g}_{F^{\rm al}}$ is reduced to Sommers' duality \cite{Som01} between nilpotent orbits, which generalizes the classical Barbasch--Vogan duality. With an explicit computation, we determine the quasi-admissibility and non-raisability of such orbits.

\begin{thm}[{Theorem \ref{T:comp}}] \label{T:main2}
Let $\wt{G}^{(n)}$ be covers of the classical and exceptional groups considered in Theorem \ref{T:main1}.
\begin{enumerate}
\item[(i)] The orbit $\mca{O}_{\rm Spr}(j_{W_\tnu}^W \varepsilon_\tnu) \subset \mfr{g}_{F^{\rm al}}$ for $\wt{G}^{(n)}$ of classical groups is explicitly given as in Table \ref{table 6}; for $\wt{G}^{(n)}$ of exceptional groups, it is given in Tables \ref{table 7}--\ref{table 11}.
\item[(ii)] The $F$-split orbit $\mca{O}_\Theta \subset \mfr{g}_F$ of type  $\mca{O}_{\rm Spr}(j_{W_\tnu}^W \varepsilon_\tnu)$ is quasi-admissible and not raisable.
\item[(iii)] If the orbit $\mca{O}_\Theta$ is the regular orbit of a Levi subgroup of $G$, then it supports the generalized Whittaker model of  the theta representation $\Theta(\nu)$.
\end{enumerate}
\end{thm}

Last, in \S \ref{S:cO} we determine $c_\mca{O}$ for $\mca{O}$ in $\mca{N}_\Wh^{\rm max}(\Theta(\wt{\GL}_r^{(n)}))$, which was already shown by Savin and Y.-Q. Cai to be equal to 
$$\mca{O}_{\rm Spr}(j_{W_\tnu}^W \varepsilon_\tnu) = (n^a b).$$
The main result in \S \ref{S:cO} is then
\begin{thm}[{Theorem \ref{T:cO}}] \label{T:main3}
Consider a theta representation $\Theta(\wt{\GL}_r^{(n)})$ of the Kazhdan--Patterson cover $\wt{\GL}_r^{(n)}$. Then for the unique orbit $\mca{O}=(n^a b)$ in $\mca{N}_\Wh^{\rm max}(\Theta(\wt{\GL}_r^{(n)}))$, one has
$$c_\mca{O} =\angb{j_{W_\nu}^W (\varepsilon_\nu)}{ \varepsilon_W \otimes \sigma^{\msc{X}} }_W.$$
Thus, Conjecture \ref{C:main} holds for $\wt{G}=\wt{\GL}_r^{(n)}$.
\end{thm}
Here $\sigma^\msc{X}: W \to {\rm Perm}(\msc{X})$ is the permutation representation of $W$ on $\msc{X}$ via the twisted Weyl action, see \eqref{F:sigmaX}. The proof of Theorem \ref{T:main3} relies on the crucial fact that every nilpotent orbit of $\GL_r$ is of PL-type \`a la \cite{GoSa15}, i.e., it is the principal/regular orbit of a Levi subgroup. In the proof of Theorem \ref{T:main3}, we also use the result of Gomez--Gourevitch--Sahi \cite{GGS21} and some properties of the $j$-induction.

For general $\wt{G}^{(n)}$ we expect that whenever the orbit $\mca{O}_{\rm Spr}(j_{W_\tnu}^W \varepsilon_\tnu)$ is of PL-type, then \cite[(2.4)]{GaTs} (i.e, \eqref{E:main2} in Conjecture \ref{C:main} here)  can be checked by using similar analysis as in the case of $\wt{\GL}_r^{(n)}$. However, it seems to require new ideas to deal with the  case of non-PL-type orbits.

\subsection{Acknowledgement} We would like to thank Dima Gourevitch for several communications and some clarifications on \cite{GGS21}. The work of F. G. is partially supported by the National Key R{\&}D Program of China (No. 2022YFA1005300) and also by NNSFC-12171422. The work of B.\! L. is partially supported by the NSF Grants DMS-1702218, DMS-1848058. The work of W.-Y. Tsai is supported by the National Science and Technology Council of Taiwan  (108-2115-M-033-004-MY3 and 111-2115-M-033-001-MY3).

\section{Quasi-admissible and raisable orbits} \label{S:q-r}

\subsection{Covering groups}
Let $F$ be a number field or a local field. Let $\mbf{G}$ be a connected split linear reductive group over $F$. Denote by 
$$(X,\ \Phi, \ \Delta; \ Y, \Phi^\vee, \Delta^\vee)$$
 the root datum of $\mbf{G}$, where $X$ is the character lattice and $Y$ the cocharacter lattice of a split torus $\mbf{T}\subset \mbf{G}$. Here $\Delta$ is a choice of simple roots, and we denote by $Y^{sc} \subset Y$ the coroot lattice and $X^{sc} \subset X$ the root lattice. Denote by $W$ the Weyl group of the coroot system.

Let $$Q: Y \longrightarrow \Z$$ be a Weyl-invariant quadratic form, and let $B_Q$ be the associated bilinear form.
Assume that $F^\times$ contains the full group $\mu_n$ of $n$-th roots of unity.  Consider the pair $(D, \eta=\mbf{1})$, where $D$ is a ``bisector" of $Q$ (see \cite[\S 2.6]{GG}). If $F$ is a local field with $G_F:=\mbf{G}(F)$, then associated to $(D, \mbf{1})$ one has a covering group  $\wt{G}_F$, which is a central extension
$$\begin{tikzcd}
\mu_n \ar[r, hook] & \wt{G}_F \ar[r, two heads] & G_F
\end{tikzcd}$$
of $G$ by $\mu_n$. For simplicity, we may also write $\wt{G}:=\wt{G}_F$. If $F$ is a number field, then one has a global $n$-fold central cover $\wt{G}_{\mathbb{A}}$ of $G_{\mathbb{A}}$. For more details, see \cites{BD, GG, We6}.  Throughout, for every root $\alpha$ we denote
\begin{equation}\label{n_alpha}
    n_\alpha := \frac{n}{\gcd(n, Q(\alpha^\vee))}.
\end{equation}

\begin{dfn} \label{D:BDI}
Suppose $G_{der}$ is almost simple. Then the number
$$\BDI(\wt{G}):=Q(\alpha^\vee) \in \Z,$$
where $\alpha^\vee$ is any short coroot, is called the Brylinski--Deligne invariant associated with $\wt{G}$.
\end{dfn}

Note that the Brylinski--Deligne invariant does not depend on $n$. Moreover, it behaves functorially as follows. Let $\zeta: \mbf{G} \longrightarrow \mbf{H}$ be an algebraic group homomorphism. It induces a group homomorphism $\zeta^\natural: Y_G \longrightarrow Y_H$ on the cocharacter lattices of $\mbf{G}$ and $\mbf{H}$. Let $\wt{H}$ be an $n$-fold cover of $H$ associated with a quadratic form $Q_H$. The pull-back $n$-fold cover $\zeta^*(\wt{H})$ of $G$ via $\zeta$ is associated with $Q\circ \zeta^\natural$. In particular, we have
$$\BDI(\zeta^*(\wt{H})) = Q \circ \zeta^\natural(\alpha^\vee_G),$$
where $\alpha^\vee_G$ is any short coroot of $G$.

\subsection{Nilpotent orbits and Whittaker pairs}
Let $F$ be a number field or a local field.
For the terminology and notations in this subsection, we follow \cites{GGS17, GGS21}.
 In particular, the Lie algebra $\mfr{sl}_2$ over $F$ has a standard basis $\set{e_+, h_0, e_-}$. Recall that $\mfr{g}_F$ denotes the Lie algebra of $G_F$.

Let $u\in \mfr{g}_F$ be a nilpotent element. Given a semisimple element $h \in \mfr{g}_F$, one has a decomposition 
$$\mfr{g}_{F} = \bigoplus_{i \in I} \mfr{g}^h[i],$$
where $I \subset F$ and  $\mfr{g}^h[i]$ denotes the $i$-th eigenspace of the adjoint action of $h$ on $\mfr{g}_F$. The element $h\in \mfr{g}_F$ is called $\Q$-semisimple if 
$$I \subset \Q.$$
The pair $(h, u)$ is called a Whittaker pair if $h$ is $\Q$-semisimple and $u\in \mfr{g}^h[-2]$. A Whittaker pair is called a neutral pair if there exists a nilpotent element $f\in \mfr{g}_F$ such that 
$$(f, h, u)$$
is an $\mfr{sl}_2$-triple, in which case $h$ is called a neutral element for $u$. For any nilpotent $u$, the Jacobson--Morozov theorem gives a homomorphism 
$$\gamma: \mfr{sl}_2 \longrightarrow \mfr{g}_F$$
such that $u=\gamma(e_-)$. Conversely, naturally associated to any $\gamma$ is an $\mfr{sl}_2$-triple
$$\set{f, h, u} \subset \mfr{g}_F,$$
where $h$ is then a neutral element for $u$.

Let 
$$\kappa: \mfr{g}_F \times \mfr{g}_F \longrightarrow F $$
be the Killing form.  For any Whittaker pair $(h, u)$, one has a symplectic form
$$\omega_u: \mfr{g}_F \times \mfr{g}_F \longrightarrow F$$
given by $\omega_u(x, y):= \kappa(u, [x, y])$. For any rational $i\in \Q$, let
$$\mfr{g}_{\gest i}^h = \bigoplus_{i' \gest i} \mfr{g}^h[i'] \text{ and } \mfr{u}_h:=\mfr{g}_{\gest 1}^h.$$
The restriction $\omega_u|_{\mfr{u}_h}$ is well-defined and let $\mfr{n}_{h, u}$ be the radical of $\omega_u|_{\mfr{u}_h}$. Then
$$[\mfr{u}_h, \mfr{u}_h] \subset \mfr{g}_{\gest 2}^h \subset \mfr{n}_{h, u}.$$
By \cite[Lemma 3.2.6]{GGS17}, one has
$$\mfr{n}_{h, u} = \mfr{g}_{\gest 2}^h + \mfr{g}_1^h \cap \mfr{g}_u,$$
where $\mfr{g}_u$ denotes the centralizer of $u\in \mfr{g}$. If the Whittaker pair $(h, u)$ is a neutral pair, then $\mfr{n}_{h,u}=\mfr{g}^h_{\gest 2}$. 

Let $U_h=\exp(\mfr{u}_h)$ and $N_{h,u}=\exp(\mfr{n}_{h,u})$ be the corresponding unipotent subgroups of $G_F$. 
Let $$\psi: F \longrightarrow \C^\times$$
be a nontrivial character. Define a character of $N_{h,u}$ by
\begin{equation} \label{F:psi-u}
\psi_u(n)=\psi(\kappa(u,\log(n))).
\end{equation}
 Let $N_{h,u}' = N_{h,u} \cap {\rm Ker}(\psi_u)$. Then $U_h/N_{h,u}'$ is a Heisenberg group with center $N_{h,u}/N_{h,u}'$.
Every unipotent subgroup splits canonically into $\wt{G}$. Thus, we view all the above $U_h, N_{h, u}$ etc as subgroups of $\wt{G}$.

If $F$ is a local field and $\pi$ an irreducible genuine representation of $\wt{G}$, then the degenerate Whittaker model of $\pi$ associated to any Whittaker pair $(h, u)$ is
$$\pi_{N_{h,u}, \psi_u}:=\Hom_{N_{h, u}}(\pi, \psi_u),$$
the twisted Jacquet module of $\pi$ with respect to $(N_{h, u}, \psi_u)$. If $(h, u)$ is a neutral pair, then $\pi_{N_{h,u}, \psi_u}$ is also called a generalized Whittaker model of $\pi$. For any nilpotent orbit $\mca{O} \subset \mfr{g}_F$, we write
$$\pi_{N_\mca{O}, \psi_\mca{O}} = \pi_{N_{h,u}, \psi_u}$$
for any $u\in \mca{O}$ and any choice $(h, u)$ of neutral pair. Denote
$$\mca{N}_\Wh(\pi)=\set{\mca{O} \subset \mfr{g}_F: \ \pi_{N_\mca{O}, \psi_\mca{O}} \ne 0}$$
and let 
$$\mca{N}_\Wh^{\rm max}(\pi) \subset \mca{N}_\Wh(\pi)$$
 be the subset consisting of maximal elements.

Now assume $F$ is a number field and $u\in \mfr{g}_F$ . Let $\psi: F\backslash \mathbb{A} \to \C^\times$ be a nontrivial additive character. One can extend the Killing form to obtain $\kappa: \mfr{g}_{\mathbb{A}} \times \mfr{g}_{\mathbb{A}} \to {\mathbb{A}}_F$. This gives a character
$$\psi_u: N_{h, u}(\mathbb{A}) \to \C^\times$$
given by the same formula as in \eqref{F:psi-u}. Here $\psi_u$ is automorphic, i.e., trivial on $N_{h, u}(F)$ and hence can be viewed as a character on
$$[N_{h,u}]:=N_{h,u}(F)\backslash N_{h, u}(\mathbb{A}).$$
Let $(\pi, V_\pi)$ be an irreducible genuine automorphic representation of $\wt{G}_\mathbb{A}$. For any $\phi \in V_\pi$, the degenerate Whittaker--Fourier coefficient of $\phi$ attached to a Whittaker pair $(h, u)$ is given by
\begin{equation} \label{F:F-co}
\mca{F}_{h,u}(\phi)(g):=\int_{[N_{h,u}]} \phi(ng) \psi_u^{-1}(n) dn.
\end{equation}
For any Whittaker pair $(h, u)$, we denote
$$\mca{F}_{h,u}(\pi)=\set{ \mca{F}_{h,u}(\phi): \ \phi \in V_\pi }.$$
If $(h, u)$ is a neutral pair, then again $\mca{F}_{s,u}(\phi)$ is called a generalized Whittaker--Fourier coefficient of $\phi$.
For any orbit $\mca{O} \subset \mfr{g}_F$, we write
$$\mca{F}_\mca{O}(\phi)= \mca{F}_{h,u}(\phi)$$
for any neutral pair $(h, u)$ with $u\in \mca{O}$. This gives $\mca{F}_\mca{O}(\pi) = \set{ \mca{F}_\mca{O}(\phi): \phi\in V_\pi }$ and
$$\mca{N}_\Wh(\pi)=\set{ \mca{O}\subset \mfr{g}_F: \ \mca{F}_\mca{O}(\pi)\ne 0}.$$
Similarly, we have the subset of maximal elements $\mca{N}_\Wh^{\rm max}(\pi) \subset \mca{N}_\Wh(\pi)$, for which there are many results, see for example \cites{Gin0, JLS}. Since the set $\mca{N}_{\rm Wh}^{\rm max}(\pi)$ for $\pi=\otimes_v \pi_v$ is partially constrained by $\mca{N}_{\rm Wh}^{\rm max}(\pi_v)$ for any place $v$ of $F$, we focus in this paper only on $p$-adic local places. Henceforth, we assume that $F$ is a $p$-adic local field.

\subsection{$F$-split nilpotent orbits}
Let $\gamma=(f, h, u) \subset \mfr{g}_F$ be an $\mfr{sl}_2$-triple. Let $\mca{O}_u \subset \mfr{g}_F$. Viewing $\gamma \subset \mfr{g}_{F^{\rm al}}$, one has
$$\mbf{G}_u = \mbf{G}_\gamma \cdot \mbf{N}_u,$$
where $\mbf{G}_\gamma \subset \mbf{G}$ is the stabilizer subgroup of $\gamma$. The group $\mbf{G}_\gamma$ is a (possibly disconnected) linear algebraic group and is the reductive part of $\mbf{G}_u$. In general, we have 
$$(\mca{O}_u \otimes F^{\rm al}) \cap \mfr{g}_F = \bigsqcup_{i \in I} \mca{O}_{u_i},$$
where the $F$-rational orbits $\mca{O}_{u_i}$ are classified by the Galois cohomology group
$$H^1(F, \mbf{G}_\gamma),$$
see \cite[Chap. III]{Ser-gc} and \cite[\S 4]{Nev02}. This group is only a pointed set if $\mbf{G}_\gamma$ is not abelian.

Consider the connected component group
$$\pi_0(\mbf{G}_\gamma) = \mbf{G}_\gamma/\mbf{G}_{\gamma, 0}$$
of $\mbf{G}_\gamma$. This group is one of the symmetric groups $S_i, 1\lest i \lest 5$; moreover, for classical groups it is either trivial or $S_{2}$.
One has an exact sequence of pointed sets
$$\begin{tikzcd}
H^1(F, \mbf{G}_{\gamma, 0}) \ar[r] & H^1(F, \mbf{G}_\gamma) \ar[r, "\iota"] & H^1(F, \pi_0(\mbf{G}_\gamma)) \ar[r] & \cdots.
\end{tikzcd}$$
In particular, if $\mbf{G}_{\gamma, 0}$ is semisimple and simply-connected, then $\iota$ is injective.

The orbit $\mca{O}_u$ is called an $F$-split nilpotent orbit if the reductive group $\mbf{G}_\gamma$ is split over $F$. In this paper, we only consider $F$-split orbits, since the computation with covers of a non-split $\mbf{G}_\gamma$  involves further subtlety. Henceforth, we will assume that the orbit $\mca{O} \subset \mfr{g}_F$ is always $F$-split without explicating this again.

In this paper, we will implement the methods of M\oe glin \cite{Moe96}, Nevins \cites{Nev99, Nev02}, Jiang--Liu--Savin \cite{JLS}  and Gomez--Gourevitch--Sahi \cite{GGS21} to give some necessary or sufficient condition for an orbit $\mca{O}_u$ to possibly lie in the wavefront set of some genuine representation $\pi \in \Irrg(\wt{G})$. In fact, for general $\wt{G}$ we determine (if possible) the $F$-split orbits of $\mfr{g}_F$ that can not be the wavefront set of any genuine representation $\pi$. More precisely, we consider necessary conditions for an $F$-split orbit $\mca{O}_u$ to be quasi-admissible in the sense of \cite{GGS21}, and also give sufficient conditions for it to be raisable, a notion due to \cite{JLS}. All these together will enable us to determine the subset of the orbits in $\mca{N}$ which never occur in the wavefront set of any $\pi \in \Irrg(\wt{G})$.

Some results here also hold in the global setting for genuine automorphic representations of $\wt{G}_{\mathbb{A}}$.

\subsection{$\wt{G}^{(n)}$-quasi-admissible orbits}
Recall that by restriction, one has a non-degenerate symplectic form
$$\omega_u: \mfr{g}[1] \times \mfr{g}[1] \longrightarrow F.$$
The group $G_\gamma = \mbf{G}_\gamma(F)$ acts on $\mfr{g}[1]$ and preserves the form $\omega_u$. Thus, one has a natural group homomorphism
$$\phi: G_\gamma \longrightarrow \Sp(\mfr{g}[1]).$$
By pull-back of the metaplectic double cover $\Mp(\mfr{g}[1])$ of $\Sp(\mfr{g}[1])$ via $\phi$, one has a double cover 
$$\wt{G}_\gamma^{(2), \phi} \onto G_\gamma.$$
Here the metaplectic group $\Mp(\mfr{g}[1])$ is uniquely determined by its Brylinski--Deligne invariant
$$\BDI(\Mp(\mfr{g}[1]))=1.$$

The inclusion $G_\gamma \subset G$ gives an inherited covering $\wt{G}_\gamma^{(n)}$. We obtain the fiber product $\wt{G}_\gamma^{(n)} \times_{G_\gamma} \wt{G}_\gamma^{(2), \phi} $ as in the following diagram
$$\begin{tikzcd}
\wt{G}_\gamma^{(n)} \times_{G_\gamma} \wt{G}_\gamma^{(2), \phi} \ar[r]  \ar[d, "q_\gamma"] & \wt{G}_\gamma^{(2), \phi} \ar[r]  \ar[d, two heads] & \Mp(\mfr{g}[1])  \ar[d, two heads] \\
\wt{G}_\gamma^{(n)} \ar[r, two heads] \ar[d, hook] & G_\gamma \ar[r, "\phi"] \ar[d, hook] & \Sp(\mfr{g}[1]) \\
\wt{G}^{(n)} \ar[r, two heads] & G.
\end{tikzcd}$$

Write 
$$G_{\gamma, der} \subset G_\gamma$$
 for the subgroup generated by unipotent elements; it is equal to the derived subgroup of $G_{\gamma, 0}:=\mbf{G}_{\gamma, 0}(F) \subset G_\gamma$. Thus, we have the inclusions
$$G_{\gamma, der} \subset G_{\gamma, 0} \subset G_\gamma.$$
Let $\mbf{G}_{\gamma, der} \subset \mbf{G}_{\gamma}$ be the derived subgroup and let $\mbf{f}: \mbf{G}_{\gamma, sc} \onto \mbf{G}_{\gamma, der}$ be the simply-connected cover. Setting $G_{\gamma, sc}:=\mbf{G}_{\gamma, sc}(F)$, the map $\mbf{f}$ induces a map 
\begin{equation} \label{F:f}
f: G_{\gamma, sc} \onto G_{\gamma, der} \into \mbf{G}_{\gamma, der}(F),
\end{equation}
see \cite[\S 6]{BoTi73}. In particular, the inclusion in \eqref{F:f} may not be an equality in general.

Henceforth, we write $\wt{G}_{\gamma}^{(n, 2)}:= \wt{G}_{\gamma}^{(n)} \times_{G_{\gamma}} \wt{G}_{\gamma}^{(2), \phi}$. Similarly, for $\star \in \set{0, der, sc}$, one has the natural pull-back $\wt{G}_{\gamma, \star}^{(n)}$ and $\wt{G}_{\gamma, \star}^{(2), \phi}$, and we set
$$\wt{G}_{\gamma, \star}^{(n, 2)}:= \wt{G}_{\gamma, \star}^{(n)} \times_{G_{\gamma, \star}} \wt{G}_{\gamma, \star}^{(2), \phi}.$$
One has by construction the following diagram
\begin{equation} \label{CD:01}
\begin{tikzcd}
\wt{G}_{\gamma, \star}^{(n, 2)} \ar[r]  \ar[d] \ar[rd, two heads, "{p_{\gamma, \star}}"] & \wt{G}_{\gamma, \star}^{(2), \phi}   \ar[d, two heads] \\
\wt{G}_{\gamma,\star}^{(n)} \ar[r, two heads]  & G_{\gamma,\star},
\end{tikzcd}
\end{equation}
where $p_{\gamma, \star}$ is the canonical quotient map.

\begin{dfn} 
A representation of a group $H$ with $\mu_n \times \mu_2 \subset Z(H)$ is called  $(n, 2)$-genuine if the central subgroups $\mu_n $ and $\mu_2$ both act faithfully. An $F$-split nilpotent orbit $\mca{O}  = \mca{O}_u \subset \mfr{g}$ is called 
\begin{enumerate}
\item[--] $\wt{G}^{(n)}$-admissible if the map $p_{\gamma, 0}: \wt{G}_{\gamma, 0}^{(n, 2)} \onto G_{\gamma, 0}$ in \eqref{CD:01} splits;
\item[--] $\wt{G}^{(n)}$-quasi-admissible if $\wt{G}_\gamma^{(n, 2)}$ admits a finite dimensional $(n, 2)$-genuine representation.
%
\end{enumerate}
\end{dfn}

For simplicity, we may just use admissibility and quasi-admissibility when the underlying covering group is clear from the context. It is clear that 
$$ \text{admissible} \Longrightarrow \text{quasi-admissible}$$
for $F$-split orbits $\mca{O}$. To understand further the quasi-admissibility,  we denote by $\left(\wt{G}_{\gamma}^{(n, 2)}\right)_+ \subset \wt{G}_{\gamma}^{(n, 2)}$ the subgroup generated by unipotent elements. For $\star\in \set{0, der}$, one has a commutative diagram:
\begin{equation} \label{CD:pD}
\begin{tikzcd}
\mu_n \times \mu_2 \ar[r, hook] & \wt{G}_{\gamma, \star}^{(n, 2)} \ar[r, two heads, "{p_{\gamma, \star}}"] & G_{\gamma, \star} \\
\Ker(p_{\gamma}^D) \ar[u, hook, "{\iota^D}"] \ar[r, hook] & \left(\wt{G}_{\gamma}^{(n, 2)}\right)_+ \ar[u, hook] \ar[r, two heads, "{p_{\gamma}^D}"]  & G_{\gamma, der} \ar[u, hook] .
\end{tikzcd}
\end{equation}
Here $p_\gamma^D$ is surjective since $G_{\gamma, der}$ is generated by unipotent elements. Note that $\iota^D$ may not be surjective in general. 

\begin{lm} \label{L:qadm}
Let $\mca{O}=\mca{O}_u$ be an $F$-split $\wt{G}^{(n)}$-quasi-admissible orbit.
\begin{enumerate}
\item[(i)] If $n$ is odd, then $\Ker(p_\gamma^D)=\set{1}$ and thus $p_{\gamma, 0}$ splits over $G_{\gamma, \der}$.
\item[(ii)] If $n=2m$ is even, then either $\Ker(p_\gamma^D) = \set{1}$ or $\Ker(p_\gamma^D) = \Delta(\mu_2)$, where $\Delta: \mu_2 \into \mu_n \times \mu_2$ is the diagonal embedding. 
\end{enumerate}
\end{lm}
\begin{proof}
Let $\sigma$ be a finite-dimensional $(n, 2)$-genuine representation of $\wt{G}_{\gamma}^{(n, 2)}$, viewed as a representation of $\wt{G}_{\gamma, 0}^{(n, 2)}$ by restriction. By \cite[Lemma 4.5]{GGS21} (see also the proof of Proposition 6.4 in loc. cit.), we have
$$\left(\wt{G}_{\gamma}^{(n, 2)}\right)_+ \subset \Ker(\sigma).$$
In particular, $\Ker(p_\gamma^D) \subset \Ker(\sigma)$. But since $\sigma$ is $(n, 2)$-genuine, this immediately gives the result.
\end{proof}

Now we set 
$$n^* = {\rm lcm}(n, 2)$$
and consider the multiplication map 
\begin{equation} \label{F:mfr-m}
\mfr{m}: \mu_n \times \mu_2 \longrightarrow \mu_{n^*}.
\end{equation}
By push-out via $\mfr{m}$, we obtain the following
\begin{equation} \label{CD:q-adm}
\begin{tikzcd}
\mu_n \times \mu_2 \ar[r, hook] \ar[d, "{\mfr{m}}"] & \wt{G}_{\gamma, der}^{(n, 2)} \ar[r, two heads, "{p_{\gamma, der}}"] \ar[d] & G_{\gamma, der} \ar[d, equal] \\
\mu_{n^*}  \ar[r, hook] & \mfr{m}_*\left(\wt{G}_{\gamma, der}^{(n, 2)}\right) \ar[r, two heads, "{p_{\gamma, der}^\mfr{m}}"]  & G_{\gamma, der}.
\end{tikzcd}
\end{equation}

Assume that $G_{\gamma} = \prod_{j \in J} G_{\gamma, j}$, where $G_{\gamma, j, der}$ is almost simple for every $j$. Then one has 
\begin{equation}\label{CD:q-adm-j}
\begin{tikzcd}
\mu_{n^*}  \ar[r, hook] & \mfr{m}_*\left(\wt{G}_{\gamma, j, 0}^{(n, 2)}\right) \ar[r, two heads, "{p_{\gamma, j, 0}^\mfr{m}}"]  & G_{\gamma, j, 0} \\
\mu_{n^*}  \ar[r, hook] \ar[u, equal] & \mfr{m}_*\left(\wt{G}_{\gamma, j, der}^{(n, 2)}\right) \ar[r, two heads, "{p_{\gamma, j, der}^\mfr{m}}"]  \ar[u, hook] & G_{\gamma, j, der} \ar[u, hook]
\end{tikzcd}
\end{equation}
for every $j \in J$. In general, the natural set-theoretic map 
$$\prod_j \wt{G}_{\gamma, j}^{(n, 2)} \longrightarrow \wt{G}_{\gamma}^{(n, 2)}$$
may not be a group homomorphism, i.e., block commutativity may fail. However, for the derived subgroup, we have the natural group isomorphism
$$(\prod_j \wt{G}_{\gamma, j, der}^{(n, 2)})/K \simeq \wt{G}_{\gamma, der}^{(n,2)},$$
where $K=\set{(\zeta_j, \xi_j) \in (\mu_n \times \mu_2)^{\val{J}}: \ \prod_j (\zeta_j, \xi_j) = (1, 1)}$.
Indeed, if $Z_j$ denotes the cocharacter lattice of $G_{\gamma, j, der}$ and $Z$ that for $G_{\gamma, der}$. Then any Weyl-invariant quadratic form $Q$ on $Z$ decomposes as $Q = \bigoplus_j Q|_{Z_j}$. This shows that $\wt{G}_{\gamma, j, der}^{(n)}$ commutes with each other in $\wt{G}_{\gamma, der}^{(n)}$, and similarly for $\wt{G}_{\gamma, j, der}^{(2), \phi}$; thus, so does the $\wt{G}_{\gamma, j, der}^{(n, 2)}$.

\begin{prop} \label{P:qadm} 
Keep notations as above. Then the following are equivalent:
\begin{enumerate}
\item[(i)] The $F$-split orbit $\mca{O}$ is quasi-admissible.
\item[(ii)] The map $p_{\gamma, der}^\mfr{m}$ in \eqref{CD:q-adm} splits.
\item[(iii)] For every $j$, the map $p_{\gamma, j, der}^\mfr{m}$ in \eqref{CD:q-adm-j} splits.
\item[(iv)] For every $j$, the cover $\mfr{m}_*\left(\wt{G}_{\gamma, j, 0}^{(n, 2)}\right)$ in \eqref{CD:q-adm-j} has a  finite dimensional $\mu_{n^*}$-genuine representation.
\end{enumerate}
\end{prop}
\begin{proof} 
Let $\mca{O}$ be as given.  The equivalence between (ii) and (iii) is clear in view of the preceding discussion. We first consider the equivalence between (i) and (ii). 

Assuming (i), there are two cases as from Lemma \ref{L:qadm}:
\begin{enumerate}
\item[$\bullet$] If ${\rm Ker}(p_\gamma^D)=\set{1}$, then $p_{\gamma, der}$ splits over $G_{\gamma, der}$ and thus $p_{\gamma, der}^\mfr{m}$ splits as well.
\item[$\bullet$] If ${\rm Ker}(p_\gamma^D) = \Delta(\mu_2)$, then necessarily $2|n$ and $n^*=n$. In this case, by pushing out $\left( \wt{G}_\gamma^{(n,2)} \right)_{+}$ via $\iota^D$, we obtain an isomorphism of extensions from the bottom two lines as in
$$\begin{tikzcd}
\Ker(p_{\gamma}^D) \ar[d, hook, "{\iota^D}"] \ar[r, hook] & \left(\wt{G}_{\gamma}^{(n, 2)}\right)_+ \ar[d, hook] \ar[r, two heads, "{p_{\gamma}^D}"]  & G_{\gamma, der} \ar[d, equal] \\
 \mu_n \times \mu_2 \ar[r, hook] &  \iota^D_*\left( \wt{G}_\gamma^{(n,2)} \right)_+ \ar[r, two heads] & G_{\gamma, der} \\
 \mu_n \times \mu_2 \ar[u, equal]  \ar[r, hook] &  \wt{G}_{\gamma, der}^{(n, 2)} \ar[u, "\simeq"] \ar[r, two heads] & G_{\gamma, der} \ar[u, equal].
\end{tikzcd}$$
Since in this case ${\rm Ker}(\mfr{m}) ={\rm Im}(\iota^D)$, we see that one has a retraction of the short exact sequence
$$\mu_{n^*} \into \mfr{m}_*\left( \wt{G}_{\gamma, der}^{(n, 2)} \right) \onto G_{\gamma, der},$$
which then gives a splitting of $p_{\gamma, der}^\mfr{m}$.
\end{enumerate}
Now, assuming that $p_{\gamma, der}^\mfr{m}$ splits, we have a homomorphism $s$ such that
$$\begin{tikzcd}
\mu_n \times \mu_2 \ar[d] \ar[r, hook]  & \wt{G}_{\gamma, der}^{(n, 2)}/\left( \wt{G}_\gamma^{(n, 2)} \right)_{+}  \ar[ld, "s"]\\
\mu_{n^*}
\end{tikzcd}$$
commutes. This implies in particular that
$$\begin{tikzcd}
\mu_n \times \mu_2 \ar[d, "\mfr{m}"] \ar[r, "\iota"]  & (\mu_n \times \mu_2)/{\rm Ker}(p_\gamma^D)  \ar[ld, "s"]\\
\mu_{n^*}
\end{tikzcd}$$
commutes. Since $\mfr{m}$ is injective on $\mu_n$ and also on $\mu_2$, we see that $\iota$ is injective on $\mu_n$ and on $\mu_2$ as well. Thus, one has a finite dimensional $(n, 2)$-genuine representation of $\wt{G}_{\gamma, der}^{(n, 2)}$ and also of $\wt{G}_{\gamma,0}^{(n, 2)}$ (see \cite[Proposition 6.4]{GGS21}). 

This equivalence between (iii) and (iv) follows an analogous argument for that of (i) and (ii). This completes the proof.
\end{proof}

In view of Proposition \ref{P:qadm}, we assume now that $G_{\gamma, der}$ is almost simple and give an explicit condition for the splitting of $p_{\gamma, der}^\mfr{m}$. Consider the Brylinski--Deligne invariant (see Definition \ref{D:BDI})
\begin{equation}\label{Q1Q2}
    Q_1:=\BDI(\wt{G}_{\gamma, 0}^{(n)}), \quad Q_2:=\BDI(\wt{G}_{\gamma, 0}^{(2), \phi})
\end{equation}
associated to $\wt{G}_{\gamma, 0}^{(n)}$ and  $\wt{G}_{\gamma, 0}^{(2), \phi}$ respectively. Then we can view $\wt{G}_{\gamma, 0}^{(n)}$ as an $n^*$-fold cover associated with $(n^*/n)Q_1$, and view $\wt{G}_{\gamma, 0}^{(2), \phi}$ as $n^*$-fold cover associated with $(n^*/2)Q_2$. Thus, we see that
$$\mfr{m}_*\left(\wt{G}_{\gamma, der}^{(n, 2)}\right)$$
is an $n^*$-fold cover with Brylinski--Deligne invariant 
$$(n^*/n)Q_1 + (n^*/2)Q_2.$$

\begin{prop} \label{P:exp-spl}
Let $Q_1 \in \Z$ (resp. $Q_2\in \Z$) be the Brylinski--Deligne invariant of $\wt{G}_{\gamma, 0}^{(n)}$ (resp. $\wt{G}_{\gamma, 0}^{(2), \phi}$).  Assume that n is coprime to the size  of ${\rm Ker}(\mbf{f}: \mbf{G}_{\gamma, sc} \onto \mbf{G}_{\gamma, \der})$.
Then $p_{\gamma, der}^\mfr{m}$ splits if and only if
$$n^* \text{ divides }  (n^*/n)Q_1 + (n^*/2)Q_2,$$
or equivalently,
\begin{enumerate}
\item[(i)] $n|Q_1$ when $2|Q_2$, and
\item[(ii)] $n/\gcd(n, Q_1)=2$ when $2\nmid Q_2$. 
\end{enumerate}
\end{prop}
\begin{proof}
Recall that the homomorphism $\mbf{f}: \mbf{G}_{\gamma, sc} \onto \mbf{G}_{\gamma, der}$ induces a map 
$$f: G_{\gamma, sc} \onto G_{\gamma, der} \into \mbf{G}_{\gamma, der}(F).$$
For any $k$-fold cover $\wt{G}_{\gamma, der}^{(k)}$ of $G_{\gamma, der}$, by pull-back via $f$ one has a $k$-fold cover $\wt{G}_{\gamma, sc}^{(k)}$, as depicted in the following diagram:
\begin{equation} \label{CD:der-sc}
\begin{tikzcd}
& \mu_k \ar[r, equal] \ar[d, hook] & \mu_k \ar[d, hook] \\
{\rm Ker}(f) \ar[r, hook, "\iota"] \ar[d, equal] & \wt{G}_{\gamma, sc}^{(k)} \ar[r, two heads] \ar[d, two heads] & \wt{G}_{\gamma, der}^{(k)} \ar[d, two heads, "q"] \\
{\rm Ker}(f) \ar[r, hook, "\iota"] & G_{\gamma, sc} \ar[u, bend left= 40, "s"] \ar[r, two heads, "f"] & G_{\gamma, der}.
\end{tikzcd}
\end{equation}
If $\wt{G}_{\gamma, der}^{(k)}$ splits, then $\wt{G}_{\gamma, sc}^{(k)}$ splits by its definition as pull-back.
On the other hand, let 
$$s: G_{\gamma, sc} \into \wt{G}_{\gamma, sc}^{(k)}$$
 be a splitting, which is actually unique since $G_{\gamma, sc}$ is equal to its derived subgroup. The assumption implies that $\gcd(\val{{\rm Ker}(f)}, n)=1$, and thus the splitting of $\Ker(f)$ into $\wt{G}_{\gamma, sc}^{(k)}$  is unique. That is, the left lower square in \eqref{CD:der-sc} involving $s$ commutes. This implies that $\wt{G}_{\gamma, der}^{(k)}$ splits over $G_{\gamma, der}$.

Applying the above to the case $k=n^*$, we see that $p_{\gamma, der}^\mfr{m}$ splits if and only if $\wt{G}_{\gamma, sc}^{(n^*)}$ splits over $G_{\gamma, sc}$. However, since $G_{\gamma, sc}$ is simply-connected, its $n^*$-fold cover splits if and only if $n^*$ divides the Brylinski--Deligne invariant. This completes the proof.
\end{proof}


\begin{rmk}  \label{R:run-eg}
In general, consider $\mbf{f}: \mbf{G}_{sc} \onto \mbf{G}$, the simply-connected cover of a semi-simple group $\mbf{G}$, which gives $f: G_{sc} \onto G_{der} \into G$. There exists a non-split extension of $G_{der}$ whose pull-back is split. For example, consider $\mbf{G}=\SO_3$ and thus $\mbf{G}_{sc}=\SL_2$. Then one has $G_{der}=\SL_2(F)/\set{\pm 1}$, where we identify $\set{\pm 1}$ with the center of $\SL_2(F)$. Consider the double cover
$$\wt{G}_{der}^{(2)}:=(\mu_2 \times \SL_2(F))/{\rm Im}(\sigma),$$
where the map $\sigma: \set{\pm 1} \to \mu_2 \times \SL_2(F)$ is given by $\sigma(a)=((\varpi, a)_2, a)$. It is easy to see that the pull-back of $\wt{G}_{der}^{(2)}$ to $G_{sc}=\SL_2$ is a split extension. However, $\wt{G}_{der}^{(2)}$ splits over $G_{der}$ if and only if $(\varpi, -1)_2=1$. In fact, $\wt{G}_{der}^{(2)}$ is a Brylinski--Deligne cover associated with $(D=0, \eta)$, and is the running example for several interesting phenomena discussed in \cite{GG}.

However, in the notation of Proposition \ref{P:exp-spl}, we believe that the splitting of $\wt{G}_{\gamma, sc}^{(k)}$ is always equivalent to the splitting of $\wt{G}_{\gamma, der}^{(k)}$, without the assumption imposed there. That is, the covers $\wt{G}_{\gamma, der}^{(k)}$ that arise are not of the type in the preceding paragraph, essentially due to the fact that we consider in this paper covers associated with $(D, \eta=\mbf{1})$. We are not able to prove this expectation in full generality, though.
\end{rmk}

The importance of the quasi-admissibility is given as follows.

\begin{thm}[{\cite[Theorem 1]{GGS21}}] \label{T:GGS}
Let $\mca{O}$ be an $F$-split orbit such that  $\mca{O} \in \mca{N}_{\Wh}^{\rm max}(\pi)$ for some $\pi \in \Irrg(\wt{G})$. Then $\mca{O}$ is quasi-admissible.
\end{thm}
\begin{proof}
The proof of \cite[Theorem 1.4]{GGS21} in \S 5.1 there actually shows that if $\mca{O} \in \mca{N}_{\Wh}^{\rm max}(\pi)$, then the non-zero finite-dimensional space $\pi_{N_\mca{O}, \psi_\mca{O}}$ affords an $(n, 2)$-genuine representation of  $\wt{G}_{\gamma}^{(n, 2)}$ and thus is quasi-admissible.
\end{proof}

\subsection{Admissibility versus quasi-admissibility}
Again, consider $\mbf{f}: \mbf{G}_{sc} \onto \mbf{G}$ of a semisimple $\mbf{G}$ and the arising $f: G_{sc} \to G$, which is not necessarily surjective. It is possible that a cover over $G$ does not split, but its restriction to $G_{der}$ splits and thus also the pull-back to $G_{sc}$. Note that if $\mbf{G} = \mbf{G}_{\gamma, 0}$ for some $\gamma$, then this gives an example of orbit which is quasi-admissible but not admissible.

As a more concrete example of this, consider $\mbf{G}=\SO_k$. Let $\wt{\SL}_{k}^{(m)}$ be the $m$-fold cover with Brylinski--Deligne invariant $Q(\alpha^\vee)$, where $\alpha^\vee$ is any coroot of $\SL_k$. Consider the cover $\wt{\SO}_k^{(m)}$ obtained from restricting $\wt{\SL}_k$ via the inclusion $\SO_k \into \SL_k$. Thus,
$$\BDI(\wt{\SO}_k^{(m)})=2Q(\alpha^\vee) \text{ for } k \gest 4,  \text{ and } \BDI(\wt{\SO}_3^{(m)})=4Q(\alpha^\vee).$$

\begin{lm} \label{L:SO}
Keep notations as above. The cover $\wt{\SO}_k^{(m)}, k\gest 3$, has a finite-dimensional $\mu_m$-genuine representation if and only if $m|\BDI(\wt{\SO}_k^{(m)})$.
\end{lm}
\begin{proof}
We first show the ``only if" part. If $\wt{\SO}_k^{(m)}$ has a finite dimensional $\mu_m$-genuine representation, then it gives rise to one for $\wt{G}_{sc}=\wt{\Spin}_k$. It then follows that $\wt{G}_{sc}$ splits over $G_{sc}$, which implies that $m$ divides $\BDI(\wt{G}_{sc})=\BDI(\wt{\SO}_k^{(m)})$.

Now for the ``if" part, we first assume $m|2Q(\alpha^\vee)$.
Then the desired result essentially follows from the discussion in \cite[\S 2.4]{GSS1}. Let $\Z e$ be the cocharacter lattice of $\mbf{G}_m$ and let $Q: \Z e \to \Z$ be a quadratic form. It gives the cover 
\begin{equation} \label{F:covF}
 \mu_m \into \wt{F^\times} \onto F^\times
 \end{equation}
associated with $Q$. Now assume $m|2Q(e)$, then the cover $\wt{F^\times}$ splits over $F^{\times 2}$ and thus we obtain a cover
$$\mu_m \into \wt{F^\times/2} \onto F^\times/2,$$
where $F^\times/2:=F^\times/F^{\times 2}$. This extension is abelian, though non-split. Consider the spinor norm
$$\mca{N}: \SO_k \to F^\times/2$$
and the pull-back cover 
$$\mu_m \into \mca{N}^*(\wt{F^\times/2}) \onto \SO_k.$$
If we consider the $\wt{\SL}_k^{(m)}$ with Brylinski--Deligne invariant $Q(\alpha^\vee):=Q(e)$, then by restriction it gives rise to $\wt{\SO}_k^{(m)}$ and one has
$$ \wt{\SO}_k^{(m)} \simeq \mca{N}^*(\wt{F^\times/2}),$$
see \cite[Example 2.10]{GSS1}. Thus, there is a one-dimensional $\mu_m$-genuine representation of $\wt{\SO}_k$.

For the ``if" part, the only remaining case is when $k=3$ and $m=4Q(\alpha^\vee)$. In this case, the extension \eqref{F:covF} associated with $Q(e):=Q(\alpha^\vee)$ also splits over $F^{\times 2}$. The pull-back cover 
$$\mu_m \into \mca{N}^*(\wt{F^\times/2}) \onto \SO_3$$
is equal to the cover $\wt{\SO}_3^{(m)}$ of $\SO_3$ restricted from $\wt{\SL}_3^{(m)}$ with $\BDI(\wt{\SL}_3^{(m)})=Q(\alpha^\vee)$ via $\SO_3 \into \SL_3$. Since in this case $\wt{F^\times/2}$ clearly has a finite-dimensional $\mu_m$-genuine representation, so does the cover $\wt{\SO}_3^{(m)}$.

Combining all the above, the proof is completed.
\end{proof}

Thus, suppose $\mca{O}_u$ and $\gamma$ are such that
\begin{enumerate}
\item[--] $G_{\gamma, 0} = \SO_k$ as above,
\item[--] $\BDI(G_{\gamma, 0}^{(n)})=2Q(\alpha^\vee)$ with $n|2Q(\alpha^\vee)$,
\item[--] $4|\BDI(G_{\gamma, 0}^{(2), \phi})$.
\end{enumerate}
Such orbit $\mca{O}_u$ is always quasi-admissible. But it might not be admissible, if $n=2Q(\alpha^\vee)$. A concrete example is the orbit $B_3$ of the exceptional group $F_4$, see \S \ref{SS:F4}.

\subsection{$\wt{G}^{(n)}$-raisable orbits}\label{sec:raisable}
Now we discuss about the notion of $\wt{G}$-raisability  following \cite{JLS}. It pertains to a ``local" version of non-quasi-admissibility relative to a ``good" choice of an $\SL_2 \subset G_\gamma$, if it exists.

Let $\mfr{g}_\gamma \subset \mfr{g}$ be the centralizer of $\gamma$ in $\mfr{g}$. Assume:
\begin{enumerate}
\item[--] (C0) There is a non-trivial  map
$$\tau: \mfr{sl}_2 \to \mfr{g}_\gamma.$$
\end{enumerate}
In this case, we write
$$\mfr{sl}_{2, \tau} = \tau(\mfr{sl}_2).$$
If we set
$$u_\tau:=\tau(e_-(x))$$
for some $0 \ne x \in F$, then 
$$\gamma'=\gamma \oplus \tau: \mfr{sl}_2 \to \mfr{g}$$
is a Jacobson--Morozov map associated to the nilpotent element $u'=u + u_\tau$. Let $\mfr{g}[j, l] \subset \mfr{g}$ be the space of vectors with $\gamma$-weight $j$ and $\tau$-weight $l$.  We assume the following:
\begin{enumerate}
\item[--] (C1) The $\tau$-weights $l$ are bounded by 2.
\item[--] (C2) As $\mfr{sl}_{2, \tau}$-module,
$$\mfr{g}[1] = \mfr{g}[1]^{\mfr{sl}_{2,\tau}} \oplus m V_2.$$
\item[--] (C3) $\dim \mfr{g}[0, 2] = 1 + \dim \mfr{g}[2, 2]$.
\end{enumerate}
By abuse of notation, we still denote by 
$$\tau: \SL_2 \to G_\gamma$$
the map corresponding to $\tau$. Also, there is a natural group homomorphism
$$\phi_m: \SL_2 \to \Sp_{2m}$$
arising from $\phi \circ \tau$ and (C2) above. Recall that 
$$\phi: G_\gamma \longrightarrow \Sp(\mfr{g}[1]),$$
is the natural group homomorphism.
This gives rise to the following diagram
\begin{equation} \label{CD:02}
\begin{tikzcd}
\wt{\SL}_2^{(n), \tau} \times_{\SL_2} \wt{\SL}_2^{(2), \phi_m} \ar[r]  \ar[d] \ar[rd, "p_\tau"] & \wt{\SL}_2^{(2), \phi_m} \ar[r]  \ar[d, two heads] & \Mp_{2m}  \ar[d, two heads] \\
\wt{\SL}_2^{(n), \tau} \ar[r, two heads] \ar[d] & \SL_2 \ar[r, "\phi_m"] \ar[d, "\tau"] & \Sp_{2m} \\
\wt{G}^{(n)} \ar[r, two heads] & G.
\end{tikzcd}
\end{equation}
From now, we write
$$\wt{\SL}_{2,\tau}^{(n, 2)}:=\wt{\SL}_2^{(n), \tau} \times_{\SL_2} \wt{\SL}_2^{(2), \phi_m}.$$
The two covering groups $\wt{\SL}_2^{(n), \tau}$ and $\wt{\SL}_2^{(2), \phi_m}$ both arise from the Brylinski--Deligne framework. Consider $n^*={\rm lcm}(n, 2)$ and the push-out of $\wt{\SL}_{2,\tau}^{(n, 2)}$ via $\mfr{m}$ as in \eqref{F:mfr-m}. This gives
\begin{equation} \label{CD:p-tau}
\begin{tikzcd}
\mu_n\times \mu_2 \ar[r, hook] \ar[d, "\mfr{m}"] & \wt{\SL}_{2,\tau}^{(n, 2)} \ar[r, two heads] \ar[d] & \SL_2 \ar[d, equal] \\
\mu_{n^*} \ar[r, hook] & \mfr{m}_*(\wt{\SL}_{2,\tau}^{(n, 2)}) \ar[r, two heads, "{p_\tau^\mfr{m}}"] & \SL_2.
\end{tikzcd}
\end{equation}

\begin{dfn} \label{D:rais}
An $F$-split nilpotent orbit $\mca{O} \subset \mfr{g}$ is called $\wt{G}^{(n)}$-raisable if there exists $\tau: \mfr{sl}_2 \to \mfr{g}_\gamma$ satisfying (C0)--(C3) such that the projection $p_\tau^\mfr{m}$ in \eqref{CD:p-tau} does not split.
\end{dfn}

Again, we denote by 
\begin{equation}\label{Q1Q2tau}
    (Q_1^\tau, Q_2^\tau):=\left(\BDI(\wt{\SL}_2^{(n), \tau}), \BDI(\wt{\SL}_2^{(2), \phi_m})\right)
\end{equation}
the pair of Brylinski--Deligne invariants associated with $\wt{\SL}_2^{(n), \tau}$ and $\wt{\SL}_2^{(2), \phi_m}$ respectively. We have an analogue of Proposition \ref{P:exp-spl}, noting that $Q_2^\tau =m$ in this case.

\begin{prop} \label{P:rais}
Keep notations as above. Then $p_\tau^\mfr{m}$ splits if and only if 
$$n^*| \left( Q_1^\tau n^*/n + Q_2^\tau n^*/2 \right),$$
or equivalently,
\begin{enumerate}
\item[(i)]  $n|Q_1^\tau$ when $m\in \N_{\gest 0}$ is even, and
\item[(ii)] $n/\gcd(n, Q_1^\tau)=2$ when $m\in \N_{\gest 0}$ is odd.
\end{enumerate}
\end{prop}

As a first example, we have
\begin{cor} \label{C:rais-0}
Let $\wt{G}^{(n)}$ be the $n$-fold cover of an almost simple simply-connected $G$ with $\BDI(\wt{G}^{(n)})=1$. Then the zero orbit is always raisable if $n\gest 2$.  
\end{cor}
\begin{proof}
Pick any long root $\alpha$ and let $\tau: \mfr{sl}_2 \to \mfr{sl}_{2,\alpha}$ be the identity. In this case (C1)--(C3) are all satisfied. We have $\mfr{g}[1]=0$. This shows that
$$(Q_1^\tau, Q_2^\tau) = (1, 0).$$
It then follows from Proposition \ref{P:rais} that the orbit $\set{0}$ is raisable if $n\gest 2$.
\end{proof}

The importance of the raisability is seen from the following result given by Jiang--Liu--Savin.

\begin{thm}[{\cite{JLS}}] \label{T:JLS}
Let $\mca{O}$ be a $\wt{G}$-raisable orbit. Then it does not lie in the wavefront set $\mca{N}_{\rm Wh}^{\rm max}(\pi)$ of any $\pi \in \Irrg(\wt{G})$.
\end{thm}

\subsection{A comparison} 
Consider the three properties of $\mca{O} \subset \mfr{g}$:
\begin{enumerate}
\item[(P1)] $\mca{O}$ is $\wt{G}$-raisable;
\item[(P2)] $\mca{O}$ is not $\wt{G}$-quasi-admissible;
\item[(P3)] $\mca{O}$ does not lie in $\mca{N}_{\rm Wh}^{\rm max}(\pi)$ for any $\pi \in \Irrg(\wt{G})$.
\end{enumerate}
In view of Theorems \ref{T:GGS} and \ref{T:JLS}, one has the implications
$$\text{(P1)} \Longrightarrow \text{(P3)} \Longleftarrow \text{(P2)}.$$
In general, raisable and non-quasi-admissible orbits are not identical. Even in the case $G_{\gamma, 0} = \SL_2$, these two properties are not necessarily equivalent, though can be shown to be so in many cases.

\section{Explicit analysis for each Cartan type} \label{S:ex-ana}
For each covering group $\wt{G}^{(n)}$ and each $F$-split nilpotent orbit $\mca{O} \subset \mfr{g}$, our goal is to determine the $n$ such that $\mca{O}$ is quasi-admissible and raisable. 

The computation in this section is case by case. For classical groups of type $A$ to $D$, we work out the details. For exceptional groups, we illustrate the method by giving full details for $G_2$ and $F_4$. For $E_6, E_7$ and $E_8$, we only consider the raisability and quasi-admissibility of those orbits which are conjecturally equal to the wavefront sets of theta representations. Whenever our method applies, we show that these conjectural wavefront orbits of theta representations are not raisable and are also quasi-admissible.

Let 
$$\mfr{p}=(p_1^{d_1} \cdots p_i^{d_i} \cdots p_k^{d_k})$$
 denote a partition of $r$, where $p_i$'s are distinct with $d_i \in \N_{\gest 1}$. For any $p \in \set{p_i}_i$ appearing in $\mfr{p}$, we define two functions
\begin{equation} \label{D:A-B}
\mfr{A}(\mfr{p}, p) = \sum_{\substack{i:\ p_i> p, \\ p_i - p +1 \in 2\Z}} d_i, \quad \mfr{B}(\mfr{p}, p) = \sum_{\substack{i:\ p_i < p, \\ p_i - p +1 \in 2\Z}} d_i. 
\end{equation}


\subsection{Type A} 
For type $A$ groups, we analyze $\GL_r$ instead of $\SL_r$ for convenience, since the parabolic subgroups of the former allow for a simpler description.
Consider the Dynkin diagram of simple roots for $\GL_r$:

$$\qquad 
\begin{picture}(4.7,0.2)(0,0)
\put(1,0){\circle{0.08}}
\put(1.5,0){\circle{0.08}}
\put(2,0){\circle{0.08}}
\put(2.5,0){\circle{0.08}}
\put(3,0){\circle{0.08}}
\put(1.04,0){\line(1,0){0.42}}
\multiput(1.55,0)(0.05,0){9}{\circle{0.02}}
\put(2.04,0){\line(1,0){0.42}}
\put(2.54,0){\line(1,0){0.42}}
\put(1,0.1){\footnotesize $\alpha_{1}$}
\put(1.5,0.1){\footnotesize $\alpha_{2}$}
\put(2,0.1){\footnotesize $\alpha_{r-3}$}
\put(2.5,0.1){\footnotesize $\alpha_{r-2}$}
\put(3,0.1){\footnotesize $\alpha_{r-1}$}
\end{picture}
$$
\vskip 10pt

For the orbit $\mca{O} \subset \GL_r$ associated with the partition 
$$\mfr{p}_\mca{O}=(p_1^{d_1} \cdots p_k^{d_k} q_1 \cdots q_m),$$
where $d_i \gest 2$ and $p_i, q_j$ are distinct, one has
$$G_\gamma = \prod_{i=1}^k \GL_{d_i}^{\Delta, p_i} \simeq \prod_{i=1}^k \GL_{d_i}.$$
Here $\GL_{d}^{\Delta, p}$ means the image of the diagonal embedding of $\GL_d$ into $\prod_{i=1}^p \GL_d$. For each pair $(p_i, p_j)$ with $p_i$ even and $p_j$ odd, we have a map
$$\phi_{p_i, p_j}: \GL_{d_i} \times \GL_{d_j} \to \Sp(\mfr{g}_{p_i, p_j}[1])$$
arising from the natural action of $\GL_{d_i} \times \GL_{d_j}$, where
$$\mfr{g}_{p_i, p_j}[1] \simeq 2 \min \set{p_i, q_j} \cdot (\C^{d_i} \otimes \C^{d_j}),$$
with $\C^{d_i}$ affording the standard representation of $\GL_{d_i}$.
We get
$$\mfr{g}[1] = \bigoplus_{(p_i, p_j)} \mfr{g}_{p_i, p_j}[1]$$
ranging over pairs $(p_i, p_j)$ as above. Then the representation
$$\phi: G_\gamma \to \Sp(\mfr{g}[1])$$
arises from gluing all the $\phi_{p_i, p_j}$ together. For each $d_i$, denote by 
$$(Q_{1, d_i}, Q_{2, d_i})$$
the two Brylinski--Deligne invariants attached to $\GL_{d_i} \subset G_\gamma$ as defined in \eqref{Q1Q2}. We have 
$$Q_{1, d_i} = p_i \cdot Q(\alpha^\vee)$$
 for any root $\alpha$ of $\GL_r$, and $Q_{2, d_i} \in 2\Z$ for every $1\lest i\lest k$. By Propositions \ref{P:qadm} and \ref{P:exp-spl}, we see that $\mca{O}$ is quasi-admissible if and only if 
$$n_\alpha|p_i$$
for every $i$. Recall that $n_{\alpha}$ is defined in \eqref{n_alpha}.


Now we turn to raisability. Here we choose $1\lest i \lest k$ and take 
$$\tau_i: \SL_2 \to \GL_{d_i}^{\Delta, p_i} \into G_\gamma$$
to be the embedding corresponding to any simple root of $\GL_{d_i}$. Then this $\tau_i$ satisfies (C0)--(C3) as in \S \ref{sec:raisable}, and regardless of $m$ one has
$$\wt{\SL}_2^{(2), \phi_m} \simeq \mu_2 \times \SL_2,$$
i.e., $Q_{2, d_i}^{\tau_i} \in 2\Z$. On the other hand, for any Brylinski--Deligne cover of $\GL_r$, the pull-back covering
$$\mu_n \into \wt{\SL}_2^{(n), \tau} \onto \SL_2$$
has Brylinski--Deligne invariant $Q_{1, d_i}^{\tau_i} = p_i \cdot Q(\alpha^\vee)$. Hence, by Proposition \ref{P:rais}, the orbit  $\mca{O}$ is raisable if $n\nmid p_i \cdot Q(\alpha^\vee)$, i.e., $n_\alpha \nmid p_i$. For the definitions of $Q_{1, d_i}^{\tau_i}$ and $Q_{2, d_i}^{\tau_i}$, see \eqref{Q1Q2tau}.

\begin{thm} \label{T:typeA}
Keep notations as above and assume $Q(\alpha^\vee) \ne 0$. Then the orbit $\mca{O}=(p_1^{d_1} \cdots p_k^{d_k} q_1 \cdots q_m)$ is 
\begin{enumerate}
\item[(i)] quasi-admissible if and only if $n_\alpha | p_i$ for every i;
\item[(ii)] raisable if $n_\alpha \nmid p_i$ for some $i$.
\end{enumerate}
In particular, the orbit $\mca{O}$ is always $\wt{\GL}_r^{(n)}$-raisable for any $n > \val{Q(\alpha^\vee)} \cdot \min\set{p_i}$.
\end{thm}

Consider the orbit
$$\mca{O}^{r, n}:=(n_\alpha^ab)$$
of $\GL_r$, where $r=a \cdot n_\alpha+b, 0\lest b < n_{\alpha}$. We believe the following holds:

\begin{conj} \label{C:low-bd}
Let $\pi \in \Irrg(\wt{\GL}_r^{(n)})$. Then one has $\mca{O}^{r, n} \lest \mca{O}$ for every $\mca{O} \in \mca{N}_{\rm Wh}^{\rm max}(\pi)$.
\end{conj}
Here the orbit $\mca{O}^{r, n}$ is equal to the leading wavefront set of the theta representation of $\wt{\GL}_r^{(n)}$, see \cites{Sav2, Cai19}. If $\mca{O}^{r, n}$ is the regular orbit, then Conjecture \ref{C:low-bd} implies that every $\pi \in \Irrg(\wt{\GL}_r^{(n)})$ is generic in this case, which is exactly the content of \cite[Conjecture 6.9 (ii)]{Ga6}.

\subsection{Type B and D}
In this subsection, we consider classical groups of orthogonal type.

For $\SO_{2r+1}, r\gest 2$, its Dynkin diagram of simple roots  is given as follows

$$ \qquad 
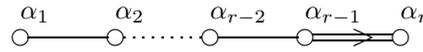
\begin{picture}(4.7,0.2)(0,0)
\put(1,0){\circle{0.08}}
\put(1.5,0){\circle{0.08}}
\put(2,0){\circle{0.08}}
\put(2.5,0){\circle{0.08}}
\put(3,0){\circle{0.08}}
\put(1.04,0){\line(1,0){0.42}}
\multiput(1.55,0)(0.05,0){9}{\circle*{0.02}}
\put(2.04,0){\line(1,0){0.42}}
\put(2.54,0.015){\line(1,0){0.42}}
\put(2.54,-0.015){\line(1,0){0.42}}
\put(2.74,-0.04){$>$}
\put(1,0.1){\footnotesize $\alpha_1$}
\put(1.5,0.1){\footnotesize $\alpha_2$}
\put(2,0.1){\footnotesize $\alpha_{r-2}$}
\put(2.5,0.1){\footnotesize $\alpha_{r-1}$}
\put(3,0.1){\footnotesize $\alpha_r$}
\end{picture}
$$
\vskip 20pt
\noindent 
We consider the natural covering
$$\wt{\SO}_{2r+1} \into \wt{\SL}_{2r+1}$$
obtained from restriction of the covering $\wt{\SL}_{2r+1}$ with Brylinski--Deligne invariant 
$$\BDI(\wt{\SL}_{2r+1})=1=Q(\alpha^\vee),$$
where $\alpha^\vee$ is any coroot of ${\SL}_{2r+1}$.  
This gives
$$Q(\alpha_r^\vee)=4\cdot Q(\alpha^\vee)=4 \text{ and } Q(\alpha_1^\vee) = 2\cdot Q(\alpha^\vee)=2.$$
That is, 
$$\BDI(\wt{\SO}_{2r+1}) = 2.$$
For $\SO_{2r}, r\gest 3$ we have the Dynkin diagram of simple roots to be

$$
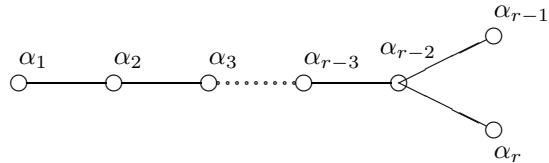
\begin{picture}(4.7,0.4)(0,0)
\put(1,0){\circle{0.08}}
\put(1.5,0){\circle{0.08}}
\put(2,0){\circle{0.08}}
\put(2.5,0){\circle{0.08}}
\put(3,0){\circle{0.08}}
\put(3.5, 0.25){\circle{0.08}}
\put(3.5, -0.25){\circle{0.08}}
\put(1.04,0){\line(1,0){0.42}}
\put(1.54,0){\line(1,0){0.42}}
\multiput(2.05,0)(0.05,0){9}{\circle{0.02}}
\put(2.54,0){\line(1,0){0.42}}
\put(3.00,0){\line(2,1){0.46}}
\put(3.00,0){\line(2,-1){0.46}}
\put(1,0.1){\footnotesize $\alpha_1$}
\put(1.5,0.1){\footnotesize $\alpha_2$}
\put(2,0.1){\footnotesize $\alpha_3$}
\put(2.5,0.1){\footnotesize $\alpha_{r-3}$}
\put(2.9,0.15){\footnotesize $\alpha_{r-2}$}
\put(3.5,0.35){\footnotesize $\alpha_{r-1}$}
\put(3.5,-0.4){\footnotesize $\alpha_r$}
\end{picture}
$$
\vskip 30pt
\noindent We have the cover $\wt{\SO}_{2r}$ obtained from restricting the above $\wt{\SL}_{2r}$ via $\SO_{2r} \into \SL_{2r}$. In this case,
$$\BDI(\wt{\SO}_{2r}) =2.$$


Consider a nilpotent orbit
$$\mfr{p}_\mca{O}=(p_1^{d_1} \cdots p_i^{d_i}\cdots p_k^{d_k} q_1^{e_1} \cdots q_j^{e_j} \cdots q_m^{e_m}),$$
where $p_i$ are even, $q_j$ are odd, and $d_i\gest 1, e_j\gest 1$. Thus, $d_i$ are necessarily even.
To discuss about the quasi-admissibility and raisability of the orbit $\mfr{p}_\mca{O}$, we first note that
$$G_{\gamma, 0} = \left( \prod_{i=1}^k \Sp_{d_i}^{\Delta, p_i} \right) \times \left( \prod_{j=1}^m \SO_{e_j}^{\Delta, q_j} \right).$$
For each pair $(p_i, q_j)$, we have 
$$\mfr{g}_{p_i, q_j}[1] \simeq \min \set{p_i, q_j} \cdot (\C^{d_i} \otimes \C^{e_j}),$$
where $\C^{d_i}$ affords the natural action of $\Sp_{d_i}$, and $\C^{e_j}$ that of $\SO_{e_j}$; also, 
$$\mfr{g}[1] = \bigoplus_{(p_i, q_j)} \mfr{g}_{p_i, q_j}[1]$$
as $G_\gamma$-representations, see \cite[\S 5]{Nev99}. Since $\BDI(\Mp(\mfr{g}[1]))=1$, we get
\begin{enumerate}
\item[--] the Brylinski--Deligne invariant for $\wt{\Sp}_{d_i}^{(2), \phi}$ is 
$$\sum_j \min \set{p_i, q_j} e_j = \left( \sum_{j:\ p_i > q_j} q_j e_j \right) + p_i \left(\sum_{j:\ p_i < q_j} e_j\right);$$
\item[--] the Brylinski--Deligne invariant for $\wt{\SO}_{e_j}^{(2), \phi}, e_j\gest 4$ is 
$$\sum_i 2 \min \set{p_i, q_j} d_i = 2q_j \left( \sum_{i:\ p_i > q_j} d_i \right) + 2 \left( \sum_{i:\ p_i < q_j} p_i d_i\right).$$
\end{enumerate}
Similarly, we have that
\begin{enumerate}
\item[--] the Brylinski--Deligne invariant for $\wt{\Sp}_{d_i}^{(n)}$ is $p_i$;
\item[--] the Brylinski--Deligne invariant for $\wt{\SO}_{e_j}^{(n)}, e_j\gest 4$ is $2q_j$ (and for $\wt{\SO}_{3}^{(n)}$, it is $4q_j$).
\end{enumerate}

Note that in general
$$\wt{G}_{\gamma, 0}^{(n, 2)} \neq \left( \prod_{i=1}^k \wt{\Sp}_{d_i}^{(n, 2)} \right) \times_{\mu_n} \left( \prod_{j=1}^m \wt{\SO}_{e_j}^{(n, 2)} \right).$$
However, $\wt{G}_{\gamma, 0}^{(n, 2)}$ has $(n, 2)$-genuine representation if and only if each factor has as well, which follows easily from \cite[Proposition 6.4]{GGS21}. Thus, it suffices to consider each factor in $G_{\gamma, 0}$.
The $n^*$-fold cover $\mfr{m}_*(\wt{\Sp}_{d_i}^{(n, 2)})$ has Brylinski--Deligne invariant
$$(n^*p_i/n) + \left( \sum_j \min \set{p_i, q_j} e_j\right) n^*/2$$
and $\mfr{m}_*(\wt{\SO}_{e_j}^{(n, 2)}), e_j\gest 4$ has Brylinski--Deligne invariant
$$(2n^*q_j/n) + n^*\left(\sum_i \min \set{p_i, q_j} d_i\right),$$
and 
$\mfr{m}_*(\wt{\SO}_{3}^{(n, 2)})$ has Brylinski--Deligne invariant
$$(4n^*q_j/n) + 2n^*\left(\sum_i \min \set{p_i, q_j} d_i\right).$$

Hence, $\mfr{m}_*(\wt{\Sp}_{d_i}^{(n, 2)}), d_i \gest 2$ splits over $\Sp_{d_i}$ if and only if the following hold:
\begin{enumerate}
\item[--] $n| p_i$ and $\mfr{B}(\mfr{p}_\mca{O}, p_i) \in 2\Z$, if $n$ is odd;
\item[--] $n| p_i$, if $n$ is even and $\mfr{B}(\mfr{p}_\mca{O}, p_i) \in 2\Z$;
\item[--] $\gcd(n, p_i)=n/2$, if $n$ is even and $\mfr{B}(\mfr{p}_\mca{O}, p_i) \notin 2\Z$.
\end{enumerate}
 On the other hand, it follows from Lemma \ref{L:SO} that $\wt{\SO}_{e_j}^{(n, 2)}, e_j\gest 3$, has a finite dimensional $(n, 2)$-genuine representation in the following cases:
 \begin{enumerate}
 \item[--] $n|q_j$, if $n$ is odd;
 \item[--] $n|(2q_j)$, if $n$ is even and $e_j\gest 4$;
 \item[--] $n|(4q_j)$, if $n$ is even and $e_j=3$.
 \end{enumerate}
%

Now, regarding raisability, we implement \cite[\S 8--9]{JLS} and consider the two cases:
\begin{enumerate}
\item[$\bullet$] (B/D-Sym) Suppose there exists $d_i \gest 2$. Let $\tau: \SL_2 \into \Sp_{d_i}$ be the embedding corresponding to the long simple root of $\Sp_{d_i}$. Then this $\tau$ satisfies (C1)--(C3) where $\mfr{g}[1]=\mfr{g}[1]^{\mfr{sl}_{2,\tau}} \oplus m V_2$ with
$$m=p_i \cdot \left(\sum_{q_j: \ q_j > p_i} e_j \right) + \sum_{q_j: \ q_j < p_i} q_j \cdot e_j = p_i \cdot \mfr{A}(\mfr{p}_\mca{O}, p_i) + \sum_{q_j: \ q_j < p_i} q_j \cdot e_j.$$
\item[$\bullet$] (B/D-Ort) Suppose there exists $e_j \gest 4$ and thus a Levi subgroup 
$$\GL_2 \times \SO_k \subset \SO_{e_j}.$$
Let $\tau: \SL_2 \into \SO_{e_j}$ be corresponding to the root of $\GL_2$. This $\tau$ satisfies (C1)--(C3) and
$\mfr{g}[1]=\mfr{g}[1]^{\mfr{sl}_{2,\tau}} \oplus m V_2$, where
$$m=q_j \cdot \left( \sum_{p_i:\ p_i > q_j} d_i \right) + \sum_{p_i: \ p_i < q_j} p_i \cdot d_i =q_j \cdot \mfr{A}(\mfr{p}_\mca{O}, q_j) +   \sum_{p_i: \ p_i < q_j} p_i \cdot d_i.$$
\end{enumerate}

To give a sufficient condition for the orbit $\mca{O}$ to be raisable, we proceed with considering the two covering groups
$$\wt{\SL}_2^{(2), \phi_m} \text{ and } \wt{\SL}_2^{(n), \tau}.$$
Again, we have two cases depending on the consideration of symplectic or orthogonal stabilizer.
\begin{enumerate}
\item[$\bullet$] (B/D-Sym) We have $Q_2^\tau=m$ and  
$$Q_1^\tau=p_i$$
for the Brylinski--Deligne invariants.
 It follows from Proposition \ref{P:rais} that the cover
 $$\mfr{m}_*(\wt{\SL}_2^{(n,2)}) \onto \SL_2$$
 splits if and only if $ \sum_{q_j: \ q_j < p_i} e_j$ is even and $n|(p_i \cdot Q(\alpha^\vee))$, or $ \sum_{q_j: \ q_j < p_i} e_j$ is odd and $n/\gcd(n, p_i\cdot Q(\alpha^\vee))=2$.
It follows that the orbit $\mca{O}$ is raisable if one of the following holds:
\begin{enumerate}
\item[--] $n\nmid p_i$ or $\mfr{B}(\mfr{p}, p_i) \notin 2\Z$, if $n$ is odd;
\item[--] $n\nmid p_i$, if $n$ is even and $\mfr{B}(\mfr{p}, p_i) \in 2\Z$;
\item[--] $\gcd(n, p_i)\ne n/2$, if $n$ is even and $\mfr{B}(\mfr{p}, p_i) \notin 2\Z$.
\end{enumerate}
\item[$\bullet$] (B/D-Ort) In this case, since $d_i$ is always even we get $m\in 2\Z$. Here $\SO_{e_i}$ has Brylinski--Deligne invariant  equal to
$2q_j$. Thus, the covering $\wt{\SL}_2^{(n), \tau}$ has Brylinski--Deligne invariant $Q_1^\tau=2q_j$ as well. We see that $\mfr{m}_*(\wt{\SL}_2^{(n,2)})$ splits if and only if $n|2q_j$. That is, the orbit $\mca{O}$ is raisable if $n\nmid (2q_j)$.
\end{enumerate}

\begin{thm} \label{T:tyBD}
Consider the orthogonal orbit $\mca{O}=(p_1^{d_1} \cdots p_i^{d_i}\cdots p_k^{d_k} q_1^{e_1} \cdots q_j^{e_j}\cdots q_m^{e_m})$ for the orthogonal group $G_r=\SO_{2r}$ or $\SO_{2r+1}$ with $p_i$ even and $q_j$ odd.
\begin{enumerate}
\item[(i)] It is $\overline{G}_r^{(n)}$-quasi-admissible if and only if for every $i, j$ with $d_i \gest 2$ and $e_j\gest 3$ the following hold:
\begin{enumerate}
\item[--] $n| p_i$ and $\mfr{B}(\mfr{p}_\mca{O}, p_i) \in 2\Z$, if $n$ is odd;
\item[--] $n| p_i$, if $n$ is even and $\mfr{B}(\mfr{p}_\mca{O}, p_i) \in 2\Z$;
\item[--] $\gcd(n, p_i)=n/2$, if $n$ is even and $\mfr{B}(\mfr{p}_\mca{O}, p_i) \notin 2\Z$;
\item[--] $n|q_j$, if $n$ is odd;
 \item[--] $n|(2q_j)$, if $n$ is even and $e_j\gest 4$;
 \item[--] $n|(4q_j)$, if $n$ is even and $e_j=3$.
\end{enumerate}
\item[(ii)] It is $\overline{G}_r^{(n)}$-raisable if 
for some $i, j$ with $d_i \gest 2$ and $e_j\gest 4$
one of the following holds:
\begin{enumerate}
\item[--] $n\nmid p_i$ or $\mfr{B}(\mfr{p}, p_i) \notin 2\Z$, if $n$ is odd;
\item[--] $n\nmid p_i$, if $n$ is even and $\mfr{B}(\mfr{p}, p_i) \in 2\Z$;
\item[--] $\gcd(n, p_i)\ne n/2$, if $n$ is even and $\mfr{B}(\mfr{p}, p_i) \notin 2\Z$;
\item[--] $n \nmid (2q_j)$.
\end{enumerate}
\end{enumerate}
\end{thm}


\subsection{Type C}
Consider the Dynkin diagram for the simple roots of type $C_r$:

$$ \qquad 
\begin{picture}(4.7,0.2)(0,0)
\put(1,0){\circle{0.08}}
\put(1.5,0){\circle{0.08}}
\put(2,0){\circle{0.08}}
\put(2.5,0){\circle{0.08}}
\put(3,0){\circle{0.08}}
\put(1.04,0){\line(1,0){0.42}}
\multiput(1.55,0)(0.05,0){9}{\circle*{0.02}}
\put(2.04,0){\line(1,0){0.42}}
\put(2.54,0.015){\line(1,0){0.42}}
\put(2.54,-0.015){\line(1,0){0.42}}
\put(2.74,-0.04){$<$}
\put(1,0.1){\footnotesize $\alpha_1$}
\put(1.5,0.1){\footnotesize $\alpha_2$}
\put(2,0.1){\footnotesize $\alpha_{r-2}$}
\put(2.5,0.1){\footnotesize $\alpha_{r-1}$}
\put(3,0.1){\footnotesize $\alpha_r$}
\end{picture}
$$
\vskip 10pt

We consider the group $G=\Sp_{2r}$ and a nilpotent orbit
$$\mca{O}=(p_1^{d_1} \cdots p_i^{d_i}\cdots p_k^{d_k} q_1^{e_1} \cdots q_j^{e_j} q_m^{e_m}),$$
where $p_i$ is even and $q_j$ is odd. Here, $e_j$ is necessarily even.
To discuss about the quasi-admissibility, we first note that
$$G_{\gamma, 0} = \left( \prod_{i=1}^k \SO_{d_i}^{\Delta, p_i} \right) \times \left( \prod_{j=1}^m \Sp_{e_j}^{\Delta, q_j} \right).$$
For each pair $(p_i, q_j)$, we have
$$\mfr{g}_{p_i, q_j}[1] \simeq \min \set{p_i, q_j} \cdot ( \C^{d_i} \otimes \C^{e_j} ),$$
where $\C^{d_i}$ affords the natural action of $\SO_{d_i}$, and $\C^{e_j}$ that of $\Sp_{e_j}$. Also, 
$$\mfr{g}[1] = \bigoplus_{(p_i, q_j)} \mfr{g}_{p_i, q_j}[1],$$
where $G_{\gamma, 0}$ acts on each component as given above. Consider the cover $\wt{\Sp}_{2r}^{(n)}$  obtained from the restriction of $\wt{\SL}_{2r}^{(n)}$ via inclusion $\Sp_{2r} \subset \SL_{2r}$, one has $\BDI(\wt{\Sp}_{2r}^{(n)})= Q(\alpha_r^\vee) = Q(\alpha^\vee)=1$, where $\alpha^\vee$ is any coroot of ${\SL}_{2r}$. We have
\begin{enumerate}
\item[--] the Brylinski--Deligne invariant for $\wt{\Sp}_{e_j}^{(2), \phi}$ is 
$$\sum_i \min \set{p_i, q_j} d_i = \left( \sum_{i:\ p_i < q_j} p_i d_i \right) + q_j \left(\sum_{i:\ p_i > q_j} d_i\right);$$
\item[--] the Brylinski--Deligne invariant for $\wt{\SO}_{d_i}^{(2), \phi}, d_i \gest 4$ is 
$$\sum_j 2 \min \set{p_i, q_j} d_j = 2p_i \left( \sum_{j:\ q_j > p_i} e_j \right) + 2 \left( \sum_{j:\ p_i > q_j} q_j e_j\right).$$
\end{enumerate}
Similarly, we have
\begin{enumerate}
\item[--] the Brylinski--Deligne invariant for $\wt{\Sp}_{e_j}^{(n)}$ is $q_j$;
\item[--] the Brylinski--Deligne invariant for $\wt{\SO}_{d_i}^{(n)}, d_i\gest 4$ is $2p_i$, and it is $4p_i$ for $\wt{\SO}_{3}^{(n)}$.
\end{enumerate}
Again, for quasi-admissibility it suffices to consider the splitting of the $n^*$-fold cover of $\mfr{m}_*(\wt{\Sp}_{e_j}^{(n, 2)})$ and $\mfr{m}_*(\wt{\SO}_{d_i}^{(n, 2)})$ over $\Sp_{e_j}$ and $\SO_{d_i}$ respectively. The $n^*$-fold cover $\mfr{m}_*(\wt{\Sp}_{e_j}^{(n, 2)})$ has Brylinski--Deligne invariant
$$(n^*q_j/n) + \left( \sum_i \min \set{p_i, q_j} d_i\right) n^*/2$$
and $\mfr{m}_*(\wt{\SO}_{d_i}^{(n, 2)})$, $d_i \gest 4$ has Brylinski--Deligne invariant
$$(2n^*p_i/n) + n^*\left(\sum_j \min \set{p_i, q_j} e_j\right).$$
Also, $\mfr{m}_*(\wt{\SO}_{3}^{(n, 2)})$ has Brylinski--Deligne invariant
$$(4n^*p_i/n) + 2n^*\left(\sum_j \min \set{p_i, q_j} e_j\right).$$
Here $\mfr{m}_*(\wt{\Sp}_{e_j}^{(n, 2)}), e_j \gest 2$ splits over $\Sp_{e_j}$ if and only if the following hold:
\begin{enumerate}
\item[--] $n| q_j$ and $\mfr{A}(\mfr{p}_\mca{O}, q_j) \in 2\Z$, if $n$ is odd;
\item[--] $n| q_j$, if $n$ is even and $\mfr{A}(\mfr{p}_\mca{O}, q_j) \in 2\Z$;
\item[--] $\gcd(n, q_j)=n/2$, if $n$ is even and $\mfr{A}(\mfr{p}_\mca{O}, q_j) \notin 2\Z$.
\end{enumerate}
 On the other hand, $\wt{\SO}_{d_i}^{(n, 2)}, d_i\gest 3$, has a finite dimensional $(n, 2)$ representation if the following hold:
 \begin{enumerate}
 \item[--] $n|p_i$, if $n$ is odd;
 \item[--] $n|(2p_i)$, if $n$ is even and $d_i\gest 4$;
 \item[--] $n|(4p_i)$, if $n$ is even and $d_i=3$.
 \end{enumerate}

Now, regarding raisability, we again have  the two cases:
\begin{enumerate}
\item[$\bullet$] (C-Sym) Suppose there exists $e_j \gest 2$. Let $\tau: \SL_2 \into \Sp_{e_j}$ be the embedding corresponding to the long simple root of $\Sp_{e_j}$. Then this $\tau$ satisfies (C1)--(C3) where $\mfr{g}[1]=\mfr{g}[1]^{\mfr{sl}_{2,\tau}} \oplus m V_2$ with
$$m=q_j \cdot \left(\sum_{p_i: \ p_i > q_j} d_i \right) + \sum_{p_i: \ p_i < q_j} p_i \cdot d_i = q_j \cdot \mfr{A}(\mfr{p}, q_j) + \sum_{p_i: \ p_i < q_j} p_i \cdot d_i .$$
We have $Q_2^\tau=m$ and  
$$Q_1^\tau=q_j.$$
 It follows from  Proposition \ref{P:rais} that 
 $$\mfr{m}_*(\wt{\SL}_2^{(n,2)}) \onto \SL_2$$
 split if and only if $\mfr{A}(\mfr{p}, q_j)$ is even and $n|q_i$, or $\mfr{A}(\mfr{p}, q_j)$ is odd and $\gcd(n, q_j)=2$.
This shows that the orbit $\mca{O}$ is raisable if one of the following holds:
\begin{enumerate}
\item[--] $n\nmid q_j$ or $\mfr{A}(\mfr{p},q_j) \notin 2\Z$, if $n$ is odd;
\item[--] $n\nmid q_j$, if $n$ is even and $\mfr{A}(\mfr{p}, q_j) \in 2\Z$;
\item[--] $\gcd(n, q_j)\ne n/2$, if $n$ is even and $\mfr{A}(\mfr{p}, q_j) \notin 2\Z$.
\end{enumerate}
\item[$\bullet$] (C-Ort) Suppose there exists $d_i \gest 4$ and thus a Levi subgroup 
$$\GL_2 \times \SO_k \subset \SO_{d_i}.$$
Let $\tau: \SL_2 \into \SO_{d_i}$ be associated with the root of $\GL_2$. It satisfies (C1)--(C3) and
$\mfr{g}[1]=\mfr{g}[1]^{\mfr{sl}_{2,\tau}} \oplus m V_2$ where
$$m=p_i \cdot \left( \sum_{q_j:\ q_j > p_i} e_j \right) + \sum_{q_j: \ q_j < p_i} q_j \cdot e_j.$$
In this case, since $e_j$ is always even we get $m\in 2\Z$. Here $\SO_{d_i}$ has Brylinski--Deligne invariant 
$2p_i$. Thus, the covering $\wt{\SL}_2^{(n), \tau}$ has Brylinski--Deligne invariant $Q_1^\tau=2p_i$ as well. We see that $\mfr{m}_*(\wt{\SL}_2^{(n,2)})$ splits if and only if $n|(2p_i)$. That is, the orbit $\mca{O}$ is raisable if $n\nmid (2p_i)$.
\end{enumerate}

\begin{thm} \label{T:tyC}
Let $\mca{O}=(p_1^{d_1} \cdots p_i^{d_i}\cdots p_k^{d_k} q_1^{e_1} \cdots q_j^{e_j}\cdots q_m^{e_m})$ be an $F$-split symplectic orbit for $\Sp_{2r}$ with $p_i$ even and $q_j$ odd. Consider the $n$-fold cover  $\wt{\Sp}_{2r}^{(n)}$ of $\Sp_{2r}$ with Brylinski--Deligne invariant equal to $Q(\alpha^\vee_r)=1$.
\begin{enumerate}
\item[(i)] The orbit $\mca{O}$ is $\wt{\Sp}_{2r}^{(n)}$-quasi-admissible if and only if for every $i, j$ with $d_i \gest 3$ and $e_j\gest 2$ the following hold:
\begin{enumerate}
\item[--] $n| q_j$ and $\mfr{A}(\mfr{p}_\mca{O}, q_j) \in 2\Z$, if $n$ is odd;
\item[--] $n| q_j$, if $n$ is even and $\mfr{A}(\mfr{p}_\mca{O}, q_j) \in 2\Z$;
\item[--] $\gcd(n, q_j)=n/2$, if $n$ is even and $\mfr{A}(\mfr{p}_\mca{O}, q_j) \notin 2\Z$;
\item[--] $n|p_i$, if $n$ is odd;
 \item[--] $n|(2p_i)$, if $n$ is even and $d_i\gest 4$;
 \item[--] $n|(4p_i)$, if $n$ is even and $d_i=3$.
\end{enumerate}
\item[(ii)] The orbit $\mca{O}$ is $\wt{\Sp}_{2r}^{(n)}$-raisable if 
for some $i, j$ with $e_j \gest 2$ and $d_i\gest 4$
one of the following holds:
\begin{enumerate}
\item[--] $n\nmid q_j$ or $\mfr{A}(\mfr{p}, q_j) \notin 2\Z$, if $n$ is odd;
\item[--] $n\nmid q_j$, if $n$ is even and $\mfr{A}(\mfr{p}, q_j) \in 2\Z$;
\item[--] $\gcd(n, q_j)\ne n/2$, if $n$ is even and $\mfr{A}(\mfr{p}, q_j) \notin 2\Z$;
\item[--] $n \nmid (2p_i)$.
\end{enumerate}
\end{enumerate}
\end{thm}

For a partition $\mfr{p}$ and $\sharp \in \set{B, C, D}$, we denote by
$$\mfr{p}_\sharp, \quad \mfr{p}^\sharp$$
the type $\sharp$-collapse and $\sharp$-expansion of $\mfr{p}$ respectively, see \cite[\S 6.3]{CM}.

\begin{eg} \label{E:Sp-orb}
We consider three special families of orbits of $\Sp_{2r}$ and their quasi-admissibility. First, assume $n\in \N_{\gest 1}$ is odd with $2r=a n + b$. Consider the orbit 
$$\mca{O}^{2r, n}_C=(n^a b)_C = 
\begin{cases}
(n^ab) & \text{ if $a$ is even (and $b$ even)}, \\
(n^{a-1}, n-1, b+1) & \text{ if $a$ is odd}.
\end{cases}$$
It is is $\wt{\Sp}_{2r}^{(n)}$-quasi-admissible by Theorem \ref{T:tyC}, and is not raisable. Second, if $n=2k$ with $k$ odd and we write $2r=ka + b$ with $0\lest b < k$, then we have 
$$ (k+1, \mca{O}^{2r-k-1, k})_{C} = 
\begin{cases}
(k+1, k^a, b) & \text{ if $a$ is even}, \\
(k+1, k^{a-1}, k-1, b+1) & \text{ if $a$ is odd}. 
\end{cases}
$$
We see that this orbit is quasi-admissible in this case. If $a$ is even, then the orbit  is not raisable, since $n/\gcd(n, q_j\cdot Q(\alpha^\vee))=2$ in this case.
 As a last example, consider $n \in 4\Z$, then the orbit $\mca{O}^{2r, n/2}_C = \mca{O}^{2r, n/2} = ((n/2)^a b)$,
where $a, b$ are both even. This orbit is quasi-admissible. In this case, $p_i=n/2$ and clearly $n|(2p_i\cdot Q(\alpha^\vee))$. Thus, the orbit is not raisable as well.

These orbits are the speculated wavefront sets of the theta representation of $\wt{\Sp}_{2r}^{(n)}$, see \S \ref{SS:WFcl} for more details.
\end{eg}

In the remaining of this section, we consider exceptional groups and discuss in more details the case of $G_2$ and $F_4$. The computation of the types $E_i, 6\lest i \lest 8$ follows from the same techniques, the details of which however are more involved. Thus, for $E_i$ we will only discuss about the quasi-admissibility and raisability of certain nilpotent orbits, which are the speculated stable wavefront set of the theta representations.

For each orbit written in the Bala--Carter notations, we give information on whether it is special or even, and the structure of $\mbf{G}_{\gamma, der}$, which can be obtained from \cites{CM, Car, JaNo05}. By computing the Brylinski--Deligne invariants $(Q_1, Q_2)$ for each of the simple subgroup of $\mbf{G}_{\gamma, der}$, we obtain a criterion of quasi-admissibility on $n$. Similarly, by computing $(Q_1^\tau, Q_2^\tau)$ for properly chosen $\tau$, we can give condition on $n$ such that the orbit is $\wt{G}^{(n)}$-raisable. Our computations rely on the extensive results in \cite{JaNo05}, some of which were already used in \cite{JLS}.

For simplicity, in the remaining of this section we consider almost simple simply-connected exceptional group $G$ and its cover $\wt{G}^{(n)}$ with
$$\BDI(\wt{G}^{(n)})=1.$$
We use $V^k$ to denote certain irreducible $k$-dimensional representation of the underlying group, which is clear from the context.

\subsection{Type $G_2$} 
Consider the Dynkin diagram of simple roots of $G_2$:
$$
\begin{picture}(5.2,0.2)(0,0)
\put(2.5,0){\circle{0.08}}
\put(3,0){\circle{0.08}}
\put(2.53,0.018){\line(1,0){0.44}}
\put(2.54,0){\line(1,0){0.42}}
\put(2.53,-0.018){\line(1,0){0.44}}
\put(2.7,-0.04){$<$}
\put(2.5,0.1){\footnotesize $\alpha_1$}
\put(3,0.1){\footnotesize $\alpha_2$}
\end{picture}
$$
\vskip 10pt

The results are summarized in Table \ref{table 1}.
\begin{table}[H] 
\caption{Nilpotent orbits for $G_2$}
\label{table 1}
\vskip 5pt
\renewcommand{\arraystretch}{1.3}
\begin{tabular}{|c|c|c|c|c|c|}
\hline
$\mca{O}$ & special/even?  &    $\mbf{G}_{\gamma, der}$  & $(Q_1, Q_2)$ &  quasi-adm., iff  & raisable, if    \\
\hline
\hline
$\set{0}$   & yes/yes  &  $G_2$ & (1, 0)  &  $n =1 $ & $n\gest 2$      \\
\hline
 $A_1$ & no/no & $\SL_{2,\alpha_1}$ & (3, 4) &  $n = 1, 3$ & n.a.    \\
\hline
  $\tilde{A}_1$  & no/no  & $\SL_{2, \alpha_2}$ & (1, 1) &  $n=  2$ & $n\neq 2$   \\
  \hline
$G_2(a_1)$  & yes/yes  & 1 & n.a. &  all $n$  & n.a.   \\
\hline
$G_2$ & yes/yes & 1  & n.a. &  all $n$ & n.a.    \\
\hline
\end{tabular}
\end{table}

The entries for the three orbits $\set{0}, G_2(a_1)$ and $G_2$ are clear. Thus, we give a brief explanation for the two orbits $A_1$ and $\tilde{A}_1$.

The minimal orbit $A_1$ gives $G_{\gamma, der} = \SL_{2}$ associated with $\alpha_1$. One has
$$\mfr{g}[1] \simeq V^4,$$
which then gives $Q_2 = 4$. We also have $Q_1 =3$. Thus, the orbit is $\wt{G}_2^{(n)}$-quasi-admissible for $n=1, 3$. The method of raisability in \cite{JLS} does not apply.

The orbit $\tilde{A}_1$ gives $G_{\gamma, der} = \SL_2$ associated with the long root $\alpha_2$, which immediately gives $Q_1 =1$. On the other hand,
$$\mfr{g}[1] \simeq V^2$$
and thus $Q_2 =1$. This shows that the orbit $\tilde{A}_1$ is quasi-admissible if and only if $n=2$.
The map $\tau={\rm id}$ satisfies (C1)--(C3). If $n\ne 2$, then $\mfr{m}_*(\wt{\SL}_{2, \tau}^{(n,2)}) \onto \SL_2$ does not split and thus $\tilde{A}_1$ is raisable.

\subsection{Type $F_4$}  \label{SS:F4}
Consider the Dynkin diagram of simple roots of $F_4$ as follows

$$
\begin{picture}(4.7,0.2)(0,0)
\put(2,0){\circle{0.08}}
\put(2.5,0){\circle{0.08}}
\put(3,0){\circle{0.08}}
\put(3.5,0){\circle{0.08}}
\put(2.04,0){\line(1,0){0.42}}
\put(2.54,0.015){\line(1,0){0.42}}
\put(2.54,-0.015){\line(1,0){0.42}}
\put(2.72,-0.04){$>$}
\put(3.04,0){\line(1,0){0.42}}
\put(2,0.1){\footnotesize $\alpha_1$}
\put(2.5,0.1){\footnotesize $\alpha_2$}
\put(3,0.1){\footnotesize $\alpha_3$}
\put(3.5,0.1){\footnotesize $\alpha_4$}
\end{picture}
$$
\vskip 10pt

The results are given in Table \ref{table 2}.

\begin{table}[H]  
\caption{Nilpotent orbits for $F_4$}
\label{table 2}
\vskip 5pt
\renewcommand{\arraystretch}{1.3}
\begin{tabular}{|c|c|c|c|c|c|}
\hline
$\mca{O}$ & special/even?  &     $\mbf{G}_{\gamma, der}$  &  $(Q_1, Q_2)$ &   quasi-adm., iff & raisable, if     \\
\hline
\hline
$\set{0}$   & yes/yes & $F_4$   & (1, 0)   & $n=1$ & $n\gest 2$    \\
\hline
$A_1$ & no/no & $\Sp_{6}$  & (1, 5)  & $n= 2$ & $n\ne 2$   \\
\hline
$\tilde{A}_1$ & yes/no & $\SL_{4}$  & (1, 2) & $n =1$ & $n\gest 2$   \\
\hline
$A_1 + \tilde{A}_1$ & yes/no  & $\SL_{2,\alpha_1} \times \SO_3$ & (1, 6), (8, 20) & $n=1$ & $n\gest 2$   \\
\hline
$A_2$ & yes/yes & $\SL_{3}$ & (2, 0) & $n=1, 2$ & $n\gest 3$   \\
\hline
$\tilde{A}_2$ & yes/yes & $\SL_{3}$ & (1, 0) & $n=1$ & $n\gest 2$   \\
\hline
$A_2 + \tilde{A}_1$ & no/no & $\SL_2$ & $(6, 11)$  &  $n=4, 12$  & n.a.  \\
\hline
$B_2$ & no/no  & $\SL_{2, \alpha_2} \times \SL_2$ & $(1, 1), (1, 1)$ &  $n=2$  & $n\ne 2$   \\
\hline
$\tilde{A}_2 + A_1$ & no/no & $\SL_2$  & $(3, 12)$  &  $n=1, 3$ &  n.a.   \\
\hline
$C_3(a_1)$ & no/no & $\SL_{2, \alpha_2}$  & $(1, 3)$ &  $n=2$  & $n\ne 2$  \\
\hline
$F_4(a_3)$ & yes/yes & 1  & n.a.  &  all $n$ & n.a.  \\
\hline
$B_3$ & yes/yes & $\SO_3$  & $(8, 0)$  &  $n= 1, 2, 4, 8$ & $n\ne 1, 2, 4, 8$  \\
\hline
$C_3$ & yes/no & $\SL_{2, \alpha_2}$  & $(1, 2)$  &  $n=1$ & $n\gest 2$  \\
\hline
$F_4(a_2)$ & yes/yes & 1  & n.a.  &  all $n$  & n.a.  \\
\hline
$F_4(a_1)$ & yes/yes & 1  & n.a.  &  all $n$ & n.a.   \\
\hline
$F_4$ & yes/yes & 1  & n.a.  &  all $n$ & n.a.   \\
\hline
\end{tabular}
\end{table}

Again, the method for quasi-admissibility and raisability does not apply to distinguished orbits. Thus, it suffices to consider the following orbits
$$A_1,\ \tilde{A}_1, \ A_1 + \tilde{A}_1,\  A_2, \ \tilde{A}_2, \ A_2 + \tilde{A}_1,\  B_2, \ \tilde{A}_2 +A_1, \ C_3(a_1),\  B_3,\  C_3.$$
We give a case by case discussion.

The orbit $A_1$ has 
$$G_{\gamma, der} =\Sp_6$$
associated with $\set{\alpha_2, \alpha_3, \alpha_4}$. One has 
$$\mfr{g}[1]=\extp^3 V_{\rm std} / V_{\rm std}$$
as a $G_{\gamma, der}$-module and is of dimension 14, where $V_{\rm std}$ represents the  standard representation of $\Sp_6$. We have $\mfr{g}[1] = 4V^1 \oplus 5\cdot V^2$ as an $\SL_{2, \alpha_2}$-module. This gives that $Q_2=5$. Thus, it follows from Propositions \ref{P:qadm} and \ref{P:exp-spl} that the orbit  $A_1$ is $\wt{F}_4^{(n)}$-quasi-admissible if and only if $n=2$. For raisability, we take 
$$\mfr{sl}_{2, \tau} =  \mfr{sl}_{2, \alpha_2}$$
be associated with $\alpha_2$. Then (C1)--(C3) are satisfied with $m=5$. This shows that if $2\nmid n$,  then $A_1$ is $\wt{F}_4^{(n)}$-raisable by Proposition \ref{P:rais}.

The orbit $\tilde{A}_1$ gives
$$G_{\gamma, der} = \SL_4$$
associated with $\set{\alpha_1, \alpha_2, \alpha_2 + 2\alpha_3}$.  One has
$$\mfr{g}[1] = V_{\rm std} \oplus V_{\rm std}^*,$$
where $V_{\rm std}$ is the standard representation of $\SL_4$. This gives that $Q_2 =2$. Since $Q_1 =1$, we see that 
$\tilde{A}_1$ is $\wt{F}_4^{(n)}$-quasi-admissible if and only if $n=1$. For raisability, consider the $\tau: \SL_2 \into G_{\gamma, der}$ associated to the simple root $\alpha_2$ of $G_{\gamma, der}$, then (C1)--(C3) are satisfied with $Q_1 = 1$ and $m=Q_2^\tau= 2$. In this case, we see that $\tilde{A}_1$ is raisable if $n\gest 2$.

The orbit $A_1 + \tilde{A}_1$ gives
$$\mbf{G}_{\gamma, der} = \SL_2 \times \SO_3,$$
where $\SL_2$ is associated with $\alpha_1$ and $ \SO_3= \PGL_2$ is embedded in $\SL_3$ associated with $\set{\alpha_3, \alpha_4}$. 
As a $G_{\gamma, der}$-module, one has
$$\mfr{g}[1] = (V^2 \boxtimes V^5) \oplus (V^2 \boxtimes V^1).$$
Thus, we have $Q_1 = 1$ for $\SL_2$ and $Q_1(\alpha_3^\vee + \alpha_4^\vee) = 8$ for $\SO_3$. Also, $Q_2 \in 2\Z$ for both $\SL_2$ and $\SO_3$. Thus, $\mfr{m}_*(\wt{\SL}_{2,\tau}^{(n,2)})$ splits over $\SL_2$ if and only if $n=1$; also, $\mfr{m}_*(\wt{\SO}_{3, der}^{(n,2)})$ splits over $\SO_{3, der}$ if and only if $n=1, 2, 4$ or $8$. Thus, the orbit $A_1 + \tilde{A}_1$ is quasi-admissible if and only if $n=1$. For raisability, consider the embedding 
$$\tau = {\rm id} \times 1: \SL_2 \into G_{\gamma, der},$$
which satisfies (C0)--(C3) with $Q_2^\tau = m=6$. Since $Q_1^\tau = 1$. We see that if $n\gest 2$, then $A_1 + \tilde{A}_1$ is raisable.

The orbits $A_2$ and $\tilde{A}_2$ are even, and thus $Q_2=0$ for both of them. For $A_2$ one has $G_{\gamma, der} = \SL_3$ associated with $\alpha_3, \alpha_4$. For $\tilde{A}_2$ one has $G_{\gamma, der} = \SL_3$ associated with $\alpha_1, \alpha_2$. We get $Q_1 =2$ for $A_2$ and $Q_1 = 1$ for $\tilde{A}_2$.
This gives the criterion for quasi-admissibility. For $A_2$, let 
$$\tau: \SL_2 \into G_{\gamma, der}$$
be the embedding associated with $\alpha_3$, which gives $Q_1^\tau=2$, and thus we see that if $n\gest 3$, then the orbit is raisable. Similarly, if $n\gest 2$, then the orbit $\tilde{A}_2$ is raisable.

The orbit $A_2 + \tilde{A}_1$ has $G_{\gamma, der} =\nabla(\SL_2)$, where $\nabla = {\rm Sym}^2 \times {\rm id}: \SL_2 \into \SL_{3, \alpha_1, \alpha_2} \times \SL_{2, \alpha_4}$. In this case,
$$\mfr{g}[1] = V^2 \oplus V^4$$
as a $G_{\gamma, der}$-module. We have
$$(Q_1, Q_2) = (6, 11).$$
Thus, the orbit is quasi-admissible if $n=4$ or $12$. For raisability, the method in \cite{JLS} does not apply. 

The orbit $B_2$ has
$$G_{\gamma, der} \simeq \SL_{2,\alpha_2} \times \SL_{2} \subset \Sp_4,$$
where $\Sp_4$ is associated with $\alpha_2, \alpha_3$. Also 
$$\mfr{g}[1] = (V^2\boxtimes V^1)  \oplus (V^1 \boxtimes V^2).$$
For every copy of $\SL_2$ in $G_{\gamma, 0}$ one has $Q_1 =1$. Also, $Q_2 = 1$. Thus, the orbit $B_2$ is quasi-admissible if and only if $n=2$. On the other hand, if we take $\SL_{2,\tau}$ to be associated with $\alpha_2$, then (C1)--(C3) are satisfied with $m=1$. In this case, the orbit $B_2$ is raisable if $n\ne 2$.

The orbit $\tilde{A}_2 + A_1$ gives 
$$G_{\gamma, der} = \Delta (\SL_2) \subset \SL_{2, \alpha_1} \times \SL_{2, \alpha_3}$$
and
$$\mfr{g}[1] = 2V^2 \oplus V^4.$$
This gives that $(Q_1, Q_2) = (3, 12)$. Thus, the orbit is quasi-admissible if and only if $n=1, 3$. On the other hand, the method of raisability does not apply.

The orbit $C_3(a_1)$ has $G_{\gamma, der} \simeq \SL_{2, \alpha_2}$ and 
$$\mfr{g}[1] = 3V^2.$$
Thus, $(Q_1, Q_2) =(1, 3)$ and the orbit is quasi-admissible if and only if $n=2$. For raisability, taking $\tau = {\rm id}$ shows that  the orbit is raisable if $n \ne 2$.

The orbits $B_3$ gives 
$$\mbf{G}_{\gamma, der} = \SO_3 \subset \SL_3$$ associated with $\alpha_3, \alpha_4$, and $\mfr{g}[1]= 0$. 
Thus, one has 
$$(Q_1, Q_2) = (8, 0).$$
This shows that the orbit is quasi-admissible if and only if $n=1, 2, 4, 8$. Taking $\tau: \SL_2 \to \SO_3 = \PGL_2$ to be the natural map, one sees that the orbit $B_3$ is raisable if $n\ne 1, 2, 4, 8$.

The orbit $C_3$ has $G_{\gamma, der} = \SL_{2, \alpha_2}$ and 
$$\mfr{g}[1] = 2 V^2.$$
This gives that $(Q_1, Q_2) =(1, 2)$. Thus, the orbit is quasi-admissible if and only if $n=1$. Also, taking $\tau = {\rm id}$ shows that it is raisable if $n\gest 2$.

\subsection{Certain orbits for type $E_r, 6\lest r \lest 8$} 
Now we consider several specific non-distinguished orbits for each exceptional group $E_r$ and investigate their $\wt{E}_r^{(n)}$-quasi-admissibility and $\wt{E}_r^{(n)}$-raisability. The consideration of these orbits is motivated from theta representations $\Theta(\wt{E}_r^{(n)})$, since they are expected to be equal to $\mca{N}_\Wh^{\rm max}(\Theta(\wt{E}_r^{(n)})) \otimes F^{\rm al}$ for some $n$.

More precisely, we consider orbits as follows:
\begin{enumerate}
\item[$\bullet$] for $E_6$ the orbits
$$3A_1, 2A_2 + A_1, D_4, A_4 + A_1, D_5;$$
\item[$\bullet$] for $E_7$ the orbits
$$4A_1, 2A_2 + A_1, A_3 + A_2 + A_1, A_4 + A_2, A_6, E_6(a_1);$$
\item[$\bullet$] for $E_8$ the orbits
$$4A_1, 2A_2 + 2A_1, 2A_3, A_4 + A_3, A_6 + A_1, A_7.$$
\end{enumerate}

\subsubsection{Type $E_6$}
Consider the Dynkin diagram for the simple roots of $G=E_6$, following Bourbaki's labelling \cite{BouL2}.

$$
\begin{picture}(4.7,0.2)(0,0)
\put(1,0){\circle{0.08}}
\put(1.5,0){\circle{0.08}}
\put(2,0){\circle{0.08}}
\put(2.5,0){\circle{0.08}}
\put(3,0){\circle{0.08}}
\put(2,-0.5){\circle{0.08}}
\put(1.04,0){\line(1,0){0.42}}
\put(1.54,0){\line(1,0){0.42}}
\put(2.04,0){\line(1,0){0.42}}
\put(2,-0.04){\line(0,-1){0.42}}
\put(2.54,0){\line(1,0){0.42}}
%
\put(1,0.1){\footnotesize $\alpha_1$}
\put(1.5,0.1){\footnotesize $\alpha_3$}
\put(2,0.1){\footnotesize $\alpha_4$}
\put(2.5,0.1){\footnotesize $\alpha_5$}
\put(3,0.1){\footnotesize $\alpha_6$}
\put(2.08,-0.55){\footnotesize $\alpha_2$}
\end{picture}
$$
\vskip 45pt
The results are given in Table \ref{table 3}.

\begin{table}[H]  
\caption{Some nilpotent orbits for $\wt{E}_6^{(n)}$}
\label{table 3}
\vskip 5pt
\renewcommand{\arraystretch}{1.3}
\begin{tabular}{|c|c|c|c|c|c|c|c|}
\hline
$\mca{O}$ & special/even?  &     $\mbf{G}_{\gamma, der}$  &  $(Q_1, Q_2)$ &  quasi-adm., iff & raisable, if     \\
\hline
\hline
 $3A_1$   & no/no & $\SL_{2, \alpha_2} \times \SL_3$  & $(1, 9), (2, 12)$     &  $n=2$  & $n\ne 2$   \\
\hline
 $2A_2+ A_1$ & no/no & $\SL_2$  & $(3, 14)$  &  $n=1, 3$  & n.a.  \\
\hline
 $D_4(a_1)$ & yes/yes & 1 & n.a. &  all $n$ & n.a.  \\
\hline
 $A_4 + A_1$ & yes/no  & 1 & n.a. &  all $n$ & n.a.   \\
\hline
 $D_5$ & yes/yes & 1 & n.a. &  all $n$ & n.a.   \\
\hline
\end{tabular}
\end{table}

The orbit $3A_1$ has
$$G_{\gamma, der} \simeq \SL_{2, \alpha_2} \times \Delta(\SL_3),$$
where $\Delta: \SL_3 \into \SL_{3, \alpha_1, \alpha_3} \times \SL_{3, \alpha_5, \alpha_6}$ is the diagonal embedding. Also
$$\mfr{g}[1] = V^2 \boxtimes (V^1 \oplus V_{\rm adj}).$$
Thus, for $\SL_{2, \alpha_2}$ we have $(Q_1, Q_2) =(1, 9)$, and for $\Delta(\SL_3)$ one has $(Q_1, Q_2) =(2, 12)$. Hence, $\mfr{m}_*(\wt{\SL}_{2, \alpha_2}^{(n, 2)})$ splits over $\SL_{2, \alpha_2}$ if and only if $n=2$. On the other hand, the $n^*$-fold cover of $\Delta(\SL_3)$ splits if and only if $n=1, 2$. This shows that the orbit is quasi-admissible if and only if $n=2$. For raisability, the $\tau$ such that $\tau(\SL_2) = \SL_{2, \alpha_2}$ satisfies (C1)--(C3) with $Q_2^\tau = m=9$. We see  that the orbit is raisable if $n\ne 2$.

The orbit $2A_2 + A_1$ gives $G_{\gamma, der} = \SL_2$, diagonally embedded into $\SL_{2, \alpha_2} \times \SL_{2, \alpha_3} \times \SL_{2, \alpha_5}$. One has 
$$\mfr{g}[1] = 4V^2 \oplus V^4.$$
This gives $(Q_1, Q_2) = (3, 14)$. Thus, the orbit is quasi-admissible if and only if $n=1, 3$. For raisability the method in \cite{JLS} does not apply.


The orbit $D_4(a_1), A_4 + A_1$ and $D_5$ all give that $G_{\gamma, der} = 1$.

\subsubsection{Type $E_7$}
Consider the Dynkin diagram for the simple roots of $E_7$:

$$
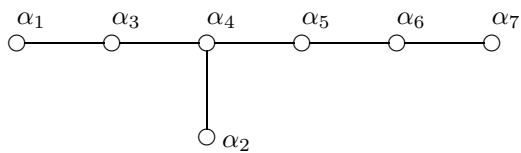
\begin{picture}(4.7,0.2)(0,0)
\put(1,0){\circle{0.08}}
\put(1.5,0){\circle{0.08}}
\put(2,0){\circle{0.08}}
\put(2.5,0){\circle{0.08}}
\put(3,0){\circle{0.08}}
\put(3.5,0){\circle{0.08}}
\put(2,-0.5){\circle{0.08}}
\put(1.04,0){\line(1,0){0.42}}
\put(1.54,0){\line(1,0){0.42}}
\put(2.04,0){\line(1,0){0.42}}
\put(2,-0.04){\line(0,-1){0.42}}
\put(2.54,0){\line(1,0){0.42}}
\put(3.04,0){\line(1,0){0.42}}
\put(1,0.1){\footnotesize $\alpha_1$}
\put(1.5,0.1){\footnotesize $\alpha_3$}
\put(2,0.1){\footnotesize $\alpha_4$}
\put(2.5,0.1){\footnotesize $\alpha_5$}
\put(3,0.1){\footnotesize $\alpha_6$}
\put(3.5,0.1){\footnotesize $\alpha_7$}
\put(2.08,-0.55){\footnotesize $\alpha_2$}
\end{picture}
$$
\vskip 45pt

The results are given in Table \ref{table 4}.

\begin{table}[H]  
\caption{Some nilpotent orbits for $\wt{E}_7^{(n)}$}
\label{table 4}
\vskip 5pt
\renewcommand{\arraystretch}{1.3}
\begin{tabular}{|c|c|c|c|c|c|c|c|}
\hline
 $\mca{O}$ & special/even?  &     $\mbf{G}_{\gamma, der}$  &  $(Q_1, Q_2)$ &  quasi-adm., iff & raisable, if     \\
\hline
\hline
 $4A_1$ & no/no & $\Sp_6$ & $(1, 7)$ &   $n=2$ & $n\ne 2$    \\
 \hline
 $2A_2 + A_1$ & no/no & $\SL_2^a \times \SL_2^b$ &  $(3, 8), (3, 16)$ & $n=1, 3$ & n.a.    \\
 \hline
 $A_3 + A_2 + A_1$ & yes/yes  & $\SL_2$ & $(24, 0)$  & $n|24$ & n.a.     \\
 \hline
 $A_4 + A_2$ & yes/yes & $\SL_2$ & $(15, 0)$  & $n|15$ & n.a.  \\
\hline
 $A_6$ & yes/yes  & $\SL_2$ & $(7, 0)$  & $n|7$ & n.a.  \\
\hline
 $E_6(a_1)$ &yes/yes & $1$ & n.a.  & all $n$ & n.a.  \\
\hline
\end{tabular}
\end{table}

The orbit $4A_1$ has  $G_{\gamma, der} \simeq \Sp_6 \subset \SL_6$, where $\SL_6$ is associated with $\Delta - \set{\alpha_2, \alpha_7}$. Also, as $G_{\gamma, der}$-module, one has
$$\mfr{g}[1]=2V^6 \oplus V^{14}.$$
Let $\SL_{2,\tau} \subset G_{\gamma, der}$ be associated with the long root. As $\SL_{2,\tau}$-module, we have
$$\mfr{g}[1] = 7V^2 \oplus 12 V^1$$
with $m=7$. This shows that 
$$(Q_1, Q_2)=(1, 7)$$
and thus the orbit is quasi-admissible if and only if $n=2$. Moreover, considering the above $\SL_{2,\tau}$, we see that it is raisable if $n\ne 2$.

The orbit $2A_2 + A_1$ has $G_{\gamma, der}=\SL_2^a \times \SL_2^b \subset \SL_{2, \alpha_1} \times \SL_4 \times \SL_{2, \alpha_7}$. Here $\SL_2^a \times \SL_2^b \into \SL_{2, \alpha_1} \times \SL_{2, \alpha_7}$ is the identity, and $\SL_2^a \times \SL_2^b \into \SL_4$ is the tensor embedding where $\SL_4$ is associated with $\set{\alpha_2, \alpha_4, \alpha_5}$. One has
$$\mfr{g}[1] = 2(V_a^2 \boxtimes V_b^3) \oplus  (V_a^4 \oplus 2V_a^2) \boxtimes V_b^1.$$
Thus, for $\SL_2^a$ one has $(Q_1, Q_2) = (3, 18)$; for $\SL_2^b$ one has $(Q_1, Q_2)= (3, 16)$. We see that the orbit is quasi-admissible if and only if $n=1, 3$. The method in \cite{JLS} does not work for raisability.

The orbit $A_3 + A_2 + A_1$  gives
$$G_{\gamma, der} \simeq \SL_2 \into \SL_5 \times \SL_3,$$
where   the embedding is given by $4\omega_1 \otimes 2\omega_1$. Here $\omega_1$ is the fundamental weight of $\SL_2$. Also, $\SL_5$ and $\SL_3$ are associated with $\set{\alpha_i: 1\lest i \lest 4}$ and $\set{\alpha_6, \alpha_7}$ respectively. One has
$$(Q_1, Q_2) = (24, 0).$$
Thus, the orbit is quasi-admissible if and only if $n|24$. For $\tau = {\rm id}: \SL_2 \to G_{\gamma, der}$ the condition (C1) is not satisfied, thus the method of \cite{JLS} does not apply.

The orbit $A_4 + A_2$ gives 
$$G_{\gamma, der} = \SL_2 \into \SL_4 \times \SL_3 \times \SL_{2, \alpha_2}$$
where the embedding is given by $3\omega_1 \otimes 2\omega_1 \otimes \omega_1$. Here $\SL_4, \SL_3$ are associated with $\set{\alpha_5, \alpha_6, \alpha_7}$ and $\set{\alpha_1, \alpha_3}$ respectively. We have
$$(Q_1, Q_2) = (15, 0)$$
and thus the orbit is quasi-admissible if and only if $n|15$. The conditions for raisability are not satisfied.

The orbit $A_6$ gives
$$G_{\gamma, der} = \SL_2 \into \SL_2^a \times  \SL_2^b \times \SL_2^c \times \SL_3,$$
where the embedding is given by identities into $\SL_2^i, i=a, b, c$ and $2\omega_1: \SL_2\into \SL_3$. We have $\mfr{g}[1]=0$.
We get
$$(Q_1, Q_2)=(7, 0)$$
and thus the orbit is quasi-admissible if and only if $7|n$. Since (C1) for the identity map $\SL_2 \to G_{\gamma, der}$ is not satisfied, the method of raisability does not apply.


%
\subsubsection{Type $E_8$}
Consider the Dynkin diagram for the simple roots of $E_8$:

$$
\begin{picture}(4.7,0.2)(0,0)
\put(1,0){\circle{0.08}}
\put(1.5,0){\circle{0.08}}
\put(2,0){\circle{0.08}}
\put(2.5,0){\circle{0.08}}
\put(3,0){\circle{0.08}}
\put(3.5,0){\circle{0.08}}
\put(4,0){\circle{0.08}}
\put(2,-0.5){\circle{0.08}}
\put(1.04,0){\line(1,0){0.42}}
\put(1.54,0){\line(1,0){0.42}}
\put(2.04,0){\line(1,0){0.42}}
\put(2,-0.04){\line(0,-1){0.42}}
\put(2.54,0){\line(1,0){0.42}}
\put(3.04,0){\line(1,0){0.42}}
\put(3.54,0){\line(1,0){0.42}}
\put(1,0.1){\footnotesize $\alpha_1$}
\put(1.5,0.1){\footnotesize $\alpha_3$}
\put(2,0.1){\footnotesize $\alpha_4$}
\put(2.5,0.1){\footnotesize $\alpha_5$}
\put(3,0.1){\footnotesize $\alpha_6$}
\put(3.5,0.1){\footnotesize $\alpha_7$}
\put(4,0.1){\footnotesize $\alpha_8$}
\put(2.08,-0.55){\footnotesize $\alpha_2$}
\end{picture}
$$
\vskip 45pt

The results are given in Table \ref{table 5}.

\begin{table}[H]  
\caption{Some nilpotent orbits for $\wt{E}_8^{(n)}$}
\label{table 5}
\vskip 5pt
\renewcommand{\arraystretch}{1.3}
\begin{tabular}{|c|c|c|c|c|c|c|c|}
\hline
 $\mca{O}$ & special/even?  &     $\mbf{G}_{\gamma, der}$  &  $(Q_1, Q_2)$ &  quasi-adm., iff & raisable, if     \\
\hline
\hline
 $4A_1$ & no/no & $\Sp_8$ & $(1, 15)$ &  $n=2$ & $n\ne 2$    \\
\hline
 $2A_2 + 2A_1$ & no/no & $\Sp_4$  & $(3, 34)$   & $n=1, 3$   & n.a.    \\
\hline
$2A_3$ & no/no & $\Sp_4$ & $(2, 15)$  &  $n=4$ & n.a.    \\
\hline
 $A_4 + A_3$ & no/no & $\SL_2$ & $(10, 58)$  &  $n|10$ & n.a.    \\
\hline
 $A_6 + A_1$ & yes/no & $\SL_2$ & $(7, 14)$  &  $n=1, 7$ & n.a.     \\
\hline
 $A_7$ & no/no & $\SL_2$ & $(4, 15)$  &   $n=8$ & n.a.    \\
\hline
\end{tabular}
\end{table}

The orbit $4A_1$ of $E_8$ gives
$$G_{\gamma, der} = \Sp_8 \into \SL_8,$$
where $\SL_8$ is associated with $\Delta - \set{\alpha_2}$.
Also, $$
\mfr{g}[1]=\extp^3 V^8.$$
Now, if we take $\SL_{2, \tau}$ to be associated with any long root of $G_{\gamma, der}$, then (C1)--(C3) are satisfied with $m=15$. 
This shows that 
$$(Q_1, Q_2) = (1, 15)$$
and thus the orbit is quasi-admissible if and only if $n=2$. It also shows that the orbit is raisable if $n\ne 2$.


The orbit $2A_3 + 2A_1$ has 
$$G_{\gamma, der} = \Sp_4 \into \SL_4 \times \SL_5$$
via the diagonal map, where $\Sp_4 \into \SL_4$ is the canonical inclusion and $\SL_4 = \Spin_5 \onto \SO_5 \into \SL_5$. Here $\SL_4$ and $\SL_5$ are associated with $\set{\alpha_6, \alpha_7, \alpha_8}$ and $\set{\alpha_1, \alpha_2, \alpha_3, \alpha_4}$ respectively. From this, we get that
$$Q_1 = 3.$$
Take $\SL_{2, \tau} \subset \Sp_4$ to be associated with the long root of $\Sp_4$, we have that 
$$\mfr{g}[1] = 8V^1 \oplus 8 V^2 \oplus 4V^3 \oplus V^4.$$
Thus, $Q_2 =34$ and this shows that the orbit is quasi-admissible if and only if $n=1, 3$. The method of raisability does not apply.

For the orbit $2A_3$, consider
$$L=\SL_4^a \times \SL_4^b$$
which acts on $\mfr{g}[1]=V^4_b \oplus (V^4_a \boxtimes \extp^2 V^4_b)$. Here $\SL_4^a, \SL_4^b$ are associated with $\set{\alpha_2, \alpha_3, \alpha_4}$ and $\set{\alpha_6, \alpha_7, \alpha_8}$ respectively. We have 
$$G_{\gamma, der} = \Sp_4 \into L$$ 
via the diagonal embedding. 
Let $\SL_{2,\tau} \subset G_{\gamma, der}$ be the one associated with the long root.  We have that as an $\SL_{2,\tau}$-module,
$$\mfr{g}[1] = 8V_1 \oplus 7V_2 \oplus 2V_3.$$
This gives that
$$(Q_1, Q_2) =(2, 15)$$
and thus the orbit is quasi-admissible if and only if $n=4$. For raisability, the method in \cite{JLS} does not work.

The orbit $A_4 + A_3$ has 
$$G_{\gamma, der} = \SL_2 \into \SL_2 \times \SL_2 \times \SL_3 \times \SL_3$$
where the embedding is diagonal via ${\rm id}: \SL_2 \to \SL_2$ and ${\rm Sym}^2: \SL_2 \into \SL_3$. Here the $\SL_2, \SL_3$ are associated to the connected components of $\Delta - \set{\alpha_4, \alpha_7}$ in the Dynkin diagram. One has
$$\mfr{g}[1] = 3V_2 \oplus 3V_4 \oplus V_6$$
as a $G_{\gamma, der}$-module. We get
$$(Q_1, Q_2) =(10, 68)$$
and thus the orbit is quasi-admissible if and only if $n|10$. The method of \cite{JLS} does not apply for raisability.

The orbit $A_6 + A_1$ has 
$$G_{\gamma, der} \simeq \SL_2 \into \SL_2 \times \SL_2 \times \SL_2 \times \SL_3.$$
Here the three $\SL_2$'s and $\SL_3$ are associated to the connected component of $\Delta -\set{\alpha_1, \alpha_4, \alpha_6}$ in the Dynkin diagram. We have
$$\mfr{g}[1] = V_4 \oplus 4V_2$$
and hence
$$(Q_1, Q_2)=(7, 14).$$
This shows that the orbit is quasi-admissible if and only if $n=1, 7$. On the other hand, the method of \cite{JLS} for raisability does not apply.

The orbit $A_7$ gives
$$G_{\gamma, der} \simeq \SL_2 \into  \SL_2 \times \SL_2 \times \SL_2 \times \SL_2,$$
diagonally embedded in the $\SL_2$'s associated with $\alpha_2, \alpha_3, \alpha_5, \alpha_8$. One has 
$$\mfr{g}[1] = 5V_2 \oplus V_4$$
as a $G_{\gamma, der}$-module.  This gives that
$$(Q_1, Q_2) = (4, 15).$$
Hence the orbit is quasi-admissible if and only if $n=8$.  The method of \cite{JLS} does not apply for raisability.

\section{Wavefront sets of theta representations} \label{S:wf}

In this section, we consider theta representations $\Theta(\nu)$ of $\wt{G}^{(n)}$, where $\nu \in X \otimes \R$ is a certain exceptional vector. We compute explicitly the $F^{\rm al}$-orbit $\mca{O}_{\rm Spr}(j_{W_\tnu}^W \varepsilon_\tnu) \subset \mfr{g}_{F^{\rm al}}$, which  is expected to be the single element in $\mca{N}_{\rm Wh}^{\rm max}(\Theta(\nu))\otimes F^{\rm al}$. We show that a method of determining $\mca{O}_{\rm Spr}(j_{W_\tnu}^W \varepsilon_\tnu) \subset \mfr{g}_{F^{\rm al}}$ is given by Sommers' duality \cite{Som01} between nilpotent orbits, which generalizes the classical Barbasch--Vogan duality. With an explicit computation, we also determine the quasi-admissibility and non-raisability of such orbits.

\subsection{Theta representation}
We introduce theta representations following the notation and exposition in \cite{GaTs}. 

Let $\wt{T} \subset \wt{G}$ be the covering torus of $\wt{G}$. We assume that there exists a certain distinguished (finite-dimensional) genuine representation $\pi^\dag$ of $\wt{T}$ determined by a distinguished genuine central character $\chi^\dag$ of $Z(\wt{T})$, see \cite[\S 6--7]{GG}.
For every $\nu \in X\otimes \R$, there is a map
$$\delta_\nu: T \longrightarrow \C^\times$$
given by 
$$\delta_\nu(y\otimes a) = |a|_F^{\nu(y)}$$
on the generators $y\otimes a\in T$, where $\nu(y)$ is the natural pairing between $Y$ and $X\otimes \R$.
For every $\nu\in X\otimes \R$, denote by
$$I(\pi^\dag, \nu):=\Ind_{\wt{B}}^{\wt{G}} (\pi^\dag \otimes \delta_\nu)$$
the normalized induced principal series representation of $\wt{G}$.

A vector $\nu \in X\otimes \R$ is called an exceptional character if $\nu(\alpha_{Q,n}^\vee) =1$ for every $\alpha\in \Delta$. Here $\alpha_{Q,n}^\vee:=n_{\alpha}\cdot  \alpha^\vee$. 
It follows from the Langlands classification theorem for covers (see \cite{BJ1}) that if $\nu\in X\otimes \R$ is exceptional, then we have
$$I(\pi^\dag, \nu) \onto \Theta(\pi^\dag, \nu),$$
where $\Theta(\pi^\dag, \nu)$ is the unique Langlands quotient of $I(\pi^\dag, \nu)$. We may write $\Theta(\wt{G}, \nu), \Theta(\nu)$ or $\Theta(\wt{G})$ for $\Theta(\pi^\dag, \nu)$, whenever the emphasis is different; but the dependence on $\pi^\dag$ and exceptional $\nu$ are both understood.

A covering group $\wt{G}$ is called saturated (see \cite[Definition 2.1]{Ga6}) if 
$$Y^{sc} \cap Y_{Q,n} = Y_{Q,n}^{sc},$$
where the one-sided inclusion $\supset$ always holds. 
If $G$ is semisimple and simply-connected, then $G$ is saturated if and only if its dual group $\wt{G}^\vee$ is of adjoint type, i.e., $Y_{Q,n} = Y_{Q,n}^{sc}$.
In general, for every $\alpha\in \Phi$ one has
$$\Z[\alpha^\vee] \cap Y_{Q,n} = \Z[i_\alpha \cdot \alpha_{Q,n}^\vee] \text{ with } i_\alpha \in \set{1, 1/2},$$
and $i_\alpha =1/2$ only if $n_\alpha$ is even. Set
$$\tilde{n}_\alpha = i_\alpha \cdot n_\alpha, \ \tilde{\alpha}_{Q,n}^\vee=\tilde{n}_\alpha \cdot \alpha^\vee, \text{ and } \tilde{\alpha}_{Q,n} = \alpha/\tilde{n}_\alpha$$
for every $\alpha \in \Phi$, and
$$\tilde{\Phi}_{Q,n}^\vee:=\set{\tilde{\alpha}_{Q,n}^\vee: \ \alpha\in \Phi}.$$
Let 
$$\tilde{Y}_{Q,n}^{sc} \subset Y_{Q,n}$$
 be the sublattice spanned by $\tilde{\Phi}_{Q,n}^\vee$, and we call it the saturation of $Y_{Q,n}^{sc}$. One has
$$Y_{Q,n}^{sc} \subset \tilde{Y}_{Q,n}^{sc} \subset Y_{Q,n},$$
and if $\wt{G}$ is saturated, then $Y_{Q,n}^{sc} = \tilde{Y}_{Q,n}^{sc}$. 

An element $\tnu \in X\otimes \R$ is called a saturation of an exceptional character $\nu \in X\otimes \R$ if $\tilde{\nu}(\tilde{\alpha}_{Q,n}^\vee) =1$ for every $\alpha\in \Delta$.
If $\wt{G}$ is saturated, then $\phi(\tnu) = \phi(\nu) \in \Hom(Y_{Q,n}^{sc}, \R)$ for every saturation $\tnu$ of an exceptional $\nu$.

We also recall the notion of a persistent cover as follows (see \cite[Definition 2.3]{Ga6}). Consider
$$\msc{X}_{Q,n}^{sc}:=Y/Y_{Q,n}^{sc}, \ \msc{X}_{Q,n}:=Y/Y_{Q,n},$$
 which are both endowed with the twisted Weyl action 
 $$w[y]:=w(y-\rho^\vee) + \rho^\vee$$
 for every $w$ in the Weyl group $W$.
 Here $\rho^\vee$ is the half sum of all positive coroots in $\Phi^\vee$. For every $y\in Y$, let $y^\dag$ and $y^\ddag$ denote its image in $\msc{X}_{Q,n}^{sc}$ and $\msc{X}_{Q,n}$ respectively.  An $n$-fold cover $\wt{G}$ is called persistent if 
$${\rm Stab}_W(y^\dag; \msc{X}_{Q,n}^{sc}) = {\rm Stab}_W(y^\ddag; \msc{X}_{Q,n})$$
for every $y\in Y$. A saturated cover is always persistent.

\subsection{The set $\mca{N}_{\rm tr}^{\rm max}(\Theta(\pi^\dag, \nu))$}  \label{SS:WF}
For every $\nu \in X\otimes \R$, denote by
$$W_\nu=\set{w\in W: \ w(\nu) - \nu \in X^{sc}} \subset W$$
the integral Weyl subgroup associated with $\nu$. It is a reflection subgroup associated with the root subsystem
\begin{equation} \label{Eq:Phi-nu}
\Phi_\nu = \set{\alpha\in \Phi: \ \angb{\nu}{\alpha^\vee} \in \Z}.
\end{equation}
The MacDonald--Lusztig--Spaltenstein $j$-induction thus gives an irreducible representation $j_{W_\nu}^W (\varepsilon_{W_\nu})$ of $W$, where $\varepsilon_{W_\nu}$ denotes the sign character of $W_\nu$. Write $\varepsilon_\nu= \varepsilon_{W_\nu}$.

Let $\mfr{g}\otimes F^{\rm al}$ be the Lie algebra of $\mbf{G}$ over the algebraically closed field $F^{\rm al}$. Let  $x\in \mfr{g}\otimes F^{\rm al}$ be a nilpotent element and consider $A_x:=\mbf{G}_{ad,x}/(\mbf{G}_{ad,x})^o$. Since $A_x$ depends only on the conjugacy class $\mca{O}_x$ of $x$, for a nilpotent orbit $\mca{O} \subset \mfr{g}\otimes F^{\rm al}$, we use  $A_{\mca{O}}$ to denote $A_x$ for any $x\in \mca{O}$.
Defining
$$\mca{N}^{\rm en}(\mbf{G})=\set{(\mca{O}, \eta): \ \mca{O} \in \mca{N}(\mbf{G}) \text{ and } \eta \in \Irr(A_\mca{O})},$$
the Springer correspondence gives an injective map
$$\begin{tikzcd}
{\rm Spr}^\mbf{G}: \Irr(W) \ar[r, hook] & \mca{N}^{\rm en}(\mbf{G})
\end{tikzcd}$$
denoted by
$${\rm Spr}^\mbf{G}(\sigma)=(\mca{O}_{\rm Spr}^\mbf{G}(\sigma), \eta(\sigma)).$$
If no confusion arises, we will write ${\rm Spr} = {\rm Spr}^\mbf{G}$, $\mca{O}_{\rm Spr}=\mca{O}_{\rm Spr}^\mbf{G}$. We call
$$\mca{O}_{\rm Spr}^\mbf{G}(\sigma) \subset \mfr{g}\otimes F^{\rm al}$$
the nilpotent orbit associated with $\sigma$. Here the normalization is such that  $\mca{O}_{\rm Spr}^\mbf{G}(\mbm{1}) = \mca{O}_{\rm reg}$, and  $\mca{O}_{\rm Spr}^\mbf{G}(\varepsilon_W) = \mca{O}_0$. Note that for every $\mca{O} \in \mca{N}(\mbf{G})$, the pair $(\mca{O}, \mbm{1})$ lies in the image of ${\rm Spr}$, i.e., $(\mca{O}, \mbm{1}) = {\rm Spr}(\sigma_\mca{O})$ for a unique $\sigma_\mca{O} \in \Irr(W)$. This gives us a well-defined injective map
$$\begin{tikzcd}
{\rm Spr}_\mbm{1}^{-1}: \mca{N}(\mbf{G}) \ar[r, hook] & \Irr(W) 
\end{tikzcd}$$
given by ${\rm Spr}_\mbm{1}^{-1}(\mca{O}):={\rm Spr}^{-1}((\mca{O}, \mbm{1}))$.
It is clear that $\mca{O}_{\rm Spr} \circ {\rm Spr}_\mbm{1}^{-1} = \text{id}_{\mca{N}(\mbf{G})}$. However, $ {\rm Spr}_\mbm{1}^{-1}\circ \mca{O}_{\rm Spr}$ may not be the identity map on $\Irr(W)$.

One has the permutation representation
\begin{equation} \label{F:sigmaX}
\sigma^\msc{X}: W \longrightarrow {\rm Perm}(\msc{X}_{Q,n})
\end{equation}
given by the twisted Weyl action $w[y] = w(y-\rho^\vee) + \rho^\vee$. We have the following expectation on the stable wavefront set of $\Theta(\pi^\dag, \nu)$.

\begin{conj}[\cite{GaTs}] \label{C:main}
Let $F$ be $p$-adic with $p\nmid n$. Let $\wt{G}$ be a persistent $n$-fold covering group. Let $\nu \in X\otimes \R$ be exceptional and let $\tilde{\nu} \in X\otimes \R$ be a saturation of $\nu$. Then for the Harish-Chandra local character expansion of $\Theta(\pi^\dag, \nu)$ as in \eqref{E:char}, one has
\begin{equation} \label{E:main1}
\mca{N}_{\rm tr}^{\rm max}(\Theta(\pi^\dag, \nu))\otimes F^{\rm al} = \set{ \mca{O}_{\rm Spr}^\mbf{G}(j_{W_\tnu}^W(\varepsilon_\tnu)) }
\end{equation}
and  
\begin{equation} \label{E:main2}
c_{\mca{O}} = \angb{j_{W_\tnu}^W(\varepsilon_\tnu)}{ \varepsilon_W \otimes \sigma^\msc{X} }_W
\end{equation}
for every orbit $\mca{O} \in \mca{N}_{\rm tr}^{\rm max}(\Theta(\pi^\dag, \nu))$.
\end{conj}

In \cite{GaTs}, we showed part of the Conjecture when $\Theta(\pi^\dag, \nu)$ is generic. Compatibility with existing work in the literature was also verified. This verification depends solely on explicating the right hand side of \eqref{E:main1} for several cases of $\wt{G}^{(n)}$ of interest.

\subsection{Method of computation} \label{SS:som-com}
Here, we want to compute the orbit $\mca{O}_{\rm Spr}^\mbf{G}(j_{W_{\tnu}}^W(\varepsilon_{\tnu}))$ explicitly for all covers, at least when $G$ is almost simple and simply-connected or is of the classical type. We show that the method of computation is reduced to the Sommers' duality as in \cite{Som01}.

For any linear algebraic group $\mbf{G}$, we consider another enhanced set of nilpotent orbits 
$$\mca{N}^{\rm en}_{\rm geo}(\mbf{G}):=\set{(\mca{O}, \mfr{c}): \mca{O} \in \mca{N}(\mbf{G}) \text{ and } \mfr{c} \in {\rm Conj}(A_\mca{O})},$$
where ${\rm Conj}(A_\mca{O})$ denotes the set of conjugacy classes of $A_\mca{O}$. There is a quotient map
$$A_\mca{O} \onto \tilde{A}_\mca{O},$$
where $\tilde{A}_\mca{O}$  is the  Lusztig canonical quotient, see \cite[\S 5]{Som01}. This gives a set
$$\tilde{\mca{N}}^{\rm en}_{\rm geo}(\mbf{G}):=\set{(\mca{O}, \tilde{\mfr{c}}): \mca{O} \in \mca{N}(\mbf{G}) \text{ and } \tilde{\mfr{c}} \in {\rm Conj}(\tilde{A}_\mca{O})}$$
together with a natural surjection
$$f: \mca{N}^{\rm en}_{\rm geo}(\mbf{G}) \onto \tilde{\mca{N}}^{\rm en}_{\rm geo}(\mbf{G})$$
given by $(\mca{O}, \mfr{c}) \mapsto (\mca{O}, \tilde{\mfr{c}})$ where $\tilde{\mfr{c}}$ is the image of $\mfr{c}$.

Recall that there is a natural order-reversing bijection on special orbits $\mca{N}^{\rm spe}(\mbf{G}) \subset \mca{N}(\mbf{G})$, which can be extended to give the Lusztig--Spaltenstein map 
$$d_{\rm LS}: \mca{N}(\mbf{G}) \onto \mca{N}^{\rm spe}(\mbf{G}).$$
Now we have the various extensions of $d_{\rm LS}$ as depicted in the following commutative diagram:
\begin{equation} \label{CD:dSom}
\begin{tikzcd}
\mca{N}^{\rm spe}(\mbf{G}) \ar[r, "\iota"] & \mca{N}^{\rm spe}(\mbf{G}^\vee) \ar[r, hook]  & \mca{N}(\mbf{G}^\vee) & \tilde{\mca{N}}^{\rm en}_{\rm geo}(\mbf{G}^\vee) \ar[l, two heads, "{\rm pr}_1"]  \\
& \mca{N}(\mbf{G}) \ar[lu, "{d_{\rm LS}}"]   \ar[u, "{d_{\rm BV}}"] \ar[r, hook] & \mca{N}^{\rm en}_{\rm geo}(\mbf{G})   \ar[u, "{d_{\rm Som}}"] \ar[r, two heads, "f"]  &  \tilde{\mca{N}}^{\rm en}_{\rm geo}(\mbf{G}) \ar[u, "{d_{\rm Ach}}"'] \ar[lu, "{d_{\rm Som}}"] .
\end{tikzcd}
\end{equation}
Here $\iota$ is the canonical bijection arising from the identification of the Weyl group $W$ of $\mbf{G}$ and that of its Langlands dual group $\mbf{G}^\vee$. Since by definition, there is a bijection between special representations of the Weyl group and the special nilpotent orbits, one has the bijection $\iota$. Moreover, the Barbasch--Vogan duality is given by
$$d_{\rm BV}:=\iota \circ d_{\rm LS}.$$
The top inclusion in \eqref{CD:dSom}  is the canonical one, and the bottom inclusion is given by map $\mca{O} \mapsto (\mca{O}, \mbm{1})$. On the other hand, the existence of the extended maps 
$$d_{\rm Som}: \mca{N}^{\rm en}_{\rm geo}(\mbf{G}) \to  \mca{N}(\mbf{G}^\vee)$$ 
and $d_{\rm Ach}$ is not trivial and are given in \cite{Som01} and \cite{Ach03} respectively. It was shown in \cite{Som01} that the map $d_{\rm Som}$ factors through $f$ and thus can be defined on $\tilde{\mca{N}}^{\rm en}_{\rm geo}(\mbf{G})$ as well.

What pertains to our work is the map $d_{\rm Som}$ and thus we give some elaboration on it. Take any 
$$(\mca{O}, \mfr{c}) \in \mca{N}^{\rm en}_{\rm geo}(\mbf{G}).$$
One picks $e \in \mca{O}$ and a semisimple element $s\in Z_\mbf{G}(e)$ such that
$$\phi(s)=\mfr{c},$$
where $\phi: Z_\mbf{G}(e) \onto A_\mca{O}$ is the canonical quotient map. We take 
$$\mbf{L}:=Z_\mbf{G}(s)$$
which is a so-called pseudo-Levi subgroup (see \cite[\S 6]{McSo03}) of $\mbf{G}$ with $e\in \mbf{L}$. Consider the nilpotent orbit $\mca{O}_e^\mbf{L} \subset {\rm Lie}(\mbf{L})$. It is known that there exists $\sigma_e \in \Irr(W(\mbf{L}))$ such that
$${\rm Spr}(\sigma_e) = (d_{\rm LS}(\mca{O}_e^\mbf{L}), \mbm{1}) \in \mca{N}^{\rm en}(\mbf{L}).$$
Hence
$$\sigma_e = {\rm Spr}^{-1}_\mbm{1} \circ d_{\rm LS}(\mca{O}_e^\mbf{L}).$$
Now consider
$$j_{W(\mbf{L})}^W(\sigma_e) \in \Irr(W),$$
where we identify the Weyl groups of $\mbf{G}$ and $\mbf{G}^\vee$. The desired orbit arising from $d_{\rm Som}$ is then
 $$d_{\rm Som}((\mca{O}, \mfr{c}))= \mca{O}_{\rm Spr}^{\mbf{G}^\vee}(j_{W(\mbf{L})}^W(\sigma_e))  \in \mca{N}(\mbf{G}^\vee),$$
 which is independent of the choices of $e$ and $s$. 
 In fact, we have ${\rm Spr}^{\mbf{G}^\vee}(j_{W(\mbf{L})}^W(\sigma_e)) = (d_{\rm Som}(\mca{O}, \mfr{c}), \mbm{1})$.
 
  
For $(\mca{O}, \mfr{c})$ above, one can pick $(s, e)$ such that $\mca{O}_e^\mbf{L}$ is a distinguished nilpotent orbit of $\mbf{L}$. In fact, a generalized Bala--Carter classification for $\mca{N}^{\rm en}_{\rm geo}(\mbf{G})$ was constructed by Sommers, using such pairs satisfying the minimal ``key property" \cite[Page 548]{Som98}. Recall that the Bala--Carter classification gives a bijection
$$f_{\rm BC}: \set{(\mbf{L}', \mca{O}_{\mbf{L}'})}  \longrightarrow  \mca{N}(\mbf{G}),$$
where $\mbf{L}' \subset \mbf{G}$ is a Levi subgroup and $\mca{O}_{\mbf{L}'} \subset {\rm Lie}(\mbf{L})$ a distinguished orbit. By considering more generally  a general pseudo-Levi subgroup $\mbf{L} \subset \mbf{G}$ and a distinguished orbit $\mca{O}_\mbf{L} \subset {\rm Lie}(\mbf{L})$, Sommers showed that there is a natural bijection $\set{(\mbf{L}, \mca{O}_\mbf{L})} \to \mca{N}^{\rm en}_{\rm geo}(\mbf{G})$ such that the following diagram
$$\begin{tikzcd}
\set{(\mbf{L}', \mca{O}_{\mbf{L}'})} \ar[d, hook] \ar[r, "f_{\rm BC}"] & \mca{N}(\mbf{G})   \ar[d, hook] \\
\set{(\mbf{L}, \mca{O}_\mbf{L})} \ar[r, "{f_{\rm BC}}"] & \mca{N}^{\rm en}_{\rm geo}(\mbf{G})
\end{tikzcd}$$
commutes, i.e., it extends the classical Bala--Carter classification.

Combining the above, one has a natural map
$$d_{\rm Som}^\heartsuit:= d_{\rm Som} \circ f_{\rm BC}: \set{(\mbf{L}, \mca{O}_\mbf{L})} \longrightarrow \mca{N}(\mbf{G}^\vee).$$
The construction of the map $d_{\rm Som}$ immediately gives the following:

\begin{prop} \label{P:dSom}
Let $\mca{O}$ be a distinguished orbit of a pseudo-Levi $\mbf{L}$. Then 
$$d_{\rm Som}^\heartsuit( (\mbf{L}, \mca{O}) ) =\mca{O}_{\rm Spr}^{\mbf{G}^\vee}(j_{W(\mbf{L})}^W \circ {\rm Spr}_\mbm{1}^{-1} \circ d_{\rm LS}(\mca{O})),$$
where $j_{W(\mbf{L})}^W \circ {\rm Spr}_\mbm{1}^{-1} \circ d_{\rm LS}(\mca{O})$ is viewed as representation of the Weyl group of $\mbf{G}^\vee$. In particular,
$$d_{\rm Som}^\heartsuit(\mbf{L}, \mca{O}_{\rm reg}) = \mca{O}_{\rm Spr}^{\mbf{G}^\vee}(j_{W(\mbf{L})}^W(\varepsilon_{W(\mbf{L})}) ).$$
\end{prop}

For applications purpose we exchange the roles of $\mbf{G}$ and $\mbf{G}^\vee$. In particular, retaining the notations in Conjecture \ref{C:main}, we have a pseudo-Levi subgroup $\mbf{L}^\vee_\tnu$ whose root system is $\Phi_\tnu^\vee$. It thus follows from Proposition \ref{P:dSom} that
\begin{equation} \label{F:dSom}
\mca{O}_{\rm Spr}^\mbf{G} (j_{W_\tnu}^W(\varepsilon_\tnu)) = d_{\rm Som}^\heartsuit(\mbf{L}_\tnu^\vee, \mca{O}_{\rm reg}).
\end{equation}

\subsection{The orbit $\mca{O}_{\rm Spr}(j_{W_\tnu}^W \varepsilon_\tnu)$ for classical groups} \label{SS:WFcl}
For type $A_m$ groups we can take $\tnu = \nu$, and pseudo-Levi groups are all Levi subgroups, and thus every $\mbf{L}_\tnu^\vee$ corresponds to a partition 
$$\mfr{p}_\nu= \mfr{p}_\tnu = (p_1^{d_1} p_2^{d_2} \cdots p_k^{d_k})$$
of $m+1$. Since $d_{\rm Som}^\heartsuit$ extends the Barbasch--Vogan duality, we have
$$\mca{O}_{\rm Spr}^\mbf{G} (j_{W_\tnu}^W(\varepsilon_\tnu)) = \mfr{p}_\nu^\top,$$
the transpose of the partition $\mfr{p}_\nu$. Consider the Kazhdan--Patterson cover $\wt{\GL}_r^{(n)}$. We write $r=an + b$ with $0\lest b < n$. Then  
$$\mfr{p}_\nu = ((a+1)^b a^{n-b})$$
and this gives
$$\mca{O}_{\rm Spr}^\mbf{G} (j_{W_\nu}^W(\varepsilon_\nu)) = (n^ab).$$

For types $B_r, C_r$ and $D_r$, we recall the formula for $d_{\rm Som}(\mca{O}, \mfr{c})$ given in \cite{Som01}, which then gives $d_{\rm Som}^\heartsuit(\mbf{L}, \mca{O}_{\rm reg})$. First, for the pair $(\mca{O}, \mfr{c}) \in \mca{N}^{\rm en}_{\rm geo}(\mbf{G})$, one can choose the aforementioned $(s, e)$ with additional properties:
\begin{enumerate}
\item[--] ${\rm Lie}(\mbf{L})$ has semisimple rank equal to that of $\mfr{g}$, and one has 
$${\rm Lie}(\mbf{L}) = \mfr{l}_1 \oplus \mfr{l}_2,$$
where $\mfr{l}_2$ is a semisimple Lie algebra of the same type as $\mfr{g}$ and $\mfr{l}_1$ a simple Lie algebra containing $-\check{\alpha} \in \check{\Delta}$ in the extended Dynkin diagram with simple roots $\check{\Delta}$.
\item[--] one has $e = e_1 + e_2, e_i \in \mfr{l}_i$ where $e_1$ is a distinguished element in $\mfr{l}_1$.
\end{enumerate}
The Lie algebra $\mfr{l}_1$ is of type $D, C$ or $D$ if $\mfr{g}$ is of type $B, C$ or $D$ respectively. To such pair $(s, e)$, one can attached a pair of partitions
$$(\mfr{p}_1, \mfr{p}_2)$$
corresponding to the orbits $\mca{O}_{e_1}^{\mfr{l}_1}, \mca{O}_{e_2}^{\mfr{l}_2}$ respectively. In fact, one has
$$\mfr{p}_\mca{O} = \mfr{p}_1 \cup \mfr{p}_2.$$

\begin{thm}[{\cite[Theorem 12]{Som01}}]\label{T:Som}
For type $B, C, D$ groups, assume that $(\mfr{p}_1, \mfr{p}_2)$ is associated to $(\mca{O}, \mfr{c})$ as above. Then
$$d_{\rm Som}(\mca{O}, \mfr{c}) = 
\begin{cases}
(\mfr{p}_1 \cup \mfr{p}_2^C)^\top_C & \text{ if $\mfr{g}$ is of type $B$}, \\
 (\mfr{p}_1 \cup \mfr{p}_2^B)^\top_B & \text{ if $\mfr{g}$ is of type $C$}, \\
(\mfr{p}_1 \cup (\mfr{p}_{2, D}^\top)^\top)^\top_D & \text{ if $\mfr{g}$ is of type $D$}.
\end{cases}
$$
\end{thm}
If $e_1=0$ (equivalently $\mfr{g}_1 = 0$), then $\mfr{p}_1 = \emptyset$; in this case, $d_{\rm Som}$ recovers the Barbasch--Vogan duality. In general, suppose $(\mca{O}, \mfr{c}) = f_{\rm BC}(\mbf{L}, \mca{O})$ for a distinguished $\mca{O}$ of $\mbf{L}$. We can write 
$$\mfr{l} = \mfr{g}_1 \oplus \mfr{g}_2'$$
where $\mfr{g}_1$ is a simple Lie algebra (if nonzero) containing $-\check{\alpha} \in \check{\Delta}$, and $\mfr{g}_2'$ is a Levi subalgebra of $\mfr{g}$. Also $\mca{O} = \mca{O}_{e_1}\oplus \mca{O}_{e_2}$ with $e_1, e_2$ distinguished in $\mfr{g}_1$ and $\mfr{g}_2'$ respectively. Then 
$$\mfr{p}_1 = \mfr{p}_{\mca{O}_{e_1}^{\mfr{g}_1}}.$$ 
On the other hand, let $\mfr{g}_2 \supset \mfr{g}_2'$ be the maximal Levi subalgebra such that its simple roots are disjoint from those of $\mfr{g}_1$ in $\check{\Delta}$. Consider the orbit $\mca{O}_{e_2}^{\mfr{l}_2}$ of $e_2$ in $\mfr{g}_2$, one has
$$\mfr{p}_2 = \mfr{p}_{\mca{O}_{e_2}^{\mfr{g}_2}}.$$ 
The above discussion readily applies to the case  of our interest when $\mca{O}_\mbf{L} = \mca{O}_{\rm reg}$ and thus $e_1$ and $e_2$ are regular in $\mfr{g}_1$ and $\mfr{g}_2'$ respectively.

Consider the cover of $G$ of type $B_r$. We have $\tnu=\nu$. We will consider both the cases of $G=\Spin_{2r+1}$ and $G=\SO_{2r+1}$. For $\wt{G}=\wt{\Spin}_{2r+1}^{(n)}$ with Brylinski--Deligne invariant 
$$\BDI(\wt{\Spin}_{2r+1}^{(n)})=1,$$
one has in this case the exceptional character 
$$\nu =
\begin{cases}
\rho/n & \text{ if $n$ is odd}, \\
2\omega_r/n + (\sum_{1\lest i \lest r-1} \omega_i)/n = \rho(C_r)/n & \text{ if $n$ is even}.
\end{cases}
$$
Now if we write
$$r=na + b \text{ with } 0\lest b <n,$$
then a direct computation gives:
\begin{enumerate}
\item[--] for odd $n=2m+1$ we have
$$
\Phi_\nu^\vee = 
\begin{cases}
C_a \times (A_{2a-1})^{m-b} \times (A_{2a})^b & \text{ if $0\lest b \lest m$}, \\
C_{a+1} \times (A_{2a})^{2m-b+1} \times (A_{2a+1})^{b-m-1} & \text{ if $m+1 \lest b \lest 2m$};
\end{cases}
$$
\item[--] for even $n=2m$ we have
$$
\Phi_\nu^\vee = 
\begin{cases}
C_a  \times C_a \times (A_{2a-1})^{m-1- b} \times (A_{2a})^b & \text{ if $0\lest b \lest m-1$}, \\
C_{a+1} \times  C_a \times  (A_{2a})^{2m-b-1} \times (A_{2a+1})^{b-m} & \text{ if $m \lest b \lest 2m-1$}.
\end{cases}
$$
\end{enumerate}
By applying \eqref{F:dSom} and Theorem \ref{T:Som}, we see that if $n=2m + 1$ is odd, then
$$
\mca{O}_{\rm Spr}(j_{W_\nu}^W \varepsilon_\nu) = \mca{O}^{2r+1, n}_B = 
\begin{cases}
(n^{2a}, 2b + 1) & \text{ if $0\lest b \lest m$},\\
(n^{2a+1}, 2b+1 - n) & \text{ if $m+1\lest b \lest 2m$}.
\end{cases}
$$
On the other hand, if $n=2m$ is even, then
$$
\mca{O}_{\rm Spr}(j_{W_\nu}^W \varepsilon_\nu) = \mca{O}^{2r+1, n}_B = 
\begin{cases}
(n^{2a}, 2b + 1) & \text{ if $0\lest b \lest m-1$},\\
(n^{2a}, n-1, 2b+1 - n, 1) & \text{ if $m\lest b \lest 2m-1$}.
\end{cases}
$$

For $\wt{\SO}_{2r+1}^{(n)}$, it depends on the parity of $n_{\alpha_1}=n/\gcd(n, 2)$. In particular, the exceptional character is of the form
$$\nu =
\begin{cases}
\rho/n_{\alpha_1} & \text{ if $n_{\alpha_1}$ is odd}, \\
2\omega_r/n_{\alpha_1} + (\sum_{1\lest i \lest r-1} \omega_i)/n_{\alpha_r} = \rho(C_r)/n_{\alpha_1} & \text{ if $n_{\alpha_1}$ is even}.
\end{cases}
$$
Thus, the root subsystem associated to an exceptional character of $\wt{\SO}_{2r+1}^{(n)}$ equals to the root subsystem associated to that of $\wt{\Spin}_{2r+1}^{(n_{\alpha_1})}$. This gives us the column for $\SO_{2r+1}$ in Table \ref{table 6}.

Now for $\wt{\Sp}_{2r}^{(n)}$ with $n$ odd, the detailed computation of $\mca{O}_{\rm Spr}(j_{W_\tnu}^W \varepsilon_\tnu)$ was already given in \cite[\S 4]{GaTs}. For $n$ even, similar analysis gives the orbit, as remarked in \cite[Remark 4.2]{GaTs}. The result is given in the column of $\Sp_{2r}$ in Table \ref{table 6}.

Consider group of type $D_r, r\gest 2$, where again we analyze both $\Spin_{2r}$ and $\SO_{2r}$. Consider $\Spin_{2r}$ and its $n$-fold cover with $\BDI(\wt{\Spin}_{2r}^{(n)})=1$, the exceptional character is 
$$\nu = \rho/n.$$
We write $r-1 = na + b$ with $0\lest b < n$. If $n=2m +1$ is odd, then the root subsystem is
$$
\Phi_\nu^\vee = 
\begin{cases}
D_{a+1} \times (A_{2a-1})^{m-b} \times (A_{2a})^b & \text{ if $0\lest b \lest m$}, \\
D_{a+1} \times (A_{2a})^{2m-b} \times (A_{2a+1})^{b-m} & \text{ if $m+1 \lest b \lest 2m$};
\end{cases}
$$
also, if $n=2m$ is even, then we have
$$
\Phi_\nu^\vee = 
\begin{cases}
D_{a+1}  \times D_a \times (A_{2a-1})^{m-1- b} \times (A_{2a})^b & \text{ if $0\lest b \lest m-1$}, \\
D_{a+1} \times  D_{a+1} \times  (A_{2a})^{2m-b-1} \times (A_{2a+1})^{b-m} & \text{ if $m \lest b \lest 2m-1$}.
\end{cases}
$$
By applying Theorem \ref{T:Som}, we see that if $n=2m + 1$ is odd, then
$$
\mca{O}_{\rm Spr}(j_{W_\nu}^W \varepsilon_\nu) = \mca{O}^{2r, n}_D = 
\begin{cases}
(n^{2a}, 2b + 1, 1) & \text{ if $0\lest b \lest m$},\\
(n^{2a+1}, 2b+2 - n) & \text{ if $m+1\lest b \lest 2m$}.
\end{cases}
$$
On the other hand, if $n=2m$ is even, then
$$
\mca{O}_{\rm Spr}(j_{W_\nu}^W \varepsilon_\nu) = (n+1, \mca{O}^{2r-n-1, n}_D) = 
\begin{cases}
(n+1, n^{2a-2}, n-1,  2b + 1, 1) & \text{ if $0\lest b \lest m-1$},\\
(n+1, n^{2a}, 2b+1 - n) & \text{ if $m\lest b \lest 2m-1$}.
\end{cases}
$$

The case of $\wt{\SO}_{2r}^{(n)}$ depends on the parity of $n_{\alpha}=n/\gcd(n, 2)$, where $\alpha$ is any root. In particular, the exceptional character is of the form
$$\nu = \rho/n_\alpha $$
Thus, the root subsystem associated to an exceptional character of $\wt{\SO}_{2r}^{(n)}$ equals to the root subsystem associated to that of $\wt{\Spin}_{2r}^{(n_\alpha)}$. This gives us the column for $\SO_{2r}$ in Table \ref{table 6}.

\begin{table}[H] 
\caption{$\mca{O}_{\rm Spr}(j_{W_\tnu}^W \varepsilon_\tnu)$ for classical groups} \label{table 6}
\vskip 5pt
\renewcommand{\arraystretch}{1.4}
\begin{tabular}{|c|c|c|c|c|c|}
\hline
 & $\wt{\GL}_r^{(n)}$  &    $\wt{\SO}_{2r+1}^{(n)}$  & $\wt{\Sp}_{2r}^{(n)}$ &  $\wt{\SO}_{2r}^{(n)}$       \\
\hline
\hline
$n$ odd   & $\mca{O}^{r, n}$   &  $\mca{O}^{2r +1, n}_{B}$   & $\mca{O}^{2r, n}_{C}$    &  $\mca{O}^{2r, n}_D$         \\
$n=2m, m$ even & $\mca{O}^{r, n}$    & $\mca{O}^{2r+1, m}_B$  & $\mca{O}^{2r, m}_{C}$    & $(m + 1, \mca{O}^{2r -m-1, m}_D)$      \\
$n=2k, k$ odd  & $\mca{O}^{r, n}$   & $\mca{O}^{2r+1, k}_B$  & $(k+1, \mca{O}^{2r-k-1, k}_{C})$   &  $\mca{O}^{2r, k}_D$    \\
\hline
\end{tabular}
\end{table}

\begin{rmk}
In a recent work \cite{BMW} by Bai--Ma--Wang, the authors devised two algorithms to compute the nilpotent orbit associated with highest (real) weight modules for all classical Lie algebras, which partially extend the recipes and some relevant results given in \cite{BX19, BXX23}. One such algorithm in \cite{BMW} is the so-called ``partition algorithm" and is based on Sommers' duality in \cite{Som01}. In particular, one can check that Table \ref{table 6} can be recovered by applying \cite[\S 1.2, Theorem]{BMW}  to $\tnu$ associated with an exceptional $\nu \in X\otimes \R$.
\end{rmk}

\subsection{The orbit $\mca{O}_{\rm Spr}(j_{W_\tnu}^W \varepsilon_\tnu)$ for exceptional groups}
For exceptional groups, $\tnu = \nu$ and the computation follows from \eqref{F:dSom} and the tables in \cite[\S 9]{Som01}. Thus, we have Table \ref{table 7} for $\wt{G}_2^{(n)}$.

\begin{table}[H] 
\caption{$\mca{O}_{\rm Spr}(j_{W_\nu}^W \varepsilon_\nu)$ for $\wt{G}_2^{(n)}$}
\label{table 7}
\vskip 5pt
\renewcommand{\arraystretch}{1.3}
\begin{tabular}{|c|c|c|c|c|c|}
\hline
$n$ & $\Phi_\chi$  &    $j_{W_\nu}^W \varepsilon_\nu$  & $\mca{O}_{\rm Spr}(j_{W_\nu}^W \varepsilon_\nu)$ &  $\dim \mca{O}$       \\
\hline
\hline
1   & $G_2$  & $\phi_{1, 6}$   & $\set{0}$   & 0     \\
\hline
2 & $\tilde{A}_1 + A_1$ & $\phi_{2, 2}$  & $\tilde{A}_1$ & 8   \\
\hline
3  & $\tilde{A}_2$  & $\phi_{1, 3}''$  & $A_1$ & 6    \\
  \hline
$4, 5, 6, 9$  & $\tilde{A}_1$  & $\phi_{2, 1}$  & $G_2(a_1)$ & 10     \\
\hline
$7, 8$ or $\gest 10$ & $\emptyset$ & $\phi_{1, 0}$   & $G_2$ & 12     \\
\hline
\end{tabular}
\end{table}

For $\wt{F}_4^{(n)}$, we have Table \ref{table 8}.

\begin{table}[H] 
\caption{$\mca{O}_{\rm Spr}(j_{W_\nu}^W \varepsilon_\nu)$ for $\wt{F}_4^{(n)}$}
\label{table 8}
\vskip 5pt
\renewcommand{\arraystretch}{1.3}
\begin{tabular}{|c|c|c|c|c|c|}
\hline
$n$ & $\Phi_\chi$  &    $j_{W_\nu}^W \varepsilon_\nu$  & $\mca{O}_{\rm Spr}(j_{W_\nu}^W \varepsilon_\nu)$ &  $\dim \mca{O}$       \\
\hline
\hline
1   & $F_4$  & $\phi_{1, 24}$  & $\set{0}$   &  0        \\
\hline
2 & $C_4$ & $\phi_{2, 16}''$  & $A_1$ & 16     \\
\hline
3  & $\tilde{A}_2 + A_2$  & $\phi_{6, 6}$  &  $\tilde{A}_2 + A_1$  & 36     \\
  \hline
4  & $A_3 + A_1$  & $\phi_{4, 7}''$  & $A_2 + \tilde{A}_1$ & 34      \\
\hline
5, 6 & $\tilde{A}_2 + A_1$ & $\phi_{12, 4}$   & $F_4(a_3)$ & 40      \\
\hline
7, 10 & $\tilde{A}_1 + A_1$ & $\phi_{9,2}$ & $F_4(a_2)$ & 44  \\
 \hline
8 & $\tilde{A}_2$ & $\phi_{8, 3}''$ & $B_3$ & 42   \\
\hline
9, 12, 14 & $\tilde{A}_1$ & $\phi_{4,1}$ & $F_4(a_1)$ & 46  \\
\hline
11, 13 or $\gest 15$ & $\emptyset$  & $\phi_{1, 0}$ & $F_4$  & 48    \\
\hline
\end{tabular}
\end{table}

For $\wt{E}_6^{(n)}$, we have Table \ref{table 9}.

\begin{table}[H] 
\caption{$\mca{O}_{\rm Spr}(j_{W_\nu}^W \varepsilon_\nu)$ for $\wt{E}_6^{(n)}$}
\label{table 9}
\vskip 5pt
\renewcommand{\arraystretch}{1.3}
\begin{tabular}{|c|c|c|c|c|c|}
\hline
$n$ & $\Phi_\chi$  &    $j_{W_\nu}^W \varepsilon_\nu$  & $\mca{O}_{\rm Spr}(j_{W_\nu}^W \varepsilon_\nu)$ &  $\dim \mca{O}$   \\
\hline
\hline
1   & $E_6$ & $\phi_{1, 36}$ & $\set{0}$  & 0         \\
\hline
2 & $A_5 + A_1$ & $\phi_{15, 16}$  & $3 A_1$ & 40     \\
\hline
3  & $3A_2$  & $\phi_{10, 9}$  &  $2A_2 + A_1$  & 54     \\
  \hline
4  & $2 A_2 + A_1$  & $\phi_{80, 7}$  & $D_4(a_1)$ & 60      \\
\hline
5 & $A_2 + 2 A_1$ & $\phi_{60, 5}$   & $A_4 + A_1$ & 62     \\
\hline
6, 7 & $3A_1$ & $\phi_{30, 3}$ & $E_6(a_3)$ & 66  \\
 \hline
8 & $ 2A_1$ & $\phi_{20, 2}$ & $D_5$ & 68   \\
\hline
9, 10, 11 & $A_1$ & $\phi_{6, 1}$ & $E_6(a_1)$ & 70  \\
\hline
$\gest 12$ & $\emptyset$  & $\phi_{1, 0}$ & $E_6$  & 72    \\
\hline
\end{tabular}
\end{table}

For $\wt{E}_7^{(n)}$, we have Table \ref{table 10}.

\begin{table}[!htbp] 
\caption{$\mca{O}_{\rm Spr}(j_{W_\nu}^W \varepsilon_\nu)$ for $\wt{E}_7^{(n)}$}
\label{table 10}
\vskip 5pt
\renewcommand{\arraystretch}{1.3}
\begin{tabular}{|c|c|c|c|c|c|}
\hline
$n$ & $\Phi_\chi$  &    $j_{W_\nu}^W \varepsilon_\nu$  & $\mca{O}_{\rm Spr}(j_{W_\nu}^W \varepsilon_\nu)$ &  $\dim \mca{O}$      \\
\hline
\hline
1   & $E_7$  & $\phi_{1, 63}$ & $\set{0}$   & 0         \\
\hline
2 & $A_7$ & $\phi_{15,28}$  & $4 A_1$ & 70     \\
\hline
3  & $A_5 + A_2$  & $\phi_{70, 18}$  &  $2A_2 + A_1$  & 90    \\
  \hline
4  & $A_4 + A_2$  & $\phi_{210, 13}$  & $A_3 + A_2 +A_1$ & 100      \\
\hline
5 & $A_3 + A_2 + A_1$ & $\phi_{210, 10}$   & $A_4 + A_2$ & 106     \\
\hline
6 & $A_2 + A_2 + A_1$ & $\phi_{315, 7}$ & $E_7(a_5)$ & 112  \\
 \hline
7 & $A_2 + 3 A_1$ & $\phi_{105, 6}$ & $A_6$ & 114   \\
 \hline
8 & $A_2 + 2 A_1$ & $\phi_{189, 5}$ & $E_7(a_4)$ & 116   \\
\hline
9 & $(4 A_1)''$ & $\phi_{120,4}$ & $E_6(a_1)$ & 118  \\
\hline
10, 11& $(3A_1)'$  & $\phi_{56,3}$ & $E_7(a_3)$  & 120    \\
\hline
12, 13 & $ 2A_1 $  & $\phi_{27, 2}$ & $E_7(a_2)$  & 122    \\
\hline
14, 15, 16, 17 & $ A_1$  & $\phi_{7,1}$ & $E_7(a_1)$  & 124    \\
\hline
$\gest 18$ & $\emptyset$ & $\phi_{1,0}$ & $E_7$ &  126   \\
\hline
\end{tabular}
\end{table}

For $\wt{E}_8^{(n)}$, we have Table \ref{table 11}.

\begin{table}[!htbp] 
\caption{$\mca{O}_{\rm Spr}(j_{W_\nu}^W \varepsilon_\nu)$ for $\wt{E}_8^{(n)}$}
\label{table 11}
\vskip 5pt
\renewcommand{\arraystretch}{1.3}
\begin{tabular}{|c|c|c|c|c|c|}
\hline
$n$ & $\Phi_\chi$  &    $j_{W_\nu}^W \varepsilon_\nu$  & $\mca{O}_{\rm Spr}(j_{W_\nu}^W \varepsilon_\nu)$ &  $\dim \mca{O}$  \\
\hline
\hline
1   & $E_8$   & $\phi_{1, 120}$ & $\set{0}$  & 0         \\
\hline
2 & $D_8$ & $\phi_{50, 56}$  & $4 A_1$ & 128    \\
\hline
3  & $A_8$  & $\phi_{175,36}$  &  $2A_2 + 2A_1$  & 168    \\
  \hline
4  & $D_5 + A_3$  & $\phi_{840,26}$  & $2A_3$ & 188    \\
\hline
5 & $A_4 + A_4$ & $\phi_{420, 20}$   & $A_4 + A_3$ & 200     \\
\hline
6 & $A_4 + A_3$ & $\phi_{4480, 16}$ & $E_8(a_7)$ & 208  \\
 \hline
7 & $A_4 + A_2 + A_1$ & $\phi_{2835,14}$ & $A_6 + A_1$ & 212  \\
\hline
8 & $A_3 + A_2 + 2 A_1$ & $\phi_{1400, 11}$ & $A_7$ & 218  \\
\hline
9 & $A_3 + A_2 + A_1$  & $\phi_{2240, 10}$ & $E_8(b_6)$  & 220    \\
\hline
10, 11 & $ 2A_2 + 2 A_2$  & $\phi_{1400, 8}$ & $E_8(a_6)$  & 224    \\
\hline
12, 13 & $A_2  + 3 A_1 $ & $\phi_{700, 6}$ & $E_8(a_5)$ &  228   \\
\hline
14 & $A_2 + 2 A_1 $ & $\phi_{560,5}$ & $E_8(b_4)$ &  230   \\
\hline
15, 16, 17 & $(4A_1)' $ & $\phi_{210,4}$ & $E_8(a_4)$ &  232   \\
\hline
18, 19 & $ 3 A_1$ & $\phi_{112,3}$ & $E_8(a_3)$ &  234  \\
\hline
20, 21, 22, 23 & $2A_1 $ & $\phi_{35,2}$ & $E_8(a_2)$ &  236   \\
\hline
24, 25, 26, 27, 28, 29 & $ A_1$ & $\phi_{8,1}$ & $E_8(a_1)$ &  238   \\
\hline
$\gest 30$ & $\emptyset$ & $\phi_{1,0}$ & $E_8$ & 240  \\
\hline 
\end{tabular}
\end{table}

\subsection{Compatibility and remarks}

In view of the above tables and the criterion of quasi-admissibility and raisability discussed in \S \ref{S:ex-ana}, we have the following:

\begin{thm} \label{T:comp}
Let $\wt{G}$ be any of the covering $\wt{\GL}_r^{(n)}, \wt{\SO}_{2r+1}^{(n)}, \wt{\Sp}_{2r}^{(n)}, \wt{\SO}_{2r}^{(n)}, \wt{G}_2^{(n)}, \wt{F}_4^{(n)}$ and $\wt{E}_r^{(n)}, 6\lest r \lest 8$ discussed above. Assume that $\wt{G}$ is persistent. Consider the $F$-split orbit $\mca{O}_\Theta$ of type $\mca{O}_{\rm Spr}(j_{W_\tnu}^W \varepsilon_\tnu)$ of $G$.
\begin{enumerate}
\item[(i)] The orbit $\mca{O}_\Theta$ is quasi-admissible and non-raisable.
\item[(ii)] If the orbit $\mca{O}_\Theta$ is the regular orbit of a Levi subgroup of $G$, then it supports certain generalized Whittaker models of the theta representation $\Theta(\wt{G}, \nu)$.
\end{enumerate}
\end{thm}
\begin{proof}
First, (i) follows from comparing Tables 6--11 with Theorems \ref{T:typeA}, \ref{T:tyBD}, \ref{T:tyC} and Tables 1--5. 

For (ii), recall from \cite{GGS17} that to any $\pi \in \Irrg(\wt{G})$ and any Whittaker pair $(S, u)$ of $\mca{O}_u$, there is a degenerate Whittaker model $\pi_{S, u}$ such that one has a $\wt{G}$-equivariant surjection
$$\pi_{N_\mca{O}, \psi_\mca{O}} \onto \pi_{S, u},$$
see \cite[Theorem A]{GGS17}. If $\mca{O}_u$ is the split regular orbit of a Levi subgroup $L$, then one can pick $(S, u)$ such that $\pi_{S, u}$ is equal to the semi-Whittaker models of $\pi$, i.e., the Whittaker model of the Jacquet model $J_U(\pi)$ with respect to the unipotent radical $U$ of the parabolic subgroup $P=LU$.

Applying the above to $\Theta(\wt{G}, \nu)$, by the periodicity of theta representations (see \cite[Theorem 2.3]{BFrG2}), one has $J_U(\Theta(\wt{G}, \nu))$ is a theta representation on the Levi subgroup $\wt{L}$. Thus, it suffices to show that every theta representation $\Theta(\wt{L}_\mca{O})$ on the Levi subgroup $\wt{L}_\mca{O}$ associated to such $\mca{O}_\Theta$ is generic. It follows from \cite[Proposition 6.2]{Ga6} that
$$\dim \Wh_\psi(\Theta(\wt{L}_\mca{O})) =\angb{\varepsilon_{W(L_\mca{O})}}{\sigma^\msc{X}}_{W(L_\mca{O})},$$
where $W(L_\mca{O})$ denotes the Weyl group of $L_\mca{O}$. Moreover, an explicit numerical criterion on the non-vanishing of $\dim \Wh_\psi(\Theta(\wt{L}_\mca{O}))$ is given in \cite[\S 3]{GaTs}. The result then follows from a direct check by using results in loc. cit. and Tables 6--11. We will illustrate this by considering  $\wt{\Sp}_{2r}^{(n)}$ and $\wt{E}_8^{(n)}$ as examples.

The cover $\wt{\Sp}_{2r}^{(n)}$ is persistent if and only if $n$ is odd or $n=2m$ with $m$ even. If $n$ is odd, then by Example \ref{E:Sp-orb} we see that (for $2r=an + b, 0\lest b <n$):
\begin{enumerate}
\item[--] If $a$ and $b$ are both even, then $\mca{O}_\Theta=(n^a b)$ is the principal orbit of a Levi with
$$L_\mca{O} = \Sp_b \times \prod_{j=1}^{a/2} \GL_n.$$
In this case, $\Theta(\wt{L}_\mca{O})$ is generic by \cite[Proposition 5.1]{Ga2} or \cite[Lemma 3.3]{GaTs}.
\item[--] If $a$ and $b$ are both odd, then $\mca{O}_\Theta$ principal orbit of a Levi only when $b=n-2$ and in this case
$$L_\mca{O} =\GL_{n-1} \times \prod_{j=1}^{(a-1)/2} \GL_n.$$
Again, $\Theta(\wt{L}_\mca{O})$ is generic in this case.
\end{enumerate}
Now if $n=2m$ with $m$ even, then the orbit $\mca{O}_\Theta=(m^a b)$ with $2r=am + b$, which is always the principal orbit of the Levi subgroup
$$L_\mca{O} = \Sp_b \times \prod_{i=1}^{a/2} \GL_m.$$
In this case, $\Theta(\wt{L}_\mca{O})$ is generic by results in \cite{Ga2} and \cite{GaTs} as well.

For $\wt{E}_8^{(n)}$ it follows immediately that the non-trivial orbits $\mca{O}_\Theta$ which are regular orbits of Levi subgroups are
$$4A_1, 2A_2 + 2A_1, 2A_3, A_4 + A_3, A_6+ A_1, A_7,$$
which are associated with $n=2, 3, 4, 5, 7, 8$ respectively. It follows that $\Theta(\wt{L}_\mca{O})$ is always generic in this case.

Other groups can be checked in the same way, and this completes the proof.
\end{proof}

\section{The coefficient $c_\mca{O}$ for $\Theta(\wt{\GL}_r^{(n)})$} \label{S:cO}
In this last section, we verify the equality \eqref{E:main2} in Conjecture  \ref{C:main} regarding the leading coefficient $c_{\mca{O}}$ for theta representation of covers of $\GL_r$. Note that the equality \eqref{E:main1} for $\wt{\GL}_r^{(n)}$ is due to Savin \cite{Sav2} and Cai \cite{Cai19}. 

\begin{thm} \label{T:cO}
Consider a Kazhdan--Patterson cover $\wt{\GL}_r^{(n)}$. Then for every unramified theta representation $\Theta(\pi^\dag, \nu)$, one has 
$$c_\mca{O} =\angb{j_{W_\nu}^W(\varepsilon_{W_\nu})}{ \varepsilon_W \otimes \sigma^\msc{X} }_W,$$
where $\mca{O} = (n^a b)$ is the unique orbit in $\mca{N}^{\rm max}_{\rm tr}(\Theta(\pi^\dag, \nu))$ with $r=na + b, 0\lest b <n$.
\end{thm}
\begin{proof}
For any partition $\mu \in \msc{P}(r)$ of $r$, we denote by $M_\mu \subset \GL_r$ the standard Levi subgroup. The Weyl group of $M_\mu$ is denoted by $W_\mu \subset W$. We write
$$\lambda:=(n^a b) \in \msc{P}(r).$$
Since $\mca{N}^{\rm max}_{\rm tr}(\Theta(\pi^\dag, \nu)) =\set{\lambda}$, then \cite{Pate} and \cite[Theorem 1.5]{GGS21} together with the periodicity of theta representations (see \cite{Cai19}) give the first equality in 
$$c_\mca{O}=\dim \Wh_\psi(\Theta(\wt{M}_\lambda^{(n)})) =\angb{\varepsilon_{W_\lambda}}{\sigma^\msc{X}}_{W_\lambda} = \angb{{\rm Ind}_{W_\lambda}^W (\varepsilon_{W_\lambda})}{\sigma^\msc{X}}_W  ,$$
where the theta representation $\wt{M}_\lambda^{(n)}$ is associated with $\pi^\dag$ and $\nu$, and the second equality follows from \cite[Proposition 6.2]{Ga6}. 
On the other hand, the Weyl subgroup $W_\nu$ is a parabolic Weyl subgroup associated with 
$$\lambda^\top =((a+1)^b a^{n-b}).$$
Thus,
$$\begin{aligned}
\angb{j_{W_\nu}^W(\varepsilon_{W_\nu})}{ \varepsilon_W \otimes \sigma^\msc{X} }_W & = \angb{j_{W_{\lambda^\top}}^W(\varepsilon_{W_{\lambda^\top}})}{ \varepsilon_W \otimes \sigma^\msc{X} }_W \\
& = \angb{\varepsilon_W \otimes j_{W_{\lambda^\top}}^W(\varepsilon_{W_{\lambda^\top}})}{  \sigma^\msc{X} }_W \\
& = \angb{ j_{W_\lambda}^W(\varepsilon_{W_\lambda}) }{ \sigma^\msc{X} }_W,
\end{aligned} $$
where the last equality follows from \cite[Corollary 5.4.9]{GePf}.

We have 
$${\rm Ind}_{W_\lambda}^W (\varepsilon_{W_\lambda}) = j_{W_\lambda}^W(\varepsilon_{W_\lambda}) + \sum_{\substack{\lambda \lest \mu \\ \mu \ne \lambda}}  j_{W_\mu}^W(\varepsilon_{W_\mu}),$$
see \cite[Theorem 5.4.7]{GePf}. Also, for any $\mu > \lambda$, 
$$\angb{ j_{W_\mu}^W(\varepsilon_{W_\mu}) }{ \sigma^\msc{X} }_W \lest \angb{ \Ind_{W_\mu}^W(\varepsilon_{W_\mu}) }{ \sigma^\msc{X} }_W = \angb{ \varepsilon_{W_\mu} }{ \sigma^\msc{X} }_{W_\mu}= \dim \Wh_\psi(\Theta(\wt{M}_\mu)),$$
where the last equality follows from \cite{Ga6}. Since $\mu > \lambda$, there is a component in the partition of $\mu$ strictly greater than $n$ and thus one has 
$$\dim \Wh_\psi(\Theta(\wt{M}_\mu))=0.$$
All the above together give that $c_\mca{O} = \angb{j_{W_\nu}^W(\varepsilon_{W_\nu})}{ \varepsilon_W \otimes \sigma^\msc{X} }_W$.  This completes the proof.
\end{proof}

Note that the restriction to Kazhdan--Patterson cover of $\GL_r$ is only for convenience. Results in Theorem \ref{T:cO} and other parts of paper hold for general Brylinski--Deligne covers of $\GL_r$, where the only essential alternation in the statement is to replace $n$ by $n_\alpha$.




\begin{bibdiv}
\begin{biblist}[\resetbiblist{9999999}]*{labels={alphabetic}}


\bib{Ach03}{article}{
  author={Achar, Pramod N.},
  title={An order-reversing duality map for conjugacy classes in Lusztig's canonical quotient},
  journal={Transform. Groups},
  volume={8},
  date={2003},
  number={2},
  pages={107--145},
  issn={1083-4362},
  review={\MR {1976456}},
  doi={10.1007/s00031-003-0422-x},
}

\bib{ABPTV}{article}{
  author={Adams, Jeffrey},
  author={Barbasch, Dan},
  author={Paul, Annegret},
  author={Trapa, Peter E.},
  author={ Vogan, David A., Jr.},
  title={Unitary Shimura correspondences for split real groups},
  journal={J. Amer. Math. Soc.},
  volume={20},
  date={2007},
  number={3},
  pages={701--751},
  issn={0017-095X},
  review={\MR {3151110}},
  doi={10.3336/gm.48.2.07},
}

\bib{BX19}{article}{
  author={Bai, Zhanqiang},
  author={Xie, Xun},
  title={Gelfand-Kirillov dimensions of highest weight Harish-Chandra modules for $SU(p,q)$},
  journal={Int. Math. Res. Not. IMRN},
  date={2019},
  number={14},
  pages={4392--4418},
  issn={1073-7928},
  review={\MR {3984073}},
  doi={10.1093/imrn/rnx247},
}

\bib{BXX23}{article}{
  author={Bai, Zhanqiang},
  author={Xiao, Wei},
  author={Xie, Xun},
  title={Gelfand--Kirillov Dimensions and Associated Varieties of Highest Weight Modules},
  journal={Int. Math. Res. Not. IMRN},
  date={2023},
  number={10},
  pages={8101--8142},
  issn={1073-7928},
  review={\MR {4589071}},
  doi={10.1093/imrn/rnac081},
}

\bib{BMW}{article}{
  author={Bai, Zhanqiang},
  author={Ma, Jiajun},
  author={Wang, Yutong},
  title={A combinatorial characterization of the annihilator varieties of highest weight modules for classical Lie algebras},
  status={preprint, available at https://arxiv.org/abs/2304.03475v1},
}

\bib{BJ1}{article}{
  author={Ban, Dubravka},
  author={Jantzen, Chris},
  title={The Langlands quotient theorem for finite central extensions of $p$-adic groups},
  journal={Glas. Mat. Ser. III},
  volume={48(68)},
  date={2013},
  number={2},
  pages={313--334},
  issn={0017-095X},
  review={\MR {3151110}},
  doi={10.3336/gm.48.2.07},
}

\bib{BV3}{article}{
  author={Barbasch, Dan},
  author={Vogan, David A., Jr.},
  title={The local structure of characters},
  journal={J. Functional Analysis},
  volume={37},
  date={1980},
  number={1},
  pages={27--55},
  issn={0022-1236},
  review={\MR {576644}},
  doi={10.1016/0022-1236(80)90026-9},
}

\bib{BV4}{article}{
  author={Barbasch, Dan},
  author={Vogan, David},
  title={Primitive ideals and orbital integrals in complex classical groups},
  journal={Math. Ann.},
  volume={259},
  date={1982},
  number={2},
  pages={153--199},
  issn={0025-5831},
  review={\MR {656661}},
}

\bib{BV5}{article}{
  author={Barbasch, Dan},
  author={Vogan, David},
  title={Primitive ideals and orbital integrals in complex exceptional groups},
  journal={J. Algebra},
  volume={80},
  date={1983},
  number={2},
  pages={350--382},
  issn={0021-8693},
  review={\MR {691809}},
}

\bib{BoTi73}{article}{
  author={Borel, Armand},
  author={Tits, Jacques},
  title={Homomorphismes ``abstraits'' de groupes alg\'{e}briques simples},
  language={French},
  journal={Ann. of Math. (2)},
  volume={97},
  date={1973},
  pages={499--571},
  issn={0003-486X},
  review={\MR {316587}},
  doi={10.2307/1970833},
}

\bib{BouL2}{book}{
  author={Bourbaki, Nicolas},
  title={Lie groups and Lie algebras. Chapters 4--6},
  series={Elements of Mathematics (Berlin)},
  note={Translated from the 1968 French original by Andrew Pressley},
  publisher={Springer-Verlag, Berlin},
  date={2002},
  pages={xii+300},
  isbn={3-540-42650-7},
  review={\MR {1890629}},
  doi={10.1007/978-3-540-89394-3},
}

\bib{BD}{article}{
  author={Brylinski, Jean-Luc},
  author={Deligne, Pierre},
  title={Central extensions of reductive groups by $\bold K_2$},
  journal={Publ. Math. Inst. Hautes \'Etudes Sci.},
  number={94},
  date={2001},
  pages={5--85},
  issn={0073-8301},
  review={\MR {1896177}},
  doi={10.1007/s10240-001-8192-2},
}

\bib{BFrG2}{article}{
  author={Bump, Daniel},
  author={Friedberg, Solomon},
  author={Ginzburg, David},
  title={Small representations for odd orthogonal groups},
  journal={Int. Math. Res. Not.},
  date={2003},
  number={25},
  pages={1363--1393},
  issn={1073-7928},
  review={\MR {1968295}},
  doi={10.1155/S1073792803210217},
}

\bib{BFrG}{article}{
  author={Bump, Daniel},
  author={Friedberg, Solomon},
  author={Ginzburg, David},
  title={Lifting automorphic representations on the double covers of orthogonal groups},
  journal={Duke Math. J.},
  volume={131},
  date={2006},
  number={2},
  pages={363--396},
  issn={0012-7094},
  review={\MR {2219245}},
}

\bib{BFG}{article}{
  author={Bump, Daniel},
  author={Furusawa, Masaaki},
  author={Ginzburg, David},
  title={Non-unique models in the Rankin-Selberg method},
  journal={J. Reine Angew. Math.},
  volume={468},
  date={1995},
  pages={77--111},
  issn={0075-4102},
  review={\MR {1361787}},
}

\bib{Cai19}{article}{
  author={Cai, Yuanqing},
  title={Fourier coefficients for theta representations on covers of general linear groups},
  journal={Trans. Amer. Math. Soc.},
  volume={371},
  date={2019},
  number={11},
  pages={7585--7626},
  issn={0002-9947},
  review={\MR {3955529}},
  doi={10.1090/tran/7429},
}

\bib{Car}{book}{
  author={Carter, Roger W.},
  title={Finite groups of Lie type},
  series={Wiley Classics Library},
  note={Conjugacy classes and complex characters; Reprint of the 1985 original; A Wiley-Interscience Publication},
  publisher={John Wiley \& Sons, Ltd., Chichester},
  date={1993},
  pages={xii+544},
  isbn={0-471-94109-3},
  review={\MR {1266626}},
}

\bib{CMO}{article}{
  author={Ciubotaru, Dan},
  author={Mason-Brown, Lucas},
  author={Okada, Emile T.},
  title={The wavefront sets of Iwahori-spherical representations of reductive $p$-adic groups},
  status={preprint, available at https://arxiv.org/abs/2112.14354v4},
}

\bib{CM}{book}{
  author={Collingwood, David H.},
  author={McGovern, William M.},
  title={Nilpotent orbits in semisimple Lie algebras},
  series={Van Nostrand Reinhold Mathematics Series},
  publisher={Van Nostrand Reinhold Co., New York},
  date={1993},
  pages={xiv+186},
  isbn={0-534-18834-6},
  review={\MR {1251060}},
}

\bib{Duf80}{article}{
  author={Duflo, Michel},
  title={Construction de repr\'{e}sentations unitaires d'un groupe de Lie},
  language={French, with English summary},
  conference={ title={Harmonic analysis and group representations}, },
  book={ publisher={Liguori, Naples}, },
  date={1982},
  pages={129--221},
  review={\MR {777341}},
}

\bib{FG15}{article}{
  author={Friedberg, Solomon},
  author={Ginzburg, David},
  title={Metaplectic theta functions and global integrals},
  journal={J. Number Theory},
  volume={146},
  date={2015},
  pages={134--149},
  issn={0022-314X},
  review={\MR {3267113}},
  doi={10.1016/j.jnt.2014.04.001},
}

\bib{FG17-2}{article}{
  author={Friedberg, Solomon},
  author={Ginzburg, David},
  title={Theta functions on covers of symplectic groups},
  journal={Bull. Iranian Math. Soc.},
  volume={43},
  date={2017},
  number={4},
  pages={89--116},
  issn={1017-060X},
  review={\MR {3711824}},
}

\bib{FG18}{article}{
  author={Friedberg, Solomon},
  author={Ginzburg, David},
  title={Descent and theta functions for metaplectic groups},
  journal={J. Eur. Math. Soc. (JEMS)},
  volume={20},
  date={2018},
  number={8},
  pages={1913--1957},
  issn={1435-9855},
  review={\MR {3854895}},
  doi={10.4171/JEMS/803},
}

\bib{FG20}{article}{
  author={Friedberg, Solomon},
  author={Ginzburg, David},
  title={Classical theta lifts for higher metaplectic covering groups},
  journal={Geom. Funct. Anal.},
  volume={30},
  date={2020},
  number={6},
  pages={1531--1582},
  issn={1016-443X},
  review={\MR {4182832}},
  doi={10.1007/s00039-020-00548-y},
}

\bib{GG}{article}{
  author={Gan, Wee Teck},
  author={Gao, Fan},
  title={The Langlands-Weissman program for Brylinski-Deligne extensions},
  language={English, with English and French summaries},
  note={L-groups and the Langlands program for covering groups},
  journal={Ast\'erisque},
  date={2018},
  number={398},
  pages={187--275},
  issn={0303-1179},
  isbn={978-2-85629-845-9},
  review={\MR {3802419}},
}

\bib{Ga2}{article}{
  author={Gao, Fan},
  title={Distinguished theta representations for certain covering groups},
  journal={Pacific J. Math.},
  volume={290},
  date={2017},
  number={2},
  pages={333--379},
  doi={10.2140/pjm.2017.290.333},
}

\bib{Ga6}{article}{
  author={Gao, Fan},
  title={Kazhdan--Lusztig representations and Whittaker space of some genuine representations},
  journal={Math. Ann.},
  volume={376},
  date={2020},
  number={1},
  pages={289--358},
  doi={10.1007/s00208-019-01925-1},
}

\bib{GSS1}{article}{
  author={Gao, Fan},
  author={Shahidi, Freydoon},
  author={Szpruch, Dani},
  title={On the local coefficients matrix for coverings of $\rm SL_2$},
  conference={ title={Geometry, algebra, number theory, and their information technology applications}, },
  book={ series={Springer Proc. Math. Stat.}, volume={251}, publisher={Springer, Cham}, },
  date={2018},
  pages={207--244},
  review={\MR {3880389}},
}

\bib{GaTs}{article}{
  author={Gao, Fan},
  author={Tsai, Wan-Yu},
  title={On the wavefront sets associated with theta representations},
  journal={Math. Z.},
  volume={301},
  date={2022},
  number={1},
  pages={1--40},
  issn={0025-5874},
  review={\MR {4405642}},
  doi={10.1007/s00209-021-02894-5},
}

\bib{GePf}{book}{
  author={Geck, Meinolf},
  author={Pfeiffer, G\"{o}tz},
  title={Characters of finite Coxeter groups and Iwahori-Hecke algebras},
  series={London Mathematical Society Monographs. New Series},
  volume={21},
  publisher={The Clarendon Press, Oxford University Press, New York},
  date={2000},
  pages={xvi+446},
  isbn={0-19-850250-8},
  review={\MR {1778802}},
}

\bib{Gin0}{article}{
  author={Ginzburg, David},
  title={Certain conjectures relating unipotent orbits to automorphic representations},
  journal={Israel J. Math.},
  volume={151},
  date={2006},
  pages={323--355},
  issn={0021-2172},
  review={\MR {2214128}},
  doi={10.1007/BF02777366},
}

\bib{GRS11}{book}{
  author={Ginzburg, David},
  author={Rallis, Stephen},
  author={Soudry, David},
  title={The descent map from automorphic representations of ${\rm GL}(n)$ to classical groups},
  publisher={World Scientific Publishing Co. Pte. Ltd., Hackensack, NJ},
  date={2011},
  pages={x+339},
  isbn={978-981-4304-98-6},
  isbn={981-4304-98-0},
  review={\MR {2848523}},
  doi={10.1142/9789814304993},
}

\bib{GGS17}{article}{
  author={Gomez, Raul},
  author={Gourevitch, Dmitry},
  author={Sahi, Siddhartha},
  title={Generalized and degenerate Whittaker models},
  journal={Compos. Math.},
  volume={153},
  date={2017},
  number={2},
  pages={223--256},
  issn={0010-437X},
  review={\MR {3705224}},
  doi={10.1112/S0010437X16007788},
}

\bib{GGS21}{article}{
  author={Gomez, Raul},
  author={Gourevitch, Dmitry},
  author={Sahi, Siddhartha},
  title={Whittaker supports for representations of reductive groups},
  language={English, with English and French summaries},
  journal={Ann. Inst. Fourier (Grenoble)},
  volume={71},
  date={2021},
  number={1},
  pages={239--286},
  issn={0373-0956},
  review={\MR {4275869}},
}

\bib{GoSa15}{article}{
  author={Gourevitch, Dmitry},
  author={Sahi, Siddhartha},
  title={Degenerate Whittaker functionals for real reductive groups},
  journal={Amer. J. Math.},
  volume={137},
  date={2015},
  number={2},
  pages={439--472},
  issn={0002-9327},
  review={\MR {3337800}},
  doi={10.1353/ajm.2015.0008},
}

\bib{Gur23}{article}{
  author={Gurevich, Maxim},
  title={A triangular system for local character expansions of Iwahori-spherical representations of general linear groups},
  language={English, with English and French summaries},
  journal={C. R. Math. Acad. Sci. Paris},
  volume={361},
  date={2023},
  pages={21--30},
  issn={1631-073X},
  review={\MR {4538538}},
  doi={10.5802/crmath.384},
}

\bib{HC99}{book}{
  author={Harish-Chandra},
  title={Admissible invariant distributions on reductive $p$-adic groups},
  series={University Lecture Series},
  volume={16},
  note={With a preface and notes by Stephen DeBacker and Paul J. Sally, Jr.},
  publisher={American Mathematical Society, Providence, RI},
  date={1999},
  pages={xiv+97},
  isbn={0-8218-2025-7},
  review={\MR {1702257}},
  doi={10.1090/ulect/016},
}

\bib{How1}{article}{
  author={Howe, Roger},
  title={The Fourier transform and germs of characters (case of ${\rm Gl}_{n}$ over a $p$-adic field)},
  journal={Math. Ann.},
  volume={208},
  date={1974},
  pages={305--322},
  issn={0025-5831},
  review={\MR {342645}},
  doi={10.1007/BF01432155},
}

\bib{JaNo05}{article}{
  author={Jackson, Steven Glenn},
  author={No\"{e}l, Alfred G.},
  title={Prehomogeneous spaces associated with complex nilpotent orbits},
  journal={J. Algebra},
  volume={289},
  date={2005},
  number={2},
  pages={515--557},
  issn={0021-8693},
  review={\MR {2142384}},
  doi={10.1016/j.jalgebra.2005.02.017},
}

\bib{Jia14}{article}{
  author={Jiang, Dihua},
  title={Automorphic integral transforms for classical groups I: Endoscopy correspondences},
  conference={ title={Automorphic forms and related geometry: assessing the legacy of I. I. Piatetski-Shapiro}, },
  book={ series={Contemp. Math.}, volume={614}, publisher={Amer. Math. Soc., Providence, RI}, },
  date={2014},
  pages={179--242},
  review={\MR {3220929}},
  doi={10.1090/conm/614/12253},
}

\bib{JiLi16}{article}{
  author={Jiang, Dihua},
  author={Liu, Baiying},
  title={Fourier coefficients for automorphic forms on quasisplit classical groups},
  conference={ title={Advances in the theory of automorphic forms and their $L$-functions}, },
  book={ series={Contemp. Math.}, volume={664}, publisher={Amer. Math. Soc., Providence, RI}, },
  date={2016},
  pages={187--208},
  review={\MR {3502983}},
  doi={10.1090/conm/664/13062},
}

\bib{JL22}{article}{
  author={Jiang, Dihua},
  author={Liu, Baiying},
  title={On wavefront sets of global Arthur packets of classical groups: upper bound},
  journal={J. European Math. Society.},
  volume={},
  date={2022},
  pages={},
  review={},
  doi={},
}

\bib{JLS}{article}{
  author={Jiang, Dihua},
  author={Liu, Baiying},
  author={Savin, Gordan},
  title={Raising nilpotent orbits in wave-front sets},
  journal={Represent. Theory},
  volume={20},
  date={2016},
  pages={419--450},
  review={\MR {3564676}},
  doi={10.1090/ert/490},
}

\bib{JLZ}{article}{
  author={Jiang, Dihua},
  author={Liu, Dongwen},
  author={Zhang, Lei},
  title={Arithmetic wavefront sets and generic $L$-packets},
  status={preprint, available at https://arxiv.org/abs/2207.04700},
}

\bib{JiZh17}{article}{
  author={Jiang, Dihua},
  author={Zhang, Lei},
  title={Automorphic integral transforms for classical groups II: Twisted descents},
  conference={ title={Representation theory, number theory, and invariant theory}, },
  book={ series={Progr. Math.}, volume={323}, publisher={Birkh\"{a}user/Springer, Cham}, },
  date={2017},
  pages={303--335},
  review={\MR {3753916}},
}

\bib{JiZh20}{article}{
  author={Jiang, Dihua},
  author={Zhang, Lei},
  title={Arthur parameters and cuspidal automorphic modules of classical groups},
  journal={Ann. of Math. (2)},
  volume={191},
  date={2020},
  number={3},
  pages={739--827},
  issn={0003-486X},
  review={\MR {4088351}},
  doi={10.4007/annals.2020.191.3.2},
}

\bib{Kap17-1}{article}{
  author={Kaplan, Eyal},
  title={The double cover of odd general spin groups, small representations, and applications},
  journal={J. Inst. Math. Jussieu},
  volume={16},
  date={2017},
  number={3},
  pages={609--671},
  issn={1474-7480},
  review={\MR {3646283}},
  doi={10.1017/S1474748015000250},
}

\bib{Li3}{article}{
  author={Li, Wen-Wei},
  title={La formule des traces pour les rev\^etements de groupes r\'eductifs connexes. II. Analyse harmonique locale},
  language={French, with English and French summaries},
  journal={Ann. Sci. \'Ec. Norm. Sup\'er. (4)},
  volume={45},
  date={2012},
  number={5},
  pages={787--859},
  issn={0012-9593},
  review={\MR {3053009}},
  doi={10.24033/asens.2178},
}

\bib{LS22a}{article}{
  author={Liu, Baiying},
  author={Shahidi, Freydoon},
  title={Jiang's conjecture on local Arthur packets},
  journal={Kudla Proceedings},
  volume={},
  date={2022},
  pages={},
  review={},
  doi={},
}

\bib{LS22b}{article}{
  author={Liu, Baiying},
  author={Shahidi, Freydoon},
  title={Jiang's conjecture on the wave front sets of local Arthur packets},
  journal={Preprint},
  volume={},
  date={2022},
  pages={},
  review={},
  doi={},
}

\bib{LX21}{article}{
  author={Liu, Baiying},
  author={Xu, Bin},
  title={On top Fourier coefficients of certain automorphic representations of $\mathrm {GL}_n$},
  journal={Manuscripta Math.},
  fjournal={Manuscripta Mathematica},
  volume={164},
  year={2021},
  number={1-2},
  pages={1--22},
  issn={0025-2611},
  mrclass={11F70 (11F30 22E50 22E55)},
  mrnumber={4203681},
  mrreviewer={\Dbar \cftil {o} Ng\d {o}c Di\cfudot {e}p},
  doi={10.1007/s00229-019-01176-z},
  url={https://doi-org.ezproxy.lib.purdue.edu/10.1007/s00229-019-01176-z},
}

\bib{McSo03}{article}{
  author={McNinch, George J.},
  author={Sommers, Eric},
  title={Component groups of unipotent centralizers in good characteristic},
  note={Special issue celebrating the 80th birthday of Robert Steinberg},
  journal={J. Algebra},
  volume={260},
  date={2003},
  number={1},
  pages={323--337},
  issn={0021-8693},
  review={\MR {1976698}},
  doi={10.1016/S0021-8693(02)00661-0},
}

\bib{Moe96}{article}{
  author={M\oe glin, C.},
  title={Front d'onde des repr\'{e}sentations des groupes classiques $p$-adiques},
  language={French, with French summary},
  journal={Amer. J. Math.},
  volume={118},
  date={1996},
  number={6},
  pages={1313--1346},
  issn={0002-9327},
  review={\MR {1420926}},
}

\bib{MW1}{article}{
  author={M\oe glin, C.},
  author={Waldspurger, J.-L.},
  title={Mod\`eles de Whittaker d\'eg\'en\'er\'es pour des groupes $p$-adiques},
  language={French},
  journal={Math. Z.},
  volume={196},
  date={1987},
  number={3},
  pages={427--452},
  issn={0025-5874},
  review={\MR {913667}},
}

\bib{Nev99}{article}{
  author={Nevins, Monica},
  title={Admissible nilpotent coadjoint orbits of $p$-adic reductive Lie groups},
  journal={Represent. Theory},
  volume={3},
  date={1999},
  pages={105--126},
  review={\MR {1698202}},
  doi={10.1090/S1088-4165-99-00072-2},
}

\bib{Nev02}{article}{
  author={Nevins, Monica},
  title={Admissible nilpotent orbits of real and $p$-adic split exceptional groups},
  journal={Represent. Theory},
  volume={6},
  date={2002},
  pages={160--189},
  review={\MR {1915090}},
  doi={10.1090/S1088-4165-02-00134-6},
}

\bib{Oka}{article}{
  author={Okada, Emile T.},
  title={The wavefront set of spherical Arthur representations},
  status={preprint, available at https://arxiv.org/abs/2107.10591v1},
}

\bib{Pate}{article}{
  author={Prakash Patel, Shiv},
  title={A theorem of M\oe glin and Waldspurger for covering groups},
  journal={Pacific J. Math.},
  volume={273},
  date={2015},
  number={1},
  pages={225--239},
  issn={0030-8730},
  review={\MR {3290452}},
}

\bib{Sav94}{article}{
  author={Savin, Gordan},
  title={Dual pair $G_{\scr J}\times {\rm PGL}_2$ [where] $G_{\scr J}$ is the automorphism group of the Jordan algebra ${\scr J}$},
  journal={Invent. Math.},
  volume={118},
  date={1994},
  number={1},
  pages={141--160},
  issn={0020-9910},
  review={\MR {1288471}},
  doi={10.1007/BF01231530},
}

\bib{Sav2}{article}{
  author={Savin, Gordan},
  title={A nice central extension of ${\rm GL}_r$},
  status={preprint},
}

\bib{Ser-gc}{book}{
  author={Serre, Jean-Pierre},
  title={Galois cohomology},
  series={Springer Monographs in Mathematics},
  edition={Corrected reprint of the 1997 English edition},
  note={Translated from the French by Patrick Ion and revised by the author},
  publisher={Springer-Verlag, Berlin},
  date={2002},
  pages={x+210},
  isbn={3-540-42192-0},
  review={\MR {1867431}},
}

\bib{Som98}{article}{
  author={Sommers, Eric},
  title={A generalization of the Bala-Carter theorem for nilpotent orbits},
  journal={Internat. Math. Res. Notices},
  date={1998},
  number={11},
  pages={539--562},
  issn={1073-7928},
  review={\MR {1631769}},
  doi={10.1155/S107379289800035X},
}

\bib{Som01}{article}{
  author={Sommers, Eric},
  title={Lusztig's canonical quotient and generalized duality},
  journal={J. Algebra},
  volume={243},
  date={2001},
  number={2},
  pages={790--812},
  issn={0021-8693},
  review={\MR {1850659}},
  doi={10.1006/jabr.2001.8868},
}

\bib{Tor}{article}{
  author={Torasso, Pierre},
  title={M\'{e}thode des orbites de Kirillov-Duflo et repr\'{e}sentations minimales des groupes simples sur un corps local de caract\'{e}ristique nulle},
  language={French},
  journal={Duke Math. J.},
  volume={90},
  date={1997},
  number={2},
  pages={261--377},
  issn={0012-7094},
  review={\MR {1484858}},
  doi={10.1215/S0012-7094-97-09009-8},
}

\bib{Tsa22}{article}{
  author={Tsai, Cheng-Chiang},
  title={Geometric wave-front set may not be a singleton},
  journal={Preprint},
  volume={},
  date={2022},
  pages={},
  review={},
  doi={},
}

\bib{Tsa19}{article}{
  author={Tsai, Wan-Yu},
  title={Some genuine small representations of a nonlinear double cover},
  journal={Trans. Amer. Math. Soc.},
  volume={371},
  date={2019},
  number={8},
  pages={5309--5340},
  issn={0002-9947},
  review={\MR {3937294}},
  doi={10.1090/tran/7351},
}

\bib{Var0}{article}{
  author={Varma, Sandeep},
  title={On a result of Moeglin and Waldspurger in residual characteristic 2},
  journal={Math. Z.},
  volume={277},
  date={2014},
  number={3-4},
  pages={1027--1048},
  issn={0025-5874},
  review={\MR {3229979}},
  doi={10.1007/s00209-014-1292-8},
}

\bib{We6}{article}{
  author={Weissman, Martin H.},
  title={L-groups and parameters for covering groups},
  language={English, with English and French summaries},
  note={L-groups and the Langlands program for covering groups},
  journal={Ast\'erisque},
  date={2018},
  number={398},
  pages={33--186},
  issn={0303-1179},
  isbn={978-2-85629-845-9},
  review={\MR {3802418}},
}






\end{biblist}
\end{bibdiv}
\end{document}